\documentclass[reqno,12pt]{amsart}
\usepackage{geometry} \usepackage{graphicx}
\usepackage{fullpage}
\usepackage{setspace}
\usepackage{datetime}
\usepackage{tikz}
\usepackage{mathtools}
\usepackage{enumitem}
\usetikzlibrary{er}
\usetikzlibrary{arrows.meta}
\usepackage{amsmath,amsfonts,amssymb}
\usepackage{cancel}
\usepackage{verbatim,enumitem,graphics}
\usepackage{microtype}
\usepackage{xcolor}
\usepackage[backref=page]{hyperref}
\hypersetup{%
    colorlinks,
    linkcolor={red!50!black},
    citecolor={green!50!black},
    urlcolor={blue!50!black}
}
\usepackage{dsfont}
\usepackage[abbrev,msc-links,backrefs]{amsrefs} 
\usepackage{doi}
\usepackage{lineno}
\usepackage[normalem]{ulem}
\usepackage{float}

\makeatletter
\def\moverlay{\mathpalette\mov@rlay}
\def\mov@rlay#1#2{\leavevmode\vtop{%
   \baselineskip\z@skip \lineskiplimit-\maxdimen
   \ialign{\hfil$\m@th#1##$\hfil\cr#2\crcr}}}
\newcommand{\charfusion}[3][\mathord]{
    #1{\ifx#1\mathop\vphantom{#2}\fi
        \mathpalette\mov@rlay{#2\cr#3}
      }
    \ifx#1\mathop\expandafter\displaylimits\fi}
\makeatother

\newtheorem{theorem}{Theorem}[section]
\newtheorem{lemma}[theorem]{Lemma}
\newtheorem{proposition}[theorem]{Proposition}
\newtheorem{fact}[theorem]{Fact}

\newtheorem{problem}[theorem]{Problem}

\theoremstyle{definition}
\newtheorem{definition}[theorem]{Definition}

\theoremstyle{remark}
\newtheorem{claim}{Claim}
\newtheorem{remark}[theorem]{Remark}
\newtheorem*{claim*}{Claim}
\newtheorem*{remark*}{Remark}

\newcommand{\ZZ}{\mathbb Z}
\newcommand{\EE}{\mathbb E}

\newcommand{\PP}{\mathbb P}
\newcommand{\cF}{\mathcal{F}}

\newcommand{\cG}{\mathcal{G}}
\newcommand{\cL}{\mathcal{L}}
\newcommand{\cE}{\mathcal{E}}

\newcommand{\Z}{\mathbb{Z}}

\newcommand{\dcup}{\charfusion[\mathbin]{\cup}{\cdot}}

\def\d{\delta }

\def\l{\ell }

\newcommand{\set}[1]{\left\{ #1 \right\}}

\newcommand\ordarrow{\mathrel{\overset{\makebox[1pt]{\mbox{\normalfont\tiny\sffamily ord}}}{\longrightarrow}}}

\newcommand{\narrow}{\charfusion[\mathbin]{\scriptsize{/}}{\longrightarrow}}

\let\tilde = \widetilde

\DeclareMathOperator{\rev}{rev}
\DeclareMathOperator{\Hyp}{HG}

\allowdisplaybreaks

\begin{document}

\onehalfspacing

\title{On Ramsey numbers of Steiner systems}

\author{Ayush Basu}
\author{Daniel Dobak}
\author{Vojt\v{e}ch R\"odl}
\address{Department of Mathematics, Emory University, 
    Atlanta, GA, USA}
\email{\{ayush.basu|daniel.dobak|vrodl\}@emory.edu}
    \author{Marcelo Sales}
    \address{Department of Mathematics, University of California, Irvine, CA, USA}
\email{mtsales@uci.edu}

\thanks{The first and second authors were partially supported by NSF grant DMS 2300347. The fourth author was supported by US Air Force grant FA9550-23-1-0298.}





\begin{abstract}
    A $k$-uniform hypergraph $H$ is called a \emph{partial $(k,\ell)$-system} if every set of $\ell$ vertices of $V(H)$ is contained in at most one edge of $H$. We prove the existence of a partial $(k,k-1)$-system $H$ whose Ramsey number with $r \geq 4$ colors grows as a tower of height $k-1$.
\end{abstract}

\maketitle

\section{Introduction}\label{sec:intro}

For a natural number $N$, we set $[N]=\{1,\ldots,N\}$. Given a set $X\subseteq [N]$, we denote by $X^{(k)}$ the set of $k$-tuples of $X$. For two sets $X$ and $Y$, we say that $X<Y$ if $\max(X) < \min(Y)$.  A $k$-uniform hypergraph (or $k$-graph) is a pair $H=(V,E)$, where $V$ is the set of vertices and $E\subseteq V^{(k)}$ is the set of edges. For convenience, we will sometimes view $H$ simply as its set of edges.

Given a $k$-graph $H$, we say
\begin{align}\label{eq:arrow}
    [N]\longrightarrow (H)^k_r
\end{align}
holds if every $r$-coloring of the $k$-tuples $[N]^{(k)}$ yields a monochromatic copy of $H$. Let the \emph{Ramsey number of $H$ in $r$ colors} denoted by $R(H;r)$ be the minimum integer $N$ such that the relation (\ref{eq:arrow}) holds. The negative relation $[N]\narrow (H)^k_r$ indicates that (\ref{eq:arrow}) does not hold.

The existence of $R(H;r)$ for arbitrary $H$ follows from Ramsey's theorem \cite{R29}. For $2$-graphs, Erd\H{o}s and Szekeres \cite{ES35} showed that 
\begin{align*}
    R(K_n^{(2)};r)<r^{rn},
\end{align*}
where $K_n^{(2)}$ is the complete $2$-graph on $n$ vertices and $r\geq 2$. This result was later generalized to complete hypergraphs by Erd\H{o}s and Rado \cite{ER52}, who proved that
\begin{align*}
    R(K_n^{(k)};r)\leq t_{k-1}(crn\log r),
\end{align*}
for $r\geq 2$ and some positive constant $c$, where $t_i$ is the tower function given by $t_0(x)=x$ and $t_{i+1}(x)=2^{t_i(x)}$ for $i\geq 0$. We remark that both bounds above are known to be of the correct order of magnitude for sufficiently many colors (see \cites{EHR65, Ab72, CFS13, CF21}).

Although Ramsey numbers for complete hypergraphs grow like a tower, the behavior of Ramsey numbers for sparser graphs can vary greatly. On one hand, it was shown by a series of authors \cites{NORS08, CFKO08, CFKO09, CRST83} that the Ramsey number of a bounded degree hypergraph is always linear. In contrast with that, it was shown in \cite{CFS09} that there are sparser $3$-uniform hypergraphs on $n$ vertices and $O(n^2)$ edges whose Ramsey number in 4 colors is at least $t_2(cn)$, i.e., doubly exponential in $n$. However, the hypergraph constructed in \cite{CFS09} contains many pairs of high codegree. Thus, a natural question arises whether there are linear $3$-uniform hypergraphs $H$ with doubly exponential Ramsey number in 4 colors. In this paper, we construct a linear 3-uniform hypergraph $H$ on $n$ vertices with $R(H;4) \ge t_2(n^c)$. 

More generally, given integers $2\leq \ell < k$, a \emph{partial Steiner $(k,\ell)$-system} is a $k$-graph $H$ with the property that every $\ell$-element subset of the vertex set of $H$ is contained in at most one edge. Throughout our exposition, we will often refer to such hypergraphs simply as $(k,\ell)$-systems. Note that a $(k,2)$-system is just a linear $k$-graph. Also note that a $(k,\ell)$-system has at most $O(n^\ell)$ edges. Our main result provides a tower-type lower bound for the Ramsey numbers of $(k,k-1)$-systems. 

\begin{theorem}\label{thm:main}
For every $k\geq 3$, there exists a positive constant $c_k$ and an integer $h_0:=h_0(k)$ such that the following holds. For every integer $h\geq h_0$, there exists a $(k,k-1)$-system $H$ on $h$ vertices such that
\begin{align*}
    R(H;4)\geq t_{k-1}(h^{c_k}).
\end{align*}
\end{theorem}
We remark that the lower bound on Theorem \ref{thm:main} matches the tower-type order of magnitude of the upper bound for the Ramsey number of complete $k$-uniform hypergraphs. We were also informed that the case $k=3$ was obtained independently by Conlon, Fox, and Sudakov in an unpublished work \cite{CFSunp}. 

The proof of Theorem \ref{thm:main} is divided into two parts. For a conveniently chosen $I\in [n]^{(k-1)}$, we construct a family of ordered $k$-graphs on $\Theta(n^c)$ vertices denoted by $\cF^{*}_I(n,k)$ such that: 
\begin{itemize}
    \item for $N \sim t_{k-1}(n^c)$, there exists a 4-coloring of $[N]^{(k)}$ for which no ordered copy $(F,<)\in \cF_I^{*}(n,k)$ is monochromatic.
    \item On the other hand, there exists an (unordered) $(k,k-1)$-system $H$ on $h=n^C$ vertices, such that any ordering of the vertices of $H$ contains an ordered copy of some $(F,<)\in \cF_I^{*}(n,k)$. 
\end{itemize}
This implies that $[N]\narrow (H)_r^k$  and thus $R(H;4)\ge t_{k-1}(h^{c/C})$.

In Section \ref{sec:steppingup}, we describe the family $\cF^{*}_I(n,k)$ and construct the 4-coloring using a variant of the stepping-up lemma due to Erd\H{o}s, Hajnal, and Rado \cite{EHR65}  (see Theorem \ref{thm: steppingup}). In Section \ref{sec:orderings}, we construct the $(k,k-1)$-system $H$ (see Theorem \ref{thm:winkler}). This construction is inspired by a result of \cite{RW89}. In Section \ref{sec:main}, we combine Theorems \ref{thm: steppingup} and \ref{thm:winkler} to prove Theorem \ref{thm:main}. 

\section{A Stepping up lemma}\label{sec:steppingup}
To state Theorem \ref{thm: steppingup}, we will first need to define the family of ordered $k$-graphs $\cF_I^{*}(n,k)$. To do this, we will first define the families of ordered $k$-graphs $\cF_I^{(k)}(n)$ and $\rev \cF_I^{(k)}(n)$. We remark that if our aim would be to prove the case $k=3$ only, it would be sufficient to consider a single ordered $3$-graph (in particular, the one drawn in Figure \ref{fig: F_Ifor3}). However, since the proof proceeds by induction, we need to consider for every $k\ge 3$ a larger family of ordered $k$-graphs as defined below. 
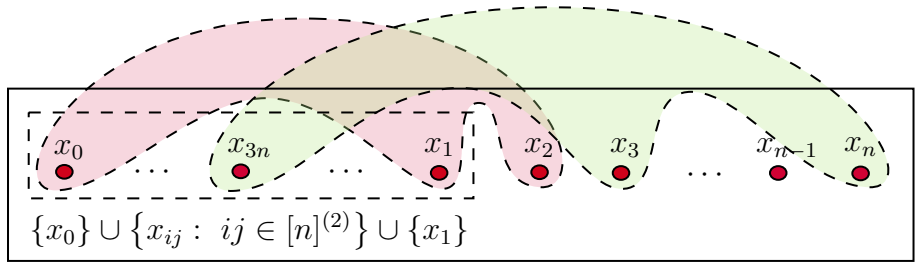
\begin{figure}[b]
    \centering

\tikzset{every picture/.style={line width=0.75pt}} 

\begin{tikzpicture}[x=0.75pt,y=0.75pt,yscale=-0.8,xscale=1]

\draw  [fill={rgb, 255:red, 208; green, 2; blue, 52 }  ,fill opacity=0.16 ][dash pattern={on 4.5pt off 4.5pt}] (265,67) .. controls (385,71) and (405.33,174.33) .. (378,173) .. controls (350.67,171.67) and (362,120) .. (347,120) .. controls (332,120) and (345.8,173.71) .. (329,174) .. controls (312.2,174.29) and (276,116) .. (240,117) .. controls (204,118) and (166,174) .. (136,175) .. controls (106,176) and (145,63) .. (265,67) -- cycle ;
\draw  [fill={rgb, 255:red, 184; green, 233; blue, 134 }  ,fill opacity=0.29 ][dash pattern={on 4.5pt off 4.5pt}] (375,60) .. controls (522.01,54.66) and (575,174) .. (542,173) .. controls (509,172) and (479,113) .. (453,113) .. controls (427,113) and (441,175) .. (418,174) .. controls (395,173) and (376.33,102.33) .. (328.99,111.34) .. controls (281.65,120.35) and (258.33,173.33) .. (224.33,175.33) .. controls (190.33,177.33) and (227.99,65.34) .. (375,60) -- cycle ;
\draw  [fill={rgb, 255:red, 208; green, 2; blue, 27 }  ,fill opacity=1 ] (136,163.31) .. controls (136,160.93) and (137.85,159) .. (140.13,159) .. controls (142.4,159) and (144.25,160.93) .. (144.25,163.31) .. controls (144.25,165.69) and (142.4,167.62) .. (140.13,167.62) .. controls (137.85,167.62) and (136,165.69) .. (136,163.31) -- cycle ;
\draw  [fill={rgb, 255:red, 208; green, 2; blue, 27 }  ,fill opacity=1 ] (323.08,163.98) .. controls (323.08,161.6) and (324.93,159.67) .. (327.2,159.67) .. controls (329.48,159.67) and (331.33,161.6) .. (331.33,163.98) .. controls (331.33,166.36) and (329.48,168.29) .. (327.2,168.29) .. controls (324.93,168.29) and (323.08,166.36) .. (323.08,163.98) -- cycle ;
\draw  [fill={rgb, 255:red, 208; green, 2; blue, 27 }  ,fill opacity=1 ] (373.21,163.57) .. controls (373.21,161.19) and (375.05,159.26) .. (377.33,159.26) .. controls (379.61,159.26) and (381.45,161.19) .. (381.45,163.57) .. controls (381.45,165.95) and (379.61,167.87) .. (377.33,167.87) .. controls (375.05,167.87) and (373.21,165.95) .. (373.21,163.57) -- cycle ;
\draw  [fill={rgb, 255:red, 208; green, 2; blue, 52 }  ,fill opacity=1 ] (533.46,164.31) .. controls (533.46,161.93) and (535.31,160) .. (537.59,160) .. controls (539.86,160) and (541.71,161.93) .. (541.71,164.31) .. controls (541.71,166.69) and (539.86,168.62) .. (537.59,168.62) .. controls (535.31,168.62) and (533.46,166.69) .. (533.46,164.31) -- cycle ;
\draw  [fill={rgb, 255:red, 208; green, 2; blue, 52 }  ,fill opacity=1 ] (492.35,164.03) .. controls (492.35,161.65) and (494.2,159.73) .. (496.48,159.73) .. controls (498.75,159.73) and (500.6,161.65) .. (500.6,164.03) .. controls (500.6,166.41) and (498.75,168.34) .. (496.48,168.34) .. controls (494.2,168.34) and (492.35,166.41) .. (492.35,164.03) -- cycle ;
\draw  [fill={rgb, 255:red, 208; green, 2; blue, 27 }  ,fill opacity=1 ] (413.87,164.03) .. controls (413.87,161.65) and (415.71,159.73) .. (417.99,159.73) .. controls (420.27,159.73) and (422.11,161.65) .. (422.11,164.03) .. controls (422.11,166.41) and (420.27,168.34) .. (417.99,168.34) .. controls (415.71,168.34) and (413.87,166.41) .. (413.87,164.03) -- cycle ;
\draw  [dash pattern={on 4.5pt off 4.5pt}] (122.33,126) -- (344.33,126) -- (344.33,180) -- (122.33,180) -- cycle ;
\draw   (112,111) -- (564,111) -- (564,219.33) -- (112,219.33) -- cycle ;
\draw  [fill={rgb, 255:red, 208; green, 2; blue, 52 }  ,fill opacity=1 ] (224.05,163) .. controls (224.05,160.62) and (225.9,158.69) .. (228.18,158.69) .. controls (230.45,158.69) and (232.3,160.62) .. (232.3,163) .. controls (232.3,165.38) and (230.45,167.31) .. (228.18,167.31) .. controls (225.9,167.31) and (224.05,165.38) .. (224.05,163) -- cycle ;

\draw (133.45,141.61) node [anchor=north west][inner sep=0.75pt]    {$x_{0}$};
\draw (318.27,140.09) node [anchor=north west][inner sep=0.75pt]    {$x_{1}$};
\draw (368.09,140.28) node [anchor=north west][inner sep=0.75pt]    {$x_{2}$};
\draw (527.95,140.02) node [anchor=north west][inner sep=0.75pt]    {$x_{n}$};
\draw (448.97,159.07) node [anchor=north west][inner sep=0.75pt]    {$\cdots $};
\draw (411.01,140.62) node [anchor=north west][inner sep=0.75pt]    {$x_{3}$};
\draw (484.03,140.62) node [anchor=north west][inner sep=0.75pt]    {$x_{n-1}$};
\draw (121.33,186.4) node [anchor=north west][inner sep=0.75pt]    {$\{x_{0}\} \cup \left\{x_{ij} :\ ij\in [ n]^{( 2)}\right\} \cup \{x_{1}\}$};
\draw (220,140.4) node [anchor=north west][inner sep=0.75pt]    {$x_{3n}$};
\draw (172.97,157.07) node [anchor=north west][inner sep=0.75pt]    {$\cdots $};
\draw (269.97,157.07) node [anchor=north west][inner sep=0.75pt]    {$\cdots $};

\end{tikzpicture}
\caption{A member of $\cF^{(3)}_I(n)$ with $I=\{1,2\}$, vertex set $\{x_0\}\cup\left\{x_{ij}:\;ij\in [n]^{( 2)}\right\}\cup\{x_1,\dots,x_n\} $ and edge set $F = \{x_{0},x_1,x_2\}\cup\set{\{x_{ij}, x_i, x_j\}: ij\in [n]^{(2)}}$.}
\label{fig: F_Ifor3}
\end{figure}

\begin{definition}[\label{def:F_I}The families $\cF_I^{(k)}(n)$, $\rev\cF_I^{(k)}(n)$ and $\cF^{*}_I(n,k)$]
 Let $k, n \geq 3$, and $I\in [n]^{(k-1)}$ with $1\in I$.
 \begin{enumerate}
\item  \textbf{$\cF_I^{(k)}(n)$:} An ordered hypergraph $(F,<)$ is an element of $\cF_I^{(k)}(n)$ if it satisfies
\begin{itemize}
    \item $V(F)=\set{x_0,\ldots,x_n}\cup \set{x_J:\;J\in \set{2,\ldots,n}^{(k-1)}\cup\set{I}}$, where for $J_1\ne J_2$ we may have $x_{J_1} = x_{J_2}$ and $x_J$ may be equal to $x_1$ or $x_0$.
    \item The ordering of the vertices satisfies $x_0<x_1<\cdots<x_n$ and $x_0\le x_J \le x_1$ for all $J\in \set{2,\ldots,n}^{(k-1)}$, and $x_I=x_0$,
    \item The edge set is given by $$F=\set{\set{x_J}\cup\set{x_j}_{j\in J}:\; J\in \set{2,\ldots,n}^{(k-1)}\cup\set{I}}.$$
\end{itemize}
We will refer to the edge $\{x_0\}\cup \{x_j\}_{j\in I}$ as the \textit{special edge} and to $\set{x_0,\ldots,x_n}$ as the \textit{distinguished set of vertices}.

 \item \textbf{$\rev \cF_I^{(k)}(n)$:} For each $(F,<)$ we will also consider the ordered hypergraph $(F,<_{\rev})$ where $<_{\rev}$ is the ``reverse'' order, i.e., the order satisfying
    \begin{align*}
        x<y \quad \text{ if and only if }\quad y<_{\rev}x.
    \end{align*}
          Let $\rev \cF_I^{(k)}(n)$ be the family of all reversed ordered hypergraphs \mbox{$(F,<_{\rev})$}. 
        \item  \textbf{$\cF_I^{*}(n,k)$:} Let $\cF_I^{\,*}(n,k)$ denote the family of ordered hypergraphs $(F,<)$ which contain both $(F_1,<)$ and $(F_2, <_{\rev})$  as  ordered subhypergraphs, for some $(F_1,<)$ and $(F_2, <)\in  \cF_I^{(k)}(n)$. 
    \end{enumerate}
\end{definition}
\begin{remark}
Note that while $k$, $I$ and $n$ are fixed for $\cF_{I}^{(k)}(n)$, its members $(F,<)$ are distinguished by
\begin{itemize} 
     \item different equivalence relations on the sets  $J\in\{2,3\dots,n\}^{(k-1)}$,
     \item the assignment of one additional vertex $x_J$ to all $(k-1)$-sets in the equivalence class of $J$, and
     \item the order on $\{x_J:J\in \{2,\dots, n\}^{(k-1)}\cup \{I\}\}.$
\end{itemize}

\end{remark}
The following theorem gives a tower-type lower bound on the Ramsey number of ordered $k$-graphs in the family $\cF_I^{*}(n,k).$ 
\begin{theorem}
\label{thm: steppingup}
   For every $k\geq 3$, there exists a constant $b_k> 0$ and $n_0:=n_0(k)$ such that for every $n\geq n_0$, there exists $I\in [n]^{(k-1)}$ such that  for any $(F,<)\in \cF_I^{\,*}(n,k)$,
   $$[t_{k-1}(n^{b_k})]\narrow (F,<)_4^k.$$
\end{theorem}
\subsection{Preliminaries}\label{subsec:prelim}
The proof of Theorem \ref{thm: steppingup} is based on a variant of the negative stepping up lemma due to Erd\H{o}s, Hajnal and Rado \cite{EHR65}. First we introduce binary trees and sequences which will be used to define the colorings used in Theorem \ref{thm: steppingup}.
\subsubsection{Binary trees, ancestors and descendants} 
\label{subsubsec: TreesBinarySequences}
Given an integer $N$, let $T = T(N)$ be the  binary tree with height $N$ and ordered set of leaves $[2^N]$. Let $V(T)$ denote the $2^{N+1}-1$ vertices of $T$. We also identify the levels of the tree with $[N+1]$, where the root is at level 1 and the leaves are at level $N+1$. 

Given a vertex $u\in V(T)$, let $\pi(u)$ denote the level of the vertex $u$ in the tree. Given vertices $u,v\in V(T)$ with $\pi(u)< \pi(v)$, we say that $u$ is an \textit{ancestor} of $v$ if the path from $u$ to $v$ has at most one vertex in each level. Given any pair of leaves $\{x,y\}\subseteq [2^N]$, let the \textit{greatest common ancestor} of $x$ and $y$, denoted by $a(x,y)$, be the vertex of $T$ of the highest level which is an ancestor of both $x$ and $y$. Further, let $\delta(x,y) = \pi (a(x,y)).$

Given a pair of vertices $u,v\in V(T)$ where $u$ is an ancestor of $v$, we say that $v$ is a \textit{left descendant} of $u$, or $v\in L(u)$, if the unique path from $u$ to $v$ in $T$ passes through the \textit{left child} of $u$ in $T$. Similarly, we say that $v$ is a \textit{right descendant} of $u$, or $v\in R(u)$, if the unique path from $u$ to $v$ in $T$ passes through the \textit{right child} of $u$ in $T$. 
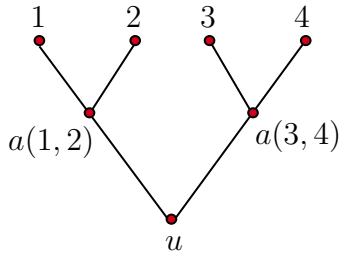
\begin{figure}[h!]
    \centering

\tikzset{every picture/.style={line width=0.75pt}} 

\begin{tikzpicture}[x=0.75pt,y=0.75pt,yscale=-0.7,xscale=0.7]

\draw    (345.71,109.65) -- (374.71,159.57) ;
\draw    (224.4,113.79) -- (318.53,241.26) ;
\draw    (414.32,109.49) -- (318.53,241.26) ;
\draw    (292.71,111.43) -- (260.12,162.24) ;
\draw  [fill={rgb, 255:red, 208; green, 2; blue, 27 }  ,fill opacity=1 ] (315.1,237.98) .. controls (315.1,236.17) and (316.63,234.7) .. (318.53,234.7) .. controls (320.42,234.7) and (321.96,236.17) .. (321.96,237.98) .. controls (321.96,239.79) and (320.42,241.26) .. (318.53,241.26) .. controls (316.63,241.26) and (315.1,239.79) .. (315.1,237.98) -- cycle ;
\draw  [fill={rgb, 255:red, 208; green, 2; blue, 27 }  ,fill opacity=1 ] (410.89,109.65) .. controls (410.89,107.83) and (412.43,106.36) .. (414.32,106.36) .. controls (416.21,106.36) and (417.75,107.83) .. (417.75,109.65) .. controls (417.75,111.46) and (416.21,112.93) .. (414.32,112.93) .. controls (412.43,112.93) and (410.89,111.46) .. (410.89,109.65) -- cycle ;
\draw  [fill={rgb, 255:red, 208; green, 2; blue, 27 }  ,fill opacity=1 ] (289.29,109.65) .. controls (289.29,107.83) and (290.82,106.36) .. (292.71,106.36) .. controls (294.61,106.36) and (296.14,107.83) .. (296.14,109.65) .. controls (296.14,111.46) and (294.61,112.93) .. (292.71,112.93) .. controls (290.82,112.93) and (289.29,111.46) .. (289.29,109.65) -- cycle ;
\draw  [fill={rgb, 255:red, 208; green, 2; blue, 27 }  ,fill opacity=1 ] (220.97,109.65) .. controls (220.97,107.83) and (222.5,106.36) .. (224.4,106.36) .. controls (226.29,106.36) and (227.82,107.83) .. (227.82,109.65) .. controls (227.82,111.46) and (226.29,112.93) .. (224.4,112.93) .. controls (222.5,112.93) and (220.97,111.46) .. (220.97,109.65) -- cycle ;
\draw  [fill={rgb, 255:red, 208; green, 2; blue, 27 }  ,fill opacity=1 ] (256.69,161.55) .. controls (256.69,159.73) and (258.22,158.26) .. (260.12,158.26) .. controls (262.01,158.26) and (263.54,159.73) .. (263.54,161.55) .. controls (263.54,163.36) and (262.01,164.83) .. (260.12,164.83) .. controls (258.22,164.83) and (256.69,163.36) .. (256.69,161.55) -- cycle ;
\draw  [fill={rgb, 255:red, 208; green, 2; blue, 27 }  ,fill opacity=1 ] (342.29,109.65) .. controls (342.29,107.83) and (343.82,106.36) .. (345.71,106.36) .. controls (347.61,106.36) and (349.14,107.83) .. (349.14,109.65) .. controls (349.14,111.46) and (347.61,112.93) .. (345.71,112.93) .. controls (343.82,112.93) and (342.29,111.46) .. (342.29,109.65) -- cycle ;
\draw  [fill={rgb, 255:red, 208; green, 2; blue, 27 }  ,fill opacity=1 ] (373.29,161.55) .. controls (373.29,159.73) and (374.82,158.26) .. (376.71,158.26) .. controls (378.61,158.26) and (380.14,159.73) .. (380.14,161.55) .. controls (380.14,163.36) and (378.61,164.83) .. (376.71,164.83) .. controls (374.82,164.83) and (373.29,163.36) .. (373.29,161.55) -- cycle ;

\draw (216,83) node [anchor=north west][inner sep=0.75pt]    {$1$};
\draw (404,83) node [anchor=north west][inner sep=0.75pt]    {$4$};
\draw (284,83) node [anchor=north west][inner sep=0.75pt]    {$2$};
\draw (337,83) node [anchor=north west][inner sep=0.75pt]    {$3$};
\draw (312,248.05) node [anchor=north west][inner sep=0.75pt]    {$u$};
\draw (200,168.8) node [anchor=north west][inner sep=0.75pt]    {$a(1,2)$};
\draw (376,165) node [anchor=north west][inner sep=0.75pt]    {$a( 3,4)$};

\end{tikzpicture}
    \caption{In this example where $N = 2$, the leaves 1 and 2 are \textit{left descendants of $u$} and $a(1,2)$ is the  \textit{left child} of $u$.}
    \label{fig: LeftRightChild}
\end{figure}

We remark that given a pair of leafs $x,y\in [2^N]$ with $x<y$, $u=a(x,y)$ is the unique vertex in $V(T)$ such that $x\in L(u)$ and $y\in R(u).$ Given a subset of the leaves $X\subseteq [2^N]$ of size at least 2, let $u_X:= a (\min X, \max X)$. Let the set of \textit{left descendants} of $u_X$ in $X$ be $X_L(u_X):= X\cap L(u_X).$ Similarly, let the set of \textit{right descendants} of $u_X$ in $X$ be $X_R(u_X):=X\cap R(u_X).$ We note that $\min X\in X_L(u_X)$ and $\max X\in X_R(u_X)$ and so each of them is non-empty, and further
\begin{align}
\label{eqn: XLXR}
    \max X_L(u_X) < \min X_R(u_X) \text{ and } X = X_L(u_X)\dcup X_R(u_X).
\end{align}

\subsubsection{Shapes of leaf sets}

Given a positive integer $t\geq 3$, we will divide subsets $X  \subseteq  [2^N]$ of size $t$ into two groups -- \textit{combs} and \textit{splits}. These will eventually be used to define the coloring of the $k$-subsets of $[2^N]$ used for Theorem \ref{thm: steppingup}. First, we consider combs, which were also used in \cite{EHR65}, and correspond to sets $X = \{x_1< \cdots < x_t\}$ where the $\delta(x_i,x_{i+1})$ form a monotone sequence.
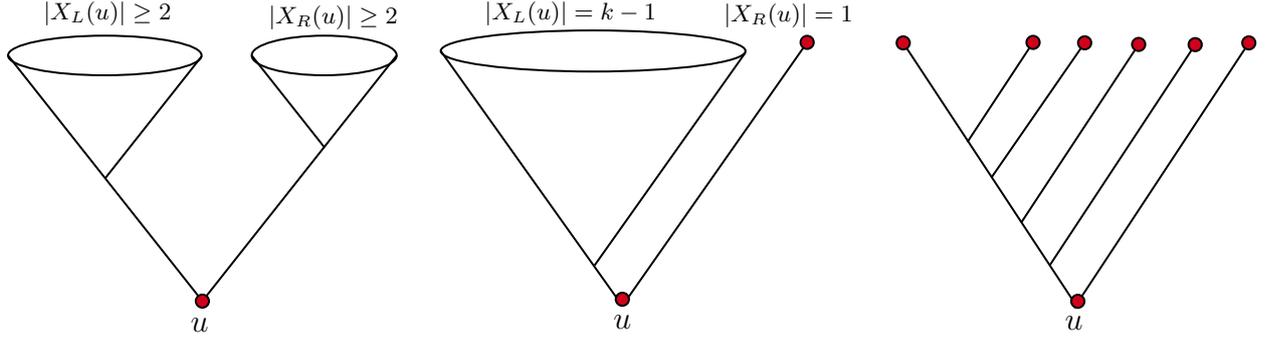
\begin{figure}[t]
    \centering
\tikzset{every picture/.style={line width=0.75pt}} 

\tikzset{every picture/.style={line width=0.75pt}} 

\begin{tikzpicture}[x=0.75pt,y=0.75pt,yscale=-1,xscale=1]

\draw    (605.58,109.8) -- (533.19,220.49) ;
\draw    (13,115.24) -- (110,238.8) ;
\draw    (207,115.24) -- (110,238.8) ;
\draw   (13,115.24) .. controls (13,109.74) and (34.63,105.28) .. (61.3,105.28) .. controls (87.98,105.28) and (109.61,109.74) .. (109.61,115.24) .. controls (109.61,120.75) and (87.98,125.21) .. (61.3,125.21) .. controls (34.63,125.21) and (13,120.75) .. (13,115.24) -- cycle ;
\draw   (134.64,115.24) .. controls (134.64,109.74) and (150.84,105.28) .. (170.82,105.28) .. controls (190.8,105.28) and (207,109.74) .. (207,115.24) .. controls (207,120.75) and (190.8,125.21) .. (170.82,125.21) .. controls (150.84,125.21) and (134.64,120.75) .. (134.64,115.24) -- cycle ;
\draw    (134.64,115.24) -- (171.02,161.58) ;
\draw    (109.61,115.74) -- (61.5,177.02) ;
\draw    (229,112.98) -- (319.65,241.17) ;
\draw    (411.89,108.66) -- (319.65,241.17) ;
\draw   (229,112.98) .. controls (229,107.27) and (263.05,102.64) .. (305.05,102.64) .. controls (347.05,102.64) and (381.1,107.27) .. (381.1,112.98) .. controls (381.1,118.69) and (347.05,123.32) .. (305.05,123.32) .. controls (263.05,123.32) and (229,118.69) .. (229,112.98) -- cycle ;
\draw    (381.1,113.76) -- (305.77,220.83) ;
\draw    (459.84,108.92) -- (547.01,242.2) ;
\draw    (632.36,108.92) -- (545.19,242.2) ;
\draw    (577.36,109.69) -- (518.79,199.24) ;
\draw    (550.42,108.89) -- (504.01,176.45) ;
\draw    (524.66,108.67) -- (491.84,158.85) ;
\draw  [fill={rgb, 255:red, 208; green, 2; blue, 27 }  ,fill opacity=1 ] (106.7,238.8) .. controls (106.7,236.98) and (108.18,235.5) .. (110,235.5) .. controls (111.82,235.5) and (113.3,236.98) .. (113.3,238.8) .. controls (113.3,240.62) and (111.82,242.1) .. (110,242.1) .. controls (108.18,242.1) and (106.7,240.62) .. (106.7,238.8) -- cycle ;
\draw  [fill={rgb, 255:red, 208; green, 2; blue, 27 }  ,fill opacity=1 ] (316.35,237.87) .. controls (316.35,236.04) and (317.82,234.56) .. (319.65,234.56) .. controls (321.47,234.56) and (322.95,236.04) .. (322.95,237.87) .. controls (322.95,239.69) and (321.47,241.17) .. (319.65,241.17) .. controls (317.82,241.17) and (316.35,239.69) .. (316.35,237.87) -- cycle ;
\draw  [fill={rgb, 255:red, 208; green, 2; blue, 27 }  ,fill opacity=1 ] (408.59,108.66) .. controls (408.59,106.83) and (410.07,105.35) .. (411.89,105.35) .. controls (413.72,105.35) and (415.19,106.83) .. (415.19,108.66) .. controls (415.19,110.48) and (413.72,111.96) .. (411.89,111.96) .. controls (410.07,111.96) and (408.59,110.48) .. (408.59,108.66) -- cycle ;
\draw  [fill={rgb, 255:red, 208; green, 2; blue, 27 }  ,fill opacity=1 ] (543.71,238.9) .. controls (543.71,237.08) and (545.19,235.6) .. (547.01,235.6) .. controls (548.84,235.6) and (550.31,237.08) .. (550.31,238.9) .. controls (550.31,240.72) and (548.84,242.2) .. (547.01,242.2) .. controls (545.19,242.2) and (543.71,240.72) .. (543.71,238.9) -- cycle ;
\draw  [fill={rgb, 255:red, 208; green, 2; blue, 27 }  ,fill opacity=1 ] (456.54,108.92) .. controls (456.54,107.1) and (458.02,105.62) .. (459.84,105.62) .. controls (461.67,105.62) and (463.15,107.1) .. (463.15,108.92) .. controls (463.15,110.74) and (461.67,112.22) .. (459.84,112.22) .. controls (458.02,112.22) and (456.54,110.74) .. (456.54,108.92) -- cycle ;
\draw  [fill={rgb, 255:red, 208; green, 2; blue, 27 }  ,fill opacity=1 ] (521.36,108.67) .. controls (521.36,106.84) and (522.84,105.37) .. (524.66,105.37) .. controls (526.49,105.37) and (527.97,106.84) .. (527.97,108.67) .. controls (527.97,110.49) and (526.49,111.97) .. (524.66,111.97) .. controls (522.84,111.97) and (521.36,110.49) .. (521.36,108.67) -- cycle ;
\draw  [fill={rgb, 255:red, 208; green, 2; blue, 27 }  ,fill opacity=1 ] (547.12,108.89) .. controls (547.12,107.07) and (548.6,105.59) .. (550.42,105.59) .. controls (552.24,105.59) and (553.72,107.07) .. (553.72,108.89) .. controls (553.72,110.71) and (552.24,112.19) .. (550.42,112.19) .. controls (548.6,112.19) and (547.12,110.71) .. (547.12,108.89) -- cycle ;
\draw  [fill={rgb, 255:red, 208; green, 2; blue, 27 }  ,fill opacity=1 ] (574.06,109.69) .. controls (574.06,107.87) and (575.54,106.39) .. (577.36,106.39) .. controls (579.18,106.39) and (580.66,107.87) .. (580.66,109.69) .. controls (580.66,111.51) and (579.18,112.99) .. (577.36,112.99) .. controls (575.54,112.99) and (574.06,111.51) .. (574.06,109.69) -- cycle ;
\draw  [fill={rgb, 255:red, 208; green, 2; blue, 27 }  ,fill opacity=1 ] (602.28,109.8) .. controls (602.28,107.97) and (603.76,106.49) .. (605.58,106.49) .. controls (607.41,106.49) and (608.89,107.97) .. (608.89,109.8) .. controls (608.89,111.62) and (607.41,113.1) .. (605.58,113.1) .. controls (603.76,113.1) and (602.28,111.62) .. (602.28,109.8) -- cycle ;
\draw  [fill={rgb, 255:red, 208; green, 2; blue, 27 }  ,fill opacity=1 ] (629.06,108.92) .. controls (629.06,107.1) and (630.54,105.62) .. (632.36,105.62) .. controls (634.19,105.62) and (635.66,107.1) .. (635.66,108.92) .. controls (635.66,110.74) and (634.19,112.22) .. (632.36,112.22) .. controls (630.54,112.22) and (629.06,110.74) .. (629.06,108.92) -- cycle ;

\draw (29.24,86.83) node [anchor=north west][inner sep=0.75pt]  [font=\scriptsize]  {$|X_{L}( u) |\geq 2$};
\draw (141.57,88.83) node [anchor=north west][inner sep=0.75pt]  [font=\scriptsize]  {$|X_{R}( u) |\geq 2$};
\draw (102.7,245.5) node [anchor=north west][inner sep=0.75pt]    {$u$};
\draw (249.58,86.62) node [anchor=north west][inner sep=0.75pt]  [font=\scriptsize]  {$|X_{L}( u) | = k-1$};
\draw (368.92,87.62) node [anchor=north west][inner sep=0.75pt]  [font=\scriptsize]  {$|X_{R}( u) | = 1$};
\draw (313.65,244.42) node [anchor=north west][inner sep=0.75pt]    {$u$};
\draw (539.21,244.9) node [anchor=north west][inner sep=0.75pt]    {$u$};

\end{tikzpicture}
    \caption{A balanced split,  $(k-1,1)$ --  split and left comb.}
    \label{fig: shapes}
\end{figure}
\begin{definition}[Combs]
\label{def: comb}
     Let $X = \{x_1<\cdots < x_t\}\subseteq [2^N]$  and let $u = u_X$.  
     \begin{enumerate}
    \item\label{it: leftcomb} \textbf{Left comb:} We say that a set $X$ forms a \textit{left comb}, if 
    $$\d(x_{1},x_{2})> \delta(x_2,x_3)> \cdots > \delta(x_{t-1},x_t)$$ 
    \item \textbf{Right comb:} We say that a set $X$ forms a \textit{right comb}, if 
    $$\d(x_{1},x_{2})< \delta(x_2,x_3)< \cdots < \delta(x_{t-1},x_t)$$ 
    \end{enumerate}
\end{definition}
Given a set $X = \{x_1<\cdots < x_t\}$, let 
\begin{align*}
    a(X) &:= \{a(x_{i-1}, x_i): 2\leq i \leq t\},\\
    \pi(X)&:= \pi(a(X)) =  \{\delta(x_{i-1},x_{i}): 2 \leq i \leq t\}.
\end{align*}
Next, we define splits. 
\begin{definition}[Splits]
    Let $X = \{x_1<\cdots < x_t\}\subseteq [2^N]$, $u = u_X$ and let $\ell$, $r$ be positive integers.  We say that a set $X$ forms an \textit{$(\ell,r)$-split} if it is not a \textit{left or right comb} and
    $$|X_L(u)| = \ell \text{ and } |X_R(u)| = r. $$
    If $X$ forms an $(\ell,r)$-split with both $\ell, r\ge 2$, then we say that $X$ forms a \textit{balanced split}. 
\end{definition}
Next, we mention some properties of combs which will be used in the proof of Theorem \ref{thm: steppingup}. First we note that for $X = \{x< y< z\}$,   either  $X$ forms a \textit{left comb} or a \textit{right comb} (see Figure \ref{fig:3combs}). Note that $X$ forming a left comb is equivalent to  $\delta(x,y) > \delta(y,z)$ and $\delta(x,z)= \delta(y,z)$,  and similarly, forming a right comb is equivalent to $\delta(x,z) = \delta(x,y) < \delta(y,z)$. It will also be convenient to use that for a left comb 
$$|X_L(u_X)| = 2 \text{ and } |X_R(u_X)| = 1, \quad \text{ while }  \quad |X_L(u_X)| = 1 \text{ and } |X_R(u_X)| = 2$$
for a right comb. Further, for $\{x< y < z\}$ we always have $\delta(x,y)\neq \delta(y,z)$ and
$$\delta(x,z) = \min \{\delta(x,y), \delta(y,z)\}.$$
Similarly, for any set $X = \{x_1< \cdots  < x_t\}\subseteq [2^N]$, one can observe that 
\begin{align}
\label{eqn: mindelta}
    \delta(x_1,x_t) = \min\{\delta(x_{i-1},x_t): 2\le  i \le t\}.
\end{align}

\begin{figure}[h]
    \centering
\tikzset{every picture/.style={line width=0.75pt}} 

\begin{tikzpicture}[x=0.75pt,y=0.75pt,yscale=-0.8,xscale=0.8]

\draw    (355.99,92.1) -- (456,220.8) ;
\draw    (552.99,93.1) -- (456,220.8) ;
\draw    (461.28,90.03) -- (509.8,150.95) ;
\draw    (88.4,95.79) -- (182.53,223.26) ;
\draw    (278.32,91.49) -- (182.53,223.26) ;
\draw    (193.27,91.94) -- (141.12,167.24) ;
\draw  [fill={rgb, 255:red, 208; green, 2; blue, 27 }  ,fill opacity=1 ] (452.7,220.8) .. controls (452.7,218.98) and (454.18,217.5) .. (456,217.5) .. controls (457.82,217.5) and (459.3,218.98) .. (459.3,220.8) .. controls (459.3,222.62) and (457.82,224.1) .. (456,224.1) .. controls (454.18,224.1) and (452.7,222.62) .. (452.7,220.8) -- cycle ;
\draw  [fill={rgb, 255:red, 208; green, 2; blue, 27 }  ,fill opacity=1 ] (179.1,219.98) .. controls (179.1,218.17) and (180.63,216.7) .. (182.53,216.7) .. controls (184.42,216.7) and (185.96,218.17) .. (185.96,219.98) .. controls (185.96,221.79) and (184.42,223.26) .. (182.53,223.26) .. controls (180.63,223.26) and (179.1,221.79) .. (179.1,219.98) -- cycle ;
\draw  [fill={rgb, 255:red, 208; green, 2; blue, 27 }  ,fill opacity=1 ] (274.89,91.49) .. controls (274.89,89.68) and (276.43,88.21) .. (278.32,88.21) .. controls (280.21,88.21) and (281.75,89.68) .. (281.75,91.49) .. controls (281.75,93.3) and (280.21,94.77) .. (278.32,94.77) .. controls (276.43,94.77) and (274.89,93.3) .. (274.89,91.49) -- cycle ;
\draw  [fill={rgb, 255:red, 208; green, 2; blue, 27 }  ,fill opacity=1 ] (189.84,91.94) .. controls (189.84,90.12) and (191.37,88.65) .. (193.27,88.65) .. controls (195.16,88.65) and (196.7,90.12) .. (196.7,91.94) .. controls (196.7,93.75) and (195.16,95.22) .. (193.27,95.22) .. controls (191.37,95.22) and (189.84,93.75) .. (189.84,91.94) -- cycle ;
\draw  [fill={rgb, 255:red, 208; green, 2; blue, 27 }  ,fill opacity=1 ] (84.97,92.51) .. controls (84.97,90.7) and (86.5,89.23) .. (88.4,89.23) .. controls (90.29,89.23) and (91.82,90.7) .. (91.82,92.51) .. controls (91.82,94.32) and (90.29,95.79) .. (88.4,95.79) .. controls (86.5,95.79) and (84.97,94.32) .. (84.97,92.51) -- cycle ;
\draw  [fill={rgb, 255:red, 208; green, 2; blue, 27 }  ,fill opacity=1 ] (352.69,92.1) .. controls (352.69,90.28) and (354.17,88.8) .. (355.99,88.8) .. controls (357.81,88.8) and (359.29,90.28) .. (359.29,92.1) .. controls (359.29,93.93) and (357.81,95.41) .. (355.99,95.41) .. controls (354.17,95.41) and (352.69,93.93) .. (352.69,92.1) -- cycle ;
\draw  [fill={rgb, 255:red, 208; green, 2; blue, 27 }  ,fill opacity=1 ] (457.97,90.03) .. controls (457.97,88.21) and (459.45,86.73) .. (461.28,86.73) .. controls (463.1,86.73) and (464.58,88.21) .. (464.58,90.03) .. controls (464.58,91.86) and (463.1,93.34) .. (461.28,93.34) .. controls (459.45,93.34) and (457.97,91.86) .. (457.97,90.03) -- cycle ;
\draw  [fill={rgb, 255:red, 208; green, 2; blue, 27 }  ,fill opacity=1 ] (551.69,89.8) .. controls (551.69,87.98) and (553.17,86.5) .. (554.99,86.5) .. controls (556.81,86.5) and (558.29,87.98) .. (558.29,89.8) .. controls (558.29,91.63) and (556.81,93.1) .. (554.99,93.1) .. controls (553.17,93.1) and (551.69,91.63) .. (551.69,89.8) -- cycle ;
\draw  [fill={rgb, 255:red, 208; green, 2; blue, 27 }  ,fill opacity=1 ] (506.49,150.95) .. controls (506.49,149.13) and (507.97,147.65) .. (509.8,147.65) .. controls (511.62,147.65) and (513.1,149.13) .. (513.1,150.95) .. controls (513.1,152.78) and (511.62,154.25) .. (509.8,154.25) .. controls (507.97,154.25) and (506.49,152.78) .. (506.49,150.95) -- cycle ;
\draw  [fill={rgb, 255:red, 208; green, 2; blue, 27 }  ,fill opacity=1 ] (137.69,167.24) .. controls (137.69,165.42) and (139.22,163.95) .. (141.12,163.95) .. controls (143.01,163.95) and (144.54,165.42) .. (144.54,167.24) .. controls (144.54,169.05) and (143.01,170.52) .. (141.12,170.52) .. controls (139.22,170.52) and (137.69,169.05) .. (137.69,167.24) -- cycle ;

\draw (397.7,225.5) node [anchor=north west][inner sep=0.75pt]    {$u_X = a( x,y)  = a( x,z)$};
\draw (347.7,65.5) node [anchor=north west][inner sep=0.75pt]    {$x$};
\draw (453.7,64.5) node [anchor=north west][inner sep=0.75pt]    {$y$};
\draw (548.7,65.5) node [anchor=north west][inner sep=0.75pt]    {$z$};
\draw (82.97,68.45) node [anchor=north west][inner sep=0.75pt]    {$x$};
\draw (189.04,68.45) node [anchor=north west][inner sep=0.75pt]    {$y$};
\draw (279.65,68.43) node [anchor=north west][inner sep=0.75pt]    {$z$};
\draw (123.7,224.5) node [anchor=north west][inner sep=0.75pt]    {$u_X = a( x,z)  = a( y,z)$};
\draw (515.49,153.35) node [anchor=north west][inner sep=0.75pt]    {$a( y,z)$};
\draw (74.49,158.35) node [anchor=north west][inner sep=0.75pt]    {$a( x,y)$};
\end{tikzpicture}
    \caption{Any set $X= \{x< y< z\}$ either forms a left comb (on the left) or a right comb (on the right).}
    \label{fig:3combs}
\end{figure}
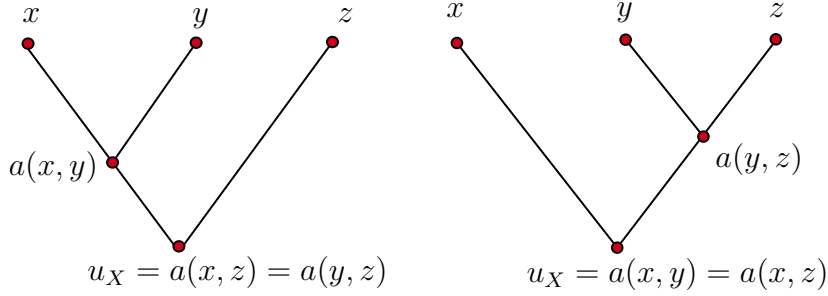
\begin{remark}
   \label{fact: combstructure}
   For larger combs, the following properties will be used later. Let $t\geq 3$ be a positive integer and let $X=\{x_1 < \cdots < x_t\} \subseteq [2^N].$
\begin{enumerate}
    \item  If $X$ forms a left comb, then
    \begin{enumerate}[label = (\roman*)]
    \item \label{it: comb1L} For any $2\le i < j \le t$, we have  $a(x_{j-1}, x_j) = a(x_i, x_j).$ 
    \item\label{it: comb2L} For any  $J = \{i_1<i_2<\cdots < i_{k}\}\subseteq [t]$, the subset $\{x_j:j\in J\}$ also forms a left comb and 
    $$\pi(\{x_j:j\in J\}) = \{\delta(x_{i_{j-1}}, x_{i_{j}})=\delta(x_{i_j-1},x_{i_j}):j\in \{2,\dots,k\}\}.$$  
\end{enumerate}
\item  \label{it: Comb2} Similarly if $X$ forms a right comb, 
    \begin{enumerate}[label = (\roman*)]
     \item\label{it: comb1R} For any $1\le i < j \le t-1$, we have   $a(x_{i}, x_{i+1}) = a(x_i, x_j)$. 
    \item\label{it: comb2R} For any  $J = \{i_1<i_2<\cdots < i_{k}\}\subseteq [t]$, the subset $ \{x_j:j\in J\}$  is also a right comb and 
    $$\pi(\{x_j:j\in J\}) = \{\delta(x_{i_j}, x_{i_{j+1}}) = \delta(x_{i_j}, x_{i_j + 1}:j\in \{1,\dots,k-1\}\}.$$ 
\end{enumerate}
\end{enumerate}
\end{remark}

\subsubsection{Coloring rules}
Here, we define the coloring rules used for coloring subsets of size $k$ of the leaf set  $[2^N]$, given a coloring of $[N]^{(k-1)}$. These will be the colorings used for the proof of Theorem \ref{thm: steppingup}. First, we describe the case $k = 3$. We remark that such a coloring was also used to give a double exponential  lower bound for $R(K_n^{(3)}; 4)$ (see \cite{GRSBook}). \\[0.3cm]
\textbf{Coloring rule for $k = 3$:} Given a positive integer $N$ and a coloring $c: [N]^{(2)}\to \{0,1\}$, let 
    $\chi_c: [2^N]^{(3)}\to \ZZ_4$ be given by: 
    \begin{align}
    \label{eqn: coloring3}
        \chi_c(X) = \begin{cases}
            c(\pi(X)) &\quad\text{if $X$ is a left comb}\\
            3- c(\pi(X))&\quad \text{if $X$ is a right comb.}
        \end{cases}
    \end{align}
For $k\geq 4$, we color \textit{left} and \textit{right} combs depending on the color of the projection, while in the case when $X$ does not form a comb, we color by the type of $(\ell,r)$-split that it forms. \\[0.3cm]
\textbf{Coloring rule for $k\geq 4$: }
    Given a positive integer $N$ and a coloring $c: [N]^{(k-1)}\to \ZZ_4$, let $\chi_c: [2^N]^{(k)} \to \ZZ_4$ be given by: 
     \begin{align}
     \label{eqn: coloring4}
        \chi_c(X) = \begin{cases}
            3- c(\pi(X)) &\quad\text{if $X$ forms a left comb}\\
            c(\pi(X))&\quad \text{if $X$ forms a right comb}\\
            0 &\quad \text{if $X$ forms a balanced split}\\
            1 &\quad \text{if $X$ forms a $(k-1,1)$-split}\\
            2 &\quad \text{if $X$ forms a $(1,k-1)$-split}.
        \end{cases}
    \end{align}
Given a coloring $\chi: [N]^{(k)} \to [r]$, we say that \textit{$\chi$ contains a monochromatic copy of $(F,<)$} if there exists some $i\in [r]$ such that there is a copy of $(F,<)$ in the ordered $k$-graph with vertex set $[N]$ and edge set $\chi^{-1}(i)$. 
\subsection{Stepping up ordered $(k,k-1)$-systems}\label{subsec:orderedstepup} In this section we prove Theorem \ref{thm: steppingup}, using a variant of the stepping up lemma introduced in \cite{EHR65}. 
The proof of Theorem \ref{thm: steppingup} is by induction. For the base case $k = 3$, we consider a 2-coloring of the projection $c:[N]^{(2)}\to \{0,1\}$ which contains no monochromatic clique $K_{n-1}$, and let $\chi_c:[2^N]^{(3)}\to \ZZ_4$ be as defined in (\ref{eqn: coloring3}). In Proposition \ref{prop: basecase}, we show that $\chi_c$ does not contain any member of 
$ \cF_I^{(k)}(n)$ in colors $\{0,1\}$ or $\rev \cF_I^{(k)}(n)$ in colors $\{2,3\}$ where $I= \{1,2\}$. 

For $k\ge 4$, let $n$, $t=n/8k$ be integers\footnote{Note that although $t=n/8k$ might not be an integer, we prefer to write in this way, since it simplifies the exposition and has no significant effect on the arguments.} and let $I'\in [t]^{(k-2)}$, $I\in [n]^{(k-1)}$ be conveniently chosen sets.  We assume that there exists a coloring $c:[N]^{(k-1)}\to \ZZ_4$, such that there is no monochromatic copy of a member of $\cF_{I'}^{(k-1)}(t)$ (as defined in Definition \ref{def:F_I}) in colors $\{0,1\}$, and no monochromatic copy of a member of $\rev \cF_{I'}^{(k-1)}(t)$ in colors $\{2,3\}$. We then consider $\chi_c:[2^N]^{(k)}\to \ZZ_4$ as defined in (\ref{eqn: coloring4}), and show that there is no monochromatic copy of any member of $ \cF_I^{(k)}(n)$ in colors $\{0,1\}$ or $\rev \cF_I^{(k)}(n)$ in colors $\{2,3\}$, in the two following steps: 
\begin{itemize}
    \item First, we show that if $\chi_c$ contains a monochromatic copy of $(F,<)\in \cF_I^{(k)}(n)$ in color $\{0,1\}$  in the leaf set  $[2^N]$, then a large subset of leaves $Y\subseteq V(F)$ forms a left or right comb in the tree $T(N)$. This requires considering several cases and is discussed in Section \ref{sec: caseanalysis}. The case when $\chi_c$ contains a monochromatic copy of some $\rev\cF_I^{(k)}(n)$ in color $\{2,3\}$ is analogous.
    \item Then we conclude that for the comb $Y\subseteq V(F)\subseteq [2^N]$, the set of vertices $\pi(Y)\subseteq [N]$ in the projection forms a monochromatic copy of some member of $\cF^{(k-1)}_{I'}(t)$ in $\{0,1\}$ or $\rev\cF_{I'}^{(k-1)}(t)$ in $\{2,3\}$, which is forbidden by the induction hypothesis. To find the forbidden copies in the projection, we will use  Propositions \ref{prop: leftcomb} and \ref{prop: rightprojection}, which we first state and prove  in Section \ref{sec: projection}. 
\end{itemize}
We first state the Propositions \ref{prop: leftcomb} and \ref{prop: rightprojection} in the following section, i.e., Section \ref{sec: projection}. In Section \ref{sec: caseanalysis}, we present the case analysis mentioned above and apply Propositions \ref{prop: leftcomb} and \ref{prop: rightprojection} stated below to complete the proof of Theorem \ref{thm: steppingup}. 
\subsubsection{Projections}\label{sec: projection} Next, we state Propositions \ref{prop: leftcomb} and \ref{prop: rightprojection} which will allow us to obtain a large forbidden monochromatic structure in the projection. When these two propositions are applied, the vertex set $Y$ in the following statements will be a subset of the vertex set of $(F,<)\in \cF_{I}^{(k)}(n)$. We will show that when $Y$ forms a left comb, $\pi(Y)$ contains a large clique, and when it forms a right comb, $\pi(Y)$ is a member of some $\cF_{I'}^{(k-1)}(t)$.
\begin{proposition}[Left Comb Projection] \label{prop: leftcomb}
    Let $N\geq \l \ge k\geq 3$  be positive integers, $c:[N]^{(k-1)}\to \ZZ_4$ be a coloring and $\chi_c:[2^N]^{(k)}\to \ZZ_4$ be as defined in (\ref{eqn: coloring3}) and (\ref{eqn: coloring4}). Let $Y = \{y_0 < y_1 < \cdots < y_\l\}\subseteq [2^N]$ be a subset of leaves of $T=T(N)$ satisfying the following three conditions:
    \begin{enumerate}[label = (\alph*)]
        \item \label{it: leftlema} $Y$ forms a left comb, 
        \item \label{it: leftlemb}for every $J\in \set{2\ldots,\l}^{(k-1)}$, there exists a vertex $y_J$ such that $y_0 \leq y_J \le y_1,$\footnote{Here we allow $y_{J_1}= y_{J_2}$ for $J_1\ne J_2$.}
         \item\label{it: leftlemmac}  there exists $\alpha\in \ZZ_4$ such that  $\chi_c(Y^J)=\alpha$ for every element $Y^J$ in 
        $$\{Y^J = \{y_J\}\cup \{y_j:j\in J\}: J\in \{2,\dots,\l\}^{(k-1)}\}.$$
    \end{enumerate}
    Then $\pi(Y\setminus\{y_0\})^{(k-1)}\subseteq [N]^{(k-1)}$ is monochromatic in the coloring $c$.  
\end{proposition}

\begin{proof}
   Let $J=\{j_1< \cdots < j_{k-1}\}\in \{2,\cdots, \ell\}^{(k-1)}$. We will first show that $Y^J$ is a left comb for every $J\in \{2,\dots, \ell\}^{(k-1)}$. For that, in view of Definition \ref{def: comb} (\ref{it: leftcomb}), we need to show that:
    \begin{align}
        \label{eqn: leftcombproj}
        \delta(y_J, y_{j_1}) > \delta(y_{j_1}, y_{j_2}) > \cdots > \delta(y_{j_{k-2}}, y_{j_{k-1}}).
    \end{align}
    Since $j_1\ge 2$ and $Y$ is a left comb, the set $\{y_0, y_1, y_{j_1}, y_{j_2},\dots, y_{j_{k-1}}\}\subseteq Y$ forms a left comb, equivalently
    \begin{align}
        \label{eqn: leftcombproj2}
        \delta(y_0,y_1) > \delta(y_1,y_{j_1}) > \delta(y_{j_1},y_{j_2}) > \dots > \delta(y_{j_{k-2}}, y_{j_{k-1}}).
    \end{align}
    Further, observe that since $y_0\le y_J\le y_1 < y_{j_1}$, and $y_0, y_1, y_{j_1}$ forms a left comb, in view of (\ref{eqn: mindelta}), we have $\delta(y_J,y_{j_1})= \delta(y_1,y_{j_1}).$
    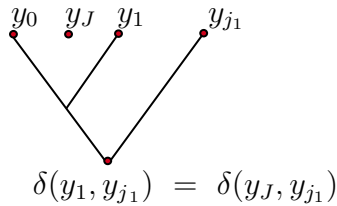
\begin{figure}[h!]
        \centering

\tikzset{every picture/.style={line width=0.75pt}} 

\begin{tikzpicture}[x=0.75pt,y=0.75pt,yscale=-0.5,xscale=0.5]

\draw    (204.4,109.79) -- (298.53,237.26) ;
\draw    (394.32,105.49) -- (298.53,237.26) ;
\draw    (309.27,105.94) -- (257.12,181.24) ;
\draw  [fill={rgb, 255:red, 208; green, 2; blue, 27 }  ,fill opacity=1 ] (295.1,233.98) .. controls (295.1,232.17) and (296.63,230.7) .. (298.53,230.7) .. controls (300.42,230.7) and (301.96,232.17) .. (301.96,233.98) .. controls (301.96,235.79) and (300.42,237.26) .. (298.53,237.26) .. controls (296.63,237.26) and (295.1,235.79) .. (295.1,233.98) -- cycle ;
\draw  [fill={rgb, 255:red, 208; green, 2; blue, 27 }  ,fill opacity=1 ] (390.89,105.49) .. controls (390.89,103.68) and (392.43,102.21) .. (394.32,102.21) .. controls (396.21,102.21) and (397.75,103.68) .. (397.75,105.49) .. controls (397.75,107.3) and (396.21,108.77) .. (394.32,108.77) .. controls (392.43,108.77) and (390.89,107.3) .. (390.89,105.49) -- cycle ;
\draw  [fill={rgb, 255:red, 208; green, 2; blue, 27 }  ,fill opacity=1 ] (305.84,105.94) .. controls (305.84,104.12) and (307.37,102.65) .. (309.27,102.65) .. controls (311.16,102.65) and (312.7,104.12) .. (312.7,105.94) .. controls (312.7,107.75) and (311.16,109.22) .. (309.27,109.22) .. controls (307.37,109.22) and (305.84,107.75) .. (305.84,105.94) -- cycle ;
\draw  [fill={rgb, 255:red, 208; green, 2; blue, 27 }  ,fill opacity=1 ] (200.97,106.51) .. controls (200.97,104.7) and (202.5,103.23) .. (204.4,103.23) .. controls (206.29,103.23) and (207.82,104.7) .. (207.82,106.51) .. controls (207.82,108.32) and (206.29,109.79) .. (204.4,109.79) .. controls (202.5,109.79) and (200.97,108.32) .. (200.97,106.51) -- cycle ;
\draw  [fill={rgb, 255:red, 208; green, 2; blue, 27 }  ,fill opacity=1 ] (255.97,106.51) .. controls (255.97,104.7) and (257.5,103.23) .. (259.4,103.23) .. controls (261.29,103.23) and (262.82,104.7) .. (262.82,106.51) .. controls (262.82,108.32) and (261.29,109.79) .. (259.4,109.79) .. controls (257.5,109.79) and (255.97,108.32) .. (255.97,106.51) -- cycle ;

\draw (198.97,77) node [anchor=north west][inner sep=0.75pt]    {$y_{0}$};
\draw (305.04,77) node [anchor=north west][inner sep=0.75pt]    {$y_{1}$};
\draw (395.65,77) node [anchor=north west][inner sep=0.75pt]    {$y_{j_{1}}$};
\draw (253.04,77) node [anchor=north west][inner sep=0.75pt]    {$y_{J}$};
\draw (220.65,240.43) node [anchor=north west][inner sep=0.75pt]    {$\delta ( y_{1} ,y_{j}{}_{_{1}}) \ =\ \delta ( y_{J} ,y_{j_{1}})$};
\end{tikzpicture}
        \caption{In view of (\ref{eqn: mindelta}), $\delta(y_J,y_1) \ge \delta(y_0,y_1).$ Since $\{y_0,y_1,y_{j_1}\}$ forms a left comb, $\delta(y_0,y_1) > \delta(y_{1}, y_{j_1}).$ Using (\ref{eqn: mindelta}), we have $\delta(y_J, y_{j_1}) = \delta(y_1,y_{j_1}).$}
        \label{fig:leftcombinsert1}
    \end{figure}
    Together with (\ref{eqn: leftcombproj2}), this implies that (\ref{eqn: leftcombproj}) holds. 
    
     Since $Y$ is a left comb, the set $\pi(Y\setminus\set{y_0})$ is equal to 
     \begin{align}
    \label{eqn: deltaleftcomb}
         \delta_2 := \delta(y_1, y_2) > \delta_3 := \delta(y_2,y_3) >\cdots > \delta_\l := \delta(y_{\l -1}, y_\l).
     \end{align}
     We will now show that 
     $$\pi(Y\setminus\{y_0\})^{(k-1)} = \{\{\delta_j: j\in J\}: J\in \set{2,\ldots,\ell}^{(k-1)}\} \subseteq [N]^{(k-1)}$$
    is monochromatic in $c$. Since $Y^J$ forms a left comb for every $J=\{j_1<\dots< j_{k-1}\}\in \set{2,\ldots,\ell}^{(k-1)}$, in view of \ref{it: comb2L} in Remark \ref{fact: combstructure}, together with the definition of the $\delta_i$ in (\ref{eqn: deltaleftcomb}) above,
    $$\pi\left(Y^J\right)= \{\delta(y_{j_{i-1}},y_{j_i})= \delta(y_{j_i-1},y_{j_i}) = \delta_{j_i}: i\in [k-1]\}=\set{\delta_j:j\in J}$$
    To summarize, we showed that for every $J\in \{2,\dots, \ell\}^{(k-1)}$, we have that $Y^J$ is a left comb and $\pi(Y^J) = \set{\delta_j:j\in J}$. Since by assumption,  $\chi_c\left(Y^J\right)=\alpha$ for every $J\in \{2,\dots, \ell\}^{(k-1)}$, and in view of how $\chi_c$ is defined in (\ref{eqn: coloring3}) and (\ref{eqn: coloring4}),  we must have that for every $J\in \{2,\dots, \ell\}^{(k-1)}$,
    \begin{align*}
        c\left(\pi\left(Y^J\right)\right) = c\left(\set{\delta_j:j\in J}\right)= \begin{cases}
            \alpha &\quad\text{if $k=3$,}\\
            3-\alpha &\quad \text{if $k\ge 4$.}
        \end{cases}
    \end{align*}
\end{proof}
\begin{proposition}[Right Comb Projection]
\label{prop: rightprojection}
    Let $N\geq \ell\geq k \geq 4$, and fix a $(k-1)$-subset $I\in [\ell]^{(k-1)}$ with $1\in I$. For a coloring $c: [N]^{(k-1)}\to \ZZ_4$, let $\chi_c:[2^N]^{(k)}\to \ZZ_4$ be as defined in (\ref{eqn: coloring4}). Let  $Y = \{y_0 < y_1 < \cdots < y_\ell\}$ be a subset of the leaves of $T = T(N)$ satisfying: 
       \begin{enumerate}[label = (\alph*)]
        \item\label{it: righta} $Y$ forms a right comb,
        \item\label{it: rightb}  for every $J\in \{2,\dots, \ell\}^{(k-1)}$, there exists a vertex $y_{J}$ such that $y_0 \leq y_{J} \leq y_1.$\footnote{Similar to Proposition \ref{prop: leftcomb} \ref{it: leftlemb}, here we allow distinct $J$s to map to the same $y_{J}$.}
        \item\label{it: rightc} There exists a vertex $y_0\le y_{I}<y_1$ satisfying $a(y_{I}, y_1) = u_Y = a(y_0,y_1).$
    \end{enumerate}
    Then the following holds:
    \begin{itemize}[label = ($\star$)]
        \item\label{it: rightlemmastar}  If there exists $\alpha\in \ZZ_4$ such that 
    $$ \chi_c({\{y_{J}\}\cup \{y_i: i\in J\}})= \alpha \text{ for every }  J \in \{2,\dots, \ell\}^{(k-1)}\cup \{I\},$$
    \end{itemize}
    then $c^{-1}(\alpha)\subseteq [N]^{(k-1)}$ contains a monochromatic copy of $(F,<)\in \cF^{(k-1)}_{I'}(\ell-1)$ with distinguished set of vertices $\pi(Y)$ and $I' = I\setminus \{\max I\}$. 
    \end{proposition}
    \vspace{-0.3cm}
\begin{figure}[h!]
    \centering

\tikzset{every picture/.style={line width=0.75pt}} 

\begin{tikzpicture}[x=0.75pt,y=0.75pt,yscale=-1.0,xscale=1]

\draw  [dash pattern={on 4.5pt off 4.5pt}]  (193.99,246.89) -- (388.57,246.89) ;
\draw    (192.33,77.57) -- (242.87,169.51) ;
\draw    (85.13,77.31) -- (193.99,246.89) ;
\draw    (300.59,77.31) -- (191.72,251.19) ;
\draw    (170.2,76.98) -- (231.64,187.89) ;
\draw  [fill={rgb, 255:red, 208; green, 2; blue, 27 }  ,fill opacity=1 ] (189.87,246.89) .. controls (189.87,244.51) and (191.72,242.58) .. (193.99,242.58) .. controls (196.27,242.58) and (198.12,244.51) .. (198.12,246.89) .. controls (198.12,249.27) and (196.27,251.19) .. (193.99,251.19) .. controls (191.72,251.19) and (189.87,249.27) .. (189.87,246.89) -- cycle ;
\draw  [fill={rgb, 255:red, 208; green, 2; blue, 27 }  ,fill opacity=1 ] (81,77.31) .. controls (81,74.93) and (82.85,73) .. (85.13,73) .. controls (87.4,73) and (89.25,74.93) .. (89.25,77.31) .. controls (89.25,79.69) and (87.4,81.62) .. (85.13,81.62) .. controls (82.85,81.62) and (81,79.69) .. (81,77.31) -- cycle ;
\draw  [fill={rgb, 255:red, 208; green, 2; blue, 27 }  ,fill opacity=1 ] (166.08,76.98) .. controls (166.08,74.6) and (167.93,72.67) .. (170.2,72.67) .. controls (172.48,72.67) and (174.33,74.6) .. (174.33,76.98) .. controls (174.33,79.36) and (172.48,81.29) .. (170.2,81.29) .. controls (167.93,81.29) and (166.08,79.36) .. (166.08,76.98) -- cycle ;
\draw  [fill={rgb, 255:red, 208; green, 2; blue, 27 }  ,fill opacity=1 ] (188.21,77.57) .. controls (188.21,75.19) and (190.05,73.26) .. (192.33,73.26) .. controls (194.61,73.26) and (196.45,75.19) .. (196.45,77.57) .. controls (196.45,79.95) and (194.61,81.87) .. (192.33,81.87) .. controls (190.05,81.87) and (188.21,79.95) .. (188.21,77.57) -- cycle ;
\draw  [fill={rgb, 255:red, 208; green, 2; blue, 52 }  ,fill opacity=1 ] (296.46,77.31) .. controls (296.46,74.93) and (298.31,73) .. (300.59,73) .. controls (302.86,73) and (304.71,74.93) .. (304.71,77.31) .. controls (304.71,79.69) and (302.86,81.62) .. (300.59,81.62) .. controls (298.31,81.62) and (296.46,79.69) .. (296.46,77.31) -- cycle ;
\draw    (243.48,77.03) -- (269.3,127) ;
\draw  [fill={rgb, 255:red, 208; green, 2; blue, 52 }  ,fill opacity=1 ] (239.35,77.03) .. controls (239.35,74.65) and (241.2,72.73) .. (243.48,72.73) .. controls (245.75,72.73) and (247.6,74.65) .. (247.6,77.03) .. controls (247.6,79.41) and (245.75,81.34) .. (243.48,81.34) .. controls (241.2,81.34) and (239.35,79.41) .. (239.35,77.03) -- cycle ;
\draw    (214.99,77.03) -- (254.7,150.45) ;
\draw  [fill={rgb, 255:red, 208; green, 2; blue, 27 }  ,fill opacity=1 ] (210.87,77.03) .. controls (210.87,74.65) and (212.71,72.73) .. (214.99,72.73) .. controls (217.27,72.73) and (219.11,74.65) .. (219.11,77.03) .. controls (219.11,79.41) and (217.27,81.34) .. (214.99,81.34) .. controls (212.71,81.34) and (210.87,79.41) .. (210.87,77.03) -- cycle ;
\draw    (108.23,77.33) .. controls (93.11,108.59) and (142.84,93.24) .. (132.32,151.09) ;
\draw  [fill={rgb, 255:red, 208; green, 2; blue, 52 }  ,fill opacity=1 ] (104.1,77.33) .. controls (104.1,74.95) and (105.95,73.02) .. (108.23,73.02) .. controls (110.51,73.02) and (112.35,74.95) .. (112.35,77.33) .. controls (112.35,79.71) and (110.51,81.64) .. (108.23,81.64) .. controls (105.95,81.64) and (104.1,79.71) .. (104.1,77.33) -- cycle ;
\draw    (387.91,69.11) -- (387.91,289) ;
\draw  [dash pattern={on 4.5pt off 4.5pt}]  (269.3,127) -- (388.57,127) ;
\draw  [dash pattern={on 4.5pt off 4.5pt}]  (254.7,150.45) -- (388.57,150.45) ;
\draw  [dash pattern={on 4.5pt off 4.5pt}]  (242.87,169.51) -- (388,169.51) ;
\draw  [dash pattern={on 4.5pt off 4.5pt}]  (231.64,187.89) -- (388.29,187.89) ;
\draw  [fill={rgb, 255:red, 65; green, 117; blue, 5 }  ,fill opacity=1 ] (385.85,127) .. controls (385.85,125.81) and (386.78,124.84) .. (387.91,124.84) .. controls (389.05,124.84) and (389.98,125.81) .. (389.98,127) .. controls (389.98,128.19) and (389.05,129.15) .. (387.91,129.15) .. controls (386.78,129.15) and (385.85,128.19) .. (385.85,127) -- cycle ;
\draw  [fill={rgb, 255:red, 65; green, 117; blue, 5 }  ,fill opacity=1 ] (385.85,150.45) .. controls (385.85,149.26) and (386.78,148.29) .. (387.91,148.29) .. controls (389.05,148.29) and (389.98,149.26) .. (389.98,150.45) .. controls (389.98,151.64) and (389.05,152.6) .. (387.91,152.6) .. controls (386.78,152.6) and (385.85,151.64) .. (385.85,150.45) -- cycle ;
\draw  [fill={rgb, 255:red, 65; green, 117; blue, 5 }  ,fill opacity=1 ] (385.85,169.87) .. controls (385.85,168.69) and (386.78,167.72) .. (387.91,167.72) .. controls (389.05,167.72) and (389.98,168.69) .. (389.98,169.87) .. controls (389.98,171.06) and (389.05,172.03) .. (387.91,172.03) .. controls (386.78,172.03) and (385.85,171.06) .. (385.85,169.87) -- cycle ;
\draw  [fill={rgb, 255:red, 65; green, 117; blue, 5 }  ,fill opacity=1 ] (386.23,187.89) .. controls (386.23,186.7) and (387.15,185.73) .. (388.29,185.73) .. controls (389.43,185.73) and (390.35,186.7) .. (390.35,187.89) .. controls (390.35,189.08) and (389.43,190.04) .. (388.29,190.04) .. controls (387.15,190.04) and (386.23,189.08) .. (386.23,187.89) -- cycle ;
\draw  [fill={rgb, 255:red, 65; green, 117; blue, 5 }  ,fill opacity=1 ] (386.14,247.3) .. controls (386.14,246.12) and (387.06,245.15) .. (388.2,245.15) .. controls (389.34,245.15) and (390.26,246.12) .. (390.26,247.3) .. controls (390.26,248.49) and (389.34,249.46) .. (388.2,249.46) .. controls (387.06,249.46) and (386.14,248.49) .. (386.14,247.3) -- cycle ;
\draw  [fill={rgb, 255:red, 208; green, 2; blue, 52 }  ,fill opacity=0.13 ][dash pattern={on 0.84pt off 2.51pt}] (101,82) .. controls (112,96) and (114,66) .. (132,65) .. controls (150,64) and (159,87) .. (183.06,86.93) .. controls (207.11,86.86) and (219.06,47.93) .. (244.06,46.93) .. controls (269.06,45.93) and (304.11,100.86) .. (309.06,83.93) .. controls (314,67) and (270.06,29.93) .. (241.06,28.93) .. controls (212.06,27.93) and (203.06,68.93) .. (181.06,67.93) .. controls (159.06,66.93) and (154,43) .. (129,44) .. controls (104,45) and (90,68) .. (101,82) -- cycle ;
\draw  [dash pattern={on 0.84pt off 2.51pt}] (70,48) -- (327,48) -- (327,88) -- (70,88) -- cycle ;
\draw  [dash pattern={on 0.84pt off 2.51pt}] (353.85,97.07) -- (427,97.07) -- (427,274.4) -- (353.85,274.4) -- cycle ;
\draw  [fill={rgb, 255:red, 241; green, 5; blue, 5 }  ,fill opacity=0.12 ][dash pattern={on 0.84pt off 2.51pt}] (390,160) .. controls (410,160) and (397,174) .. (396,204) .. controls (395,234) and (413,260) .. (389,257) .. controls (365,254) and (374.66,232.89) .. (373,213) .. controls (371.34,193.11) and (370,160) .. (390,160) -- cycle ;

\draw (192.98,252.83) node [anchor=north west][inner sep=0.75pt]    {$u_{Y}$};
\draw (77.45,51) node [anchor=north west][inner sep=0.75pt]    {$y_{0}$};
\draw (163.27,51) node [anchor=north west][inner sep=0.75pt]    {$y_{1}$};
\draw (185.09,51) node [anchor=north west][inner sep=0.75pt]    {$y_{2}$};
\draw (292.95,51) node [anchor=north west][inner sep=0.75pt]    {$y_{\ell }$};
\draw (219.97,72.07) node [anchor=north west][inner sep=0.75pt] [xscale=0.8,yscale=0.8]  {$\cdots $};
\draw (208.01,51) node [anchor=north west][inner sep=0.75pt]    {$y_{3}$};
\draw (105.41,51) node [anchor=north west][inner sep=0.75pt]    {$y_{I}$};
\draw (238.03,51) node [anchor=north west][inner sep=0.75pt]    {$y_{\ell -1}$};
\draw (400.7,236.06) node [anchor=north west][inner sep=0.75pt]    {$\delta _{0}$};
\draw (399.04,178.89) node [anchor=north west][inner sep=0.75pt]    {$\delta _{1}$};
\draw (399.04,159.54) node [anchor=north west][inner sep=0.75pt]    {$\delta _{2}$};
\draw (397.73,139.42) node [anchor=north west][inner sep=0.75pt]    {$\delta _{3}$};
\draw (394.84,115.75) node [anchor=north west][inner sep=0.75pt]    {$\delta _{\ell -1}$};
\draw (342,33.4) node [anchor=north west][inner sep=0.75pt]  [font=\footnotesize,color={rgb, 255:red, 208; green, 2; blue, 52 }  ,opacity=1 ]  {$\textcolor[rgb]{0.82,0.01,0.2}{I}\textcolor[rgb]{0.82,0.01,0.2}{}\textcolor[rgb]{0.82,0.01,0.2}{\ =\ }\textcolor[rgb]{0.82,0.01,0.2}{\{}\textcolor[rgb]{0.82,0.01,0.2}{1,2,\ell }\textcolor[rgb]{0.82,0.01,0.2}{\}}$};
\draw (147,28.4) node [anchor=north west][inner sep=0.75pt]    {$Y$};
\draw (433,168.4) node [anchor=north west][inner sep=0.75pt]    {$\pi ( Y)\subseteq [N]$};

\end{tikzpicture}
    \caption{Right Comb Projection. An example where $k = 4$ and the set $I = \{1,2,\ell\}$. The red edge is an example of the set $Y^I = \{y_{I}\}\cup \{y_1,y_2,y_\ell\}.$}
    \label{fig:RightProj}
\end{figure}
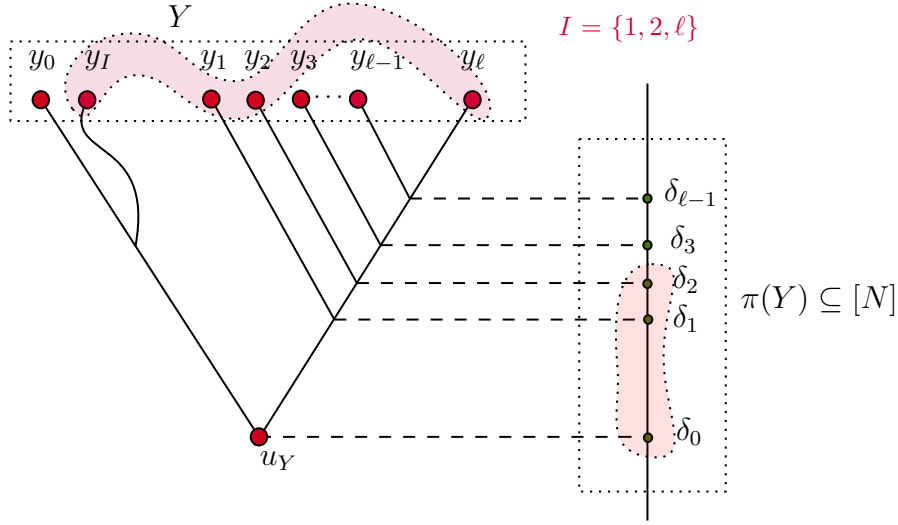
\begin{proof}
 We first describe the hypergraph $(F,<)$, which is promised in the conclusion of the proposition. 
  Fix a right comb $Y= \{y_0 < y_ 1< \cdots < y_\ell\}$ (as shown in Figure \ref{fig:RightProj}) and for   $0\le i \le \ell-1$, we let:  
    $$u_i=a(y_i,y_{i+1}), \quad \delta_i=\delta(y_i,y_{i+1})=\pi(u_i).$$
    Note that, as shown in Figure \ref{fig:RightProj},
    \begin{align}
    \label{eqn: deltasrightproj}
        \pi(Y)=\set{\delta_0<\delta_1<\ldots< \delta_{\ell-1}}.
    \end{align}
    For every $(k-2)$-set $J'\in \{2,\dots,\ell-1\}^{(k-2)}$ consider the $(k-1)$-set $J'\cup \{\ell\}$ and let: 
    $$\delta_{J'} = \delta(y_{J'\cup\{\ell\}}, y_\ell) \text{ and } \delta_{I'} = \delta_0 = \delta(y_I,y_\ell),$$
    where $I' = I\setminus \{\max I\}$. Note that since $y_0\le y_{J'\cup\{\ell\}}\le y_1$ and $Y$ is a right comb, in view of (\ref{eqn: mindelta}), we have for every $J'\in \{2,\dots,\ell-1\}^{(k-2)}$, 
    $$\delta_{J'} = \delta(y_{J'\cup \{\ell\}},y_\ell) = \min \{\delta(y_{J'\cup \{\ell\}}, y_1), \delta(y_1,y_\ell) \} \le \delta(y_1,y_\ell) =\delta(y_1,y_2) = \delta_1.$$
    Further, since $y_0\le y_{J'\cup \{\ell\}}\le y_\ell$, by (\ref{eqn: mindelta}) we have $\delta(y_0, y_\ell) \ge \delta_{J'}$. Consequently,
    \begin{align}
        \label{eqn: deltaJ'}
        \delta_0\le \delta_{J'} \le \delta_1.
    \end{align}
    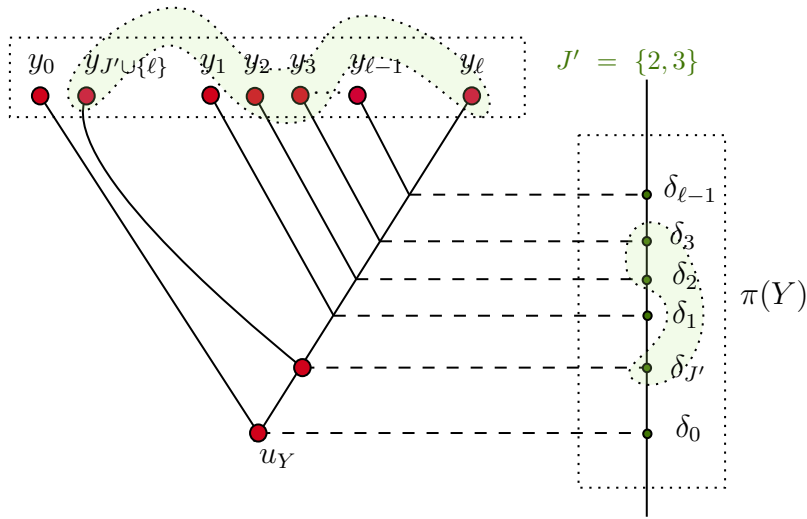
\begin{figure}[h]
        \centering
\tikzset{every picture/.style={line width=0.75pt}} 
\begin{tikzpicture}[x=0.75pt,y=0.75pt,yscale=-1,xscale=1]

\draw  [dash pattern={on 4.5pt off 4.5pt}]  (193.99,246.89) -- (388.57,246.89) ;
\draw    (192.33,77.57) -- (242.87,169.51) ;
\draw    (85.13,77.31) -- (193.99,246.89) ;
\draw    (300.59,77.31) -- (191.72,251.19) ;
\draw    (170.2,76.98) -- (231.64,187.89) ;
\draw  [fill={rgb, 255:red, 208; green, 2; blue, 27 }  ,fill opacity=1 ] (189.87,246.89) .. controls (189.87,244.51) and (191.72,242.58) .. (193.99,242.58) .. controls (196.27,242.58) and (198.12,244.51) .. (198.12,246.89) .. controls (198.12,249.27) and (196.27,251.19) .. (193.99,251.19) .. controls (191.72,251.19) and (189.87,249.27) .. (189.87,246.89) -- cycle ;
\draw  [fill={rgb, 255:red, 208; green, 2; blue, 27 }  ,fill opacity=1 ] (81,77.31) .. controls (81,74.93) and (82.85,73) .. (85.13,73) .. controls (87.4,73) and (89.25,74.93) .. (89.25,77.31) .. controls (89.25,79.69) and (87.4,81.62) .. (85.13,81.62) .. controls (82.85,81.62) and (81,79.69) .. (81,77.31) -- cycle ;
\draw  [fill={rgb, 255:red, 208; green, 2; blue, 27 }  ,fill opacity=1 ] (166.08,76.98) .. controls (166.08,74.6) and (167.93,72.67) .. (170.2,72.67) .. controls (172.48,72.67) and (174.33,74.6) .. (174.33,76.98) .. controls (174.33,79.36) and (172.48,81.29) .. (170.2,81.29) .. controls (167.93,81.29) and (166.08,79.36) .. (166.08,76.98) -- cycle ;
\draw  [fill={rgb, 255:red, 208; green, 2; blue, 27 }  ,fill opacity=1 ] (188.21,77.57) .. controls (188.21,75.19) and (190.05,73.26) .. (192.33,73.26) .. controls (194.61,73.26) and (196.45,75.19) .. (196.45,77.57) .. controls (196.45,79.95) and (194.61,81.87) .. (192.33,81.87) .. controls (190.05,81.87) and (188.21,79.95) .. (188.21,77.57) -- cycle ;
\draw  [fill={rgb, 255:red, 208; green, 2; blue, 52 }  ,fill opacity=1 ] (296.46,77.31) .. controls (296.46,74.93) and (298.31,73) .. (300.59,73) .. controls (302.86,73) and (304.71,74.93) .. (304.71,77.31) .. controls (304.71,79.69) and (302.86,81.62) .. (300.59,81.62) .. controls (298.31,81.62) and (296.46,79.69) .. (296.46,77.31) -- cycle ;
\draw    (243.48,77.03) -- (269.3,127) ;
\draw  [fill={rgb, 255:red, 208; green, 2; blue, 52 }  ,fill opacity=1 ] (239.35,77.03) .. controls (239.35,74.65) and (241.2,72.73) .. (243.48,72.73) .. controls (245.75,72.73) and (247.6,74.65) .. (247.6,77.03) .. controls (247.6,79.41) and (245.75,81.34) .. (243.48,81.34) .. controls (241.2,81.34) and (239.35,79.41) .. (239.35,77.03) -- cycle ;
\draw    (214.99,77.03) -- (254.7,150.45) ;
\draw  [fill={rgb, 255:red, 208; green, 2; blue, 27 }  ,fill opacity=1 ] (210.87,77.03) .. controls (210.87,74.65) and (212.71,72.73) .. (214.99,72.73) .. controls (217.27,72.73) and (219.11,74.65) .. (219.11,77.03) .. controls (219.11,79.41) and (217.27,81.34) .. (214.99,81.34) .. controls (212.71,81.34) and (210.87,79.41) .. (210.87,77.03) -- cycle ;
\draw    (108.23,77.33) .. controls (93.11,108.59) and (178,183) .. (216,214) ;
\draw  [fill={rgb, 255:red, 208; green, 2; blue, 52 }  ,fill opacity=1 ] (104.1,77.33) .. controls (104.1,74.95) and (105.95,73.02) .. (108.23,73.02) .. controls (110.51,73.02) and (112.35,74.95) .. (112.35,77.33) .. controls (112.35,79.71) and (110.51,81.64) .. (108.23,81.64) .. controls (105.95,81.64) and (104.1,79.71) .. (104.1,77.33) -- cycle ;
\draw    (387.91,69.11) -- (387.91,289) ;
\draw  [dash pattern={on 4.5pt off 4.5pt}]  (269.3,127) -- (388.57,127) ;
\draw  [dash pattern={on 4.5pt off 4.5pt}]  (254.7,150.45) -- (388.57,150.45) ;
\draw  [dash pattern={on 4.5pt off 4.5pt}]  (242.87,169.51) -- (388,169.51) ;
\draw  [dash pattern={on 4.5pt off 4.5pt}]  (231.64,187.89) -- (388.29,187.89) ;
\draw  [dash pattern={on 4.5pt off 4.5pt}]  (216,214) -- (388,214) ;
\draw  [fill={rgb, 255:red, 65; green, 117; blue, 5 }  ,fill opacity=1 ] (385.85,127) .. controls (385.85,125.81) and (386.78,124.84) .. (387.91,124.84) .. controls (389.05,124.84) and (389.98,125.81) .. (389.98,127) .. controls (389.98,128.19) and (389.05,129.15) .. (387.91,129.15) .. controls (386.78,129.15) and (385.85,128.19) .. (385.85,127) -- cycle ;
\draw  [fill={rgb, 255:red, 65; green, 117; blue, 5 }  ,fill opacity=1 ] (385.85,150.45) .. controls (385.85,149.26) and (386.78,148.29) .. (387.91,148.29) .. controls (389.05,148.29) and (389.98,149.26) .. (389.98,150.45) .. controls (389.98,151.64) and (389.05,152.6) .. (387.91,152.6) .. controls (386.78,152.6) and (385.85,151.64) .. (385.85,150.45) -- cycle ;
\draw  [fill={rgb, 255:red, 65; green, 117; blue, 5 }  ,fill opacity=1 ] (385.85,169.87) .. controls (385.85,168.69) and (386.78,167.72) .. (387.91,167.72) .. controls (389.05,167.72) and (389.98,168.69) .. (389.98,169.87) .. controls (389.98,171.06) and (389.05,172.03) .. (387.91,172.03) .. controls (386.78,172.03) and (385.85,171.06) .. (385.85,169.87) -- cycle ;
\draw  [fill={rgb, 255:red, 65; green, 117; blue, 5 }  ,fill opacity=1 ] (386.23,187.89) .. controls (386.23,186.7) and (387.15,185.73) .. (388.29,185.73) .. controls (389.43,185.73) and (390.35,186.7) .. (390.35,187.89) .. controls (390.35,189.08) and (389.43,190.04) .. (388.29,190.04) .. controls (387.15,190.04) and (386.23,189.08) .. (386.23,187.89) -- cycle ;
\draw  [fill={rgb, 255:red, 65; green, 117; blue, 5 }  ,fill opacity=1 ] (386.14,247.3) .. controls (386.14,246.12) and (387.06,245.15) .. (388.2,245.15) .. controls (389.34,245.15) and (390.26,246.12) .. (390.26,247.3) .. controls (390.26,248.49) and (389.34,249.46) .. (388.2,249.46) .. controls (387.06,249.46) and (386.14,248.49) .. (386.14,247.3) -- cycle ;
\draw  [fill={rgb, 255:red, 184; green, 233; blue, 134 }  ,fill opacity=0.18 ][dash pattern={on 0.84pt off 2.51pt}] (101,82) .. controls (112,96) and (131,53) .. (149,52) .. controls (167,51) and (178,91) .. (208,91) .. controls (238,91) and (227,59) .. (252,58) .. controls (277,57) and (304.11,100.86) .. (309.06,83.93) .. controls (314,67) and (273,37) .. (248,37) .. controls (223,37) and (229,62) .. (205,64) .. controls (181,66) and (176,32) .. (151,33) .. controls (126,34) and (90,68) .. (101,82) -- cycle ;
\draw  [dash pattern={on 0.84pt off 2.51pt}] (70,48) -- (327,48) -- (327,88) -- (70,88) -- cycle ;
\draw  [dash pattern={on 0.84pt off 2.51pt}] (353.85,97.07) -- (427,97.07) -- (427,274.4) -- (353.85,274.4) -- cycle ;
\draw  [fill={rgb, 255:red, 208; green, 2; blue, 27 }  ,fill opacity=1 ] (211.88,214) .. controls (211.88,211.62) and (213.72,209.69) .. (216,209.69) .. controls (218.28,209.69) and (220.12,211.62) .. (220.12,214) .. controls (220.12,216.38) and (218.28,218.31) .. (216,218.31) .. controls (213.72,218.31) and (211.88,216.38) .. (211.88,214) -- cycle ;

\draw  [fill={rgb, 255:red, 65; green, 117; blue, 5 }  ,fill opacity=1 ] (385.94,214.15) .. controls (385.94,212.96) and (386.86,212) .. (388,212) .. controls (389.14,212) and (390.06,212.96) .. (390.06,214.15) .. controls (390.06,215.34) and (389.14,216.31) .. (388,216.31) .. controls (386.86,216.31) and (385.94,215.34) .. (385.94,214.15) -- cycle ;
\draw  [fill={rgb, 255:red, 184; green, 233; blue, 134 }  ,fill opacity=0.18 ][dash pattern={on 0.84pt off 2.51pt}] (389.11,140.96) .. controls (400.11,139.96) and (408.11,153.96) .. (413.11,168.96) .. controls (418.11,183.96) and (417.11,198.96) .. (409.11,210.96) .. controls (401.11,222.96) and (386.21,225.92) .. (380.11,218.96) .. controls (374,212) and (399.11,202.96) .. (398.11,186.96) .. controls (397.11,170.96) and (380.11,176.96) .. (378.11,165.96) .. controls (376.11,154.96) and (378.11,141.96) .. (389.11,140.96) -- cycle ;

\draw (192.98,253.83) node [anchor=north west][inner sep=0.75pt]    {$u_{Y}$};
\draw (77.45,55) node [anchor=north west][inner sep=0.75pt]    {$y_{0}$};
\draw (163.27,55) node [anchor=north west][inner sep=0.75pt]    {$y_{1}$};
\draw (185.09,55) node [anchor=north west][inner sep=0.75pt]    {$y_{2}$};
\draw (292.95,55) node [anchor=north west][inner sep=0.75pt]    {$y_{\ell }$};
\draw (219.97,72.07) node [anchor=north west][inner sep=0.75pt]  [xscale=0.8,yscale=0.8]  {$\cdots $};
\draw (208.01,55) node [anchor=north west][inner sep=0.75pt]    {$y_{3}$};
\draw (105.41,55) node [anchor=north west][inner sep=0.75pt]    {$y_{J'\cup \{\ell \}}$};
\draw (238.03,55) node [anchor=north west][inner sep=0.75pt]    {$y_{\ell -1}$};
\draw (400.7,236.06) node [anchor=north west][inner sep=0.75pt]    {$\delta _{0}$};
\draw (399.04,178.89) node [anchor=north west][inner sep=0.75pt]    {$\delta _{1}$};
\draw (399.04,159.54) node [anchor=north west][inner sep=0.75pt]    {$\delta _{2}$};
\draw (397.73,139.42) node [anchor=north west][inner sep=0.75pt]    {$\delta _{3}$};
\draw (394.84,115.75) node [anchor=north west][inner sep=0.75pt]    {$\delta _{\ell -1}$};
\draw (341,52.4) node [anchor=north west][inner sep=0.75pt]  [font=\footnotesize,color={rgb, 255:red, 65; green, 117; blue, 5 }  ,opacity=1 ]  {$\textcolor[rgb]{0.25,0.46,0.02}{J'\ =\ \{2,3\}}$};
\draw (433,168.4) node [anchor=north west][inner sep=0.75pt]    {$\pi ( Y)$};
\draw (398.04,205.89) node [anchor=north west][inner sep=0.75pt]    {$\delta _{J'}$};
\end{tikzpicture}
        \caption{An example of an edge $\{\delta_{J'}\}\cup \{\delta_j:j\in J'\}$ of $F$ where $k = 4$ and $J' = \{2,3\}$. Note that $\{\delta_{J'}\}\cup \{\delta_j:j\in J'\}$ is the projection of  $Y^J = \{y_J\}\cup \{y_j:j\in J\}$ where $J= J'\cup \{\ell\}.$}
        \label{fig: rightprojdeltaJ'}
    \end{figure}
    \\
    Let $(F,<)$ be the hypergraph in the projection $[N]^{(k-1)}$ with vertex set: 
    \begin{align*}
        V(F)= \{\delta_0\}\cup \{\delta_{J'}:J' \in\{2,\dots,\ell-1\}^{(k-2)}\}\cup \{\delta_1,\dots,\delta_{\ell-1}\}
    \end{align*}
The edge set of $F$ consists of the special edge $\{\delta_0\} \cup \{\delta_i:i\in I'\}$, together with edges $\{\delta_{J'}\}\cup \{\delta_j:j\in J'\},$ for each $J'\in \{2,\dots,\ell\}^{(k-1)}$ (see Fig. \ref{fig:RightProj} and \ref{fig: rightprojdeltaJ'}). 
        In other words, 
        $$ F=\{\{\delta_0\}\cup \{\delta_j:j\in I'\}\}\cup \{\{\delta_{J'}\}\cup \{\delta_j:j\in J'\}:J'\in \{2,\dots,\ell\}^{(k-1)}\}.$$
Note that in view of (\ref{eqn: deltaJ'}), we have $(F,<) \in \cF_{I'}^{(k-1)}(\ell-1).$

   It remains to show that $c^{-1}(\alpha)$ contains $(F,<)$, i.e.,  $c\left(\set{\delta_{J'}}\cup\set{\delta_j:j\in J'}\right)=\alpha$ for all $J'\in\set{2,\ldots,\ell-1}^{(k-2)}\cup \{I'\}$. Fix $J'\in\set{2,\ldots,\ell-1}^{(k-2)}\cup \{I'\}$ and set $J' = \set{j_1<\cdots<j_{k-2}}$. Let 
    $$\text{$J = J'\cup \{\ell\}$ if $J'\ne I'$ and let $J = I$ if $J' = I'$. Let $j_{k-1}:=\max J$ in each case.}$$ 
    We first show that the set $Y^{J}=\{y_{J}\}\cup \{y_j: j\in J\}$ is a right comb, or equivalently,
    \begin{align}
        \label{eqn: rightcombeqn1}
        \delta(y_J, y_{j_1}) < \delta(y_{j_1}, y_{j_2}) < \cdots < \delta(y_{j_{k-2}}, y_{j_{k-1}}). 
    \end{align}    
    Since  
    $\{y_j:j\in J\}$ is a subset of the right comb $Y$, by Remark \ref{fact: combstructure} \ref{it: comb2R}, we have that: 
    $$\delta_{j_1} = \delta(y_{j_1}, y_{j_2}) < \cdots < \delta_{j_{k-2}}= \delta(y_{j_{k-2}}, y_{j_{k-1}})$$
    holds. 
\begin{figure}[t]
\centering
\begin{minipage}{0.48\textwidth}
  \centering
\tikzset{every picture/.style={line width=0.75pt}} 
\begin{tikzpicture}[x=0.75pt,y=0.75pt,yscale=-0.6,xscale=0.6]

\draw  [dash pattern={on 4.5pt off 4.5pt}]  (254.36,298.7) -- (519.07,298.7) ;
\draw    (252.1,68.35) -- (320.85,193.44) ;
\draw    (106.25,68) -- (254.36,298.7) ;
\draw    (399.37,68) -- (251.27,304.56) ;
\draw    (221.99,67.55) -- (305.58,218.43) ;
\draw  [fill={rgb, 255:red, 208; green, 2; blue, 27 }  ,fill opacity=1 ] (248.75,298.7) .. controls (248.75,295.46) and (251.26,292.84) .. (254.36,292.84) .. controls (257.45,292.84) and (259.97,295.46) .. (259.97,298.7) .. controls (259.97,301.93) and (257.45,304.56) .. (254.36,304.56) .. controls (251.26,304.56) and (248.75,301.93) .. (248.75,298.7) -- cycle ;
\draw  [fill={rgb, 255:red, 208; green, 2; blue, 27 }  ,fill opacity=1 ] (100.64,68) .. controls (100.64,64.77) and (103.16,62.14) .. (106.25,62.14) .. controls (109.35,62.14) and (111.86,64.77) .. (111.86,68) .. controls (111.86,71.24) and (109.35,73.86) .. (106.25,73.86) .. controls (103.16,73.86) and (100.64,71.24) .. (100.64,68) -- cycle ;
\draw  [fill={rgb, 255:red, 208; green, 2; blue, 27 }  ,fill opacity=1 ] (216.39,67.55) .. controls (216.39,64.32) and (218.9,61.69) .. (221.99,61.69) .. controls (225.09,61.69) and (227.6,64.32) .. (227.6,67.55) .. controls (227.6,70.79) and (225.09,73.41) .. (221.99,73.41) .. controls (218.9,73.41) and (216.39,70.79) .. (216.39,67.55) -- cycle ;
\draw  [fill={rgb, 255:red, 208; green, 2; blue, 27 }  ,fill opacity=1 ] (246.49,68.35) .. controls (246.49,65.12) and (249,62.49) .. (252.1,62.49) .. controls (255.19,62.49) and (257.71,65.12) .. (257.71,68.35) .. controls (257.71,71.59) and (255.19,74.21) .. (252.1,74.21) .. controls (249,74.21) and (246.49,71.59) .. (246.49,68.35) -- cycle ;
\draw  [fill={rgb, 255:red, 208; green, 2; blue, 52 }  ,fill opacity=1 ] (393.76,68) .. controls (393.76,64.77) and (396.27,62.14) .. (399.37,62.14) .. controls (402.47,62.14) and (404.98,64.77) .. (404.98,68) .. controls (404.98,71.24) and (402.47,73.86) .. (399.37,73.86) .. controls (396.27,73.86) and (393.76,71.24) .. (393.76,68) -- cycle ;
\draw    (321.68,67.63) -- (356.8,135.6) ;
\draw  [fill={rgb, 255:red, 208; green, 2; blue, 52 }  ,fill opacity=1 ] (316.07,67.63) .. controls (316.07,64.39) and (318.58,61.77) .. (321.68,61.77) .. controls (324.77,61.77) and (327.29,64.39) .. (327.29,67.63) .. controls (327.29,70.86) and (324.77,73.48) .. (321.68,73.48) .. controls (318.58,73.48) and (316.07,70.86) .. (316.07,67.63) -- cycle ;
\draw    (282.92,67.63) -- (336.95,167.5) ;
\draw  [fill={rgb, 255:red, 208; green, 2; blue, 27 }  ,fill opacity=1 ] (277.31,67.63) .. controls (277.31,64.39) and (279.82,61.77) .. (282.92,61.77) .. controls (286.02,61.77) and (288.53,64.39) .. (288.53,67.63) .. controls (288.53,70.86) and (286.02,73.48) .. (282.92,73.48) .. controls (279.82,73.48) and (277.31,70.86) .. (277.31,67.63) -- cycle ;
\draw    (137.68,68.03) .. controls (117.11,110.55) and (184.77,89.67) .. (170.46,168.37) ;
\draw  [fill={rgb, 255:red, 208; green, 2; blue, 52 }  ,fill opacity=1 ] (132.07,68.03) .. controls (132.07,64.79) and (134.58,62.17) .. (137.68,62.17) .. controls (140.78,62.17) and (143.29,64.79) .. (143.29,68.03) .. controls (143.29,71.26) and (140.78,73.89) .. (137.68,73.89) .. controls (134.58,73.89) and (132.07,71.26) .. (132.07,68.03) -- cycle ;
\draw    (518.17,56.84) -- (518.17,321.58) ;
\draw  [dash pattern={on 4.5pt off 4.5pt}]  (356.8,135.6) -- (519.07,135.6) ;
\draw  [dash pattern={on 4.5pt off 4.5pt}]  (336.95,167.5) -- (519.07,167.5) ;
\draw  [dash pattern={on 4.5pt off 4.5pt}]  (320.85,193.44) -- (518.29,193.44) ;
\draw  [dash pattern={on 4.5pt off 4.5pt}]  (305.58,218.43) -- (518.68,218.43) ;
\draw  [fill={rgb, 255:red, 65; green, 117; blue, 5 }  ,fill opacity=1 ] (515.37,135.6) .. controls (515.37,133.98) and (516.62,132.67) .. (518.17,132.67) .. controls (519.72,132.67) and (520.98,133.98) .. (520.98,135.6) .. controls (520.98,137.22) and (519.72,138.53) .. (518.17,138.53) .. controls (516.62,138.53) and (515.37,137.22) .. (515.37,135.6) -- cycle ;
\draw  [fill={rgb, 255:red, 65; green, 117; blue, 5 }  ,fill opacity=1 ] (515.37,167.5) .. controls (515.37,165.88) and (516.62,164.57) .. (518.17,164.57) .. controls (519.72,164.57) and (520.98,165.88) .. (520.98,167.5) .. controls (520.98,169.12) and (519.72,170.43) .. (518.17,170.43) .. controls (516.62,170.43) and (515.37,169.12) .. (515.37,167.5) -- cycle ;
\draw  [fill={rgb, 255:red, 65; green, 117; blue, 5 }  ,fill opacity=1 ] (515.37,193.93) .. controls (515.37,192.31) and (516.62,191) .. (518.17,191) .. controls (519.72,191) and (520.98,192.31) .. (520.98,193.93) .. controls (520.98,195.55) and (519.72,196.86) .. (518.17,196.86) .. controls (516.62,196.86) and (515.37,195.55) .. (515.37,193.93) -- cycle ;
\draw  [fill={rgb, 255:red, 65; green, 117; blue, 5 }  ,fill opacity=1 ] (515.88,218.43) .. controls (515.88,216.82) and (517.13,215.5) .. (518.68,215.5) .. controls (520.23,215.5) and (521.48,216.82) .. (521.48,218.43) .. controls (521.48,220.05) and (520.23,221.36) .. (518.68,221.36) .. controls (517.13,221.36) and (515.88,220.05) .. (515.88,218.43) -- cycle ;
\draw  [fill={rgb, 255:red, 65; green, 117; blue, 5 }  ,fill opacity=1 ] (515.75,299.27) .. controls (515.75,297.65) and (517.01,296.34) .. (518.56,296.34) .. controls (520.11,296.34) and (521.36,297.65) .. (521.36,299.27) .. controls (521.36,300.88) and (520.11,302.2) .. (518.56,302.2) .. controls (517.01,302.2) and (515.75,300.88) .. (515.75,299.27) -- cycle ;

\draw (252.42,306.8) node [anchor=north west][inner sep=0.75pt]  [xscale=0.8,yscale=0.8]  {$u_{Y}$};
\draw (96.06,32.06) node [anchor=north west][inner sep=0.75pt]  [font=\large,xscale=0.8,yscale=0.8]  {$y_{0}$};
\draw (212.8,32.71) node [anchor=north west][inner sep=0.75pt]  [font=\large,xscale=0.8,yscale=0.8]  {$y_{{1}}$};
\draw (243.39,31.97) node [anchor=north west][inner sep=0.75pt]  [font=\large,xscale=0.8,yscale=0.8]  {$y_{j_{2}{}}$};
\draw (390.12,31.61) node [anchor=north west][inner sep=0.75pt]  [font=\large,xscale=0.8,yscale=0.8]  {$ y_{j_{k-1}}=\max I$};
\draw (286.76,59.53) node [anchor=north west][inner sep=0.75pt]  [xscale=0.8,yscale=0.8]  {$\cdots $};
\draw (273.67,32.07) node [anchor=north west][inner sep=0.75pt]  [font=\large,xscale=0.8,yscale=0.8]  {$y_{j_{3}}$};
\draw (133.73,32.95) node [anchor=north west][inner sep=0.75pt]  [font=\large,xscale=0.8,yscale=0.8]  {$y_{I}$};
\draw (316.85,31.71) node [anchor=north west][inner sep=0.75pt]  [font=\large,xscale=0.8,yscale=0.8]  {$y_{j_{k-2}}$};
\draw (532,287.07) node [anchor=north west][inner sep=0.75pt]  [xscale=0.8,yscale=0.8]  {$\delta _{0}=\delta_{I'}$};
\draw (532,209.29) node [anchor=north west][inner sep=0.75pt]  [xscale=0.8,yscale=0.8]  {$\delta _{1}$};
\draw (532,182.97) node [anchor=north west][inner sep=0.75pt]  [xscale=0.8,yscale=0.8]  {$\delta _{j_2}$};
\draw (532,155.59) node [anchor=north west][inner sep=0.75pt]  [xscale=0.8,yscale=0.8]  {$\delta _{j_3}$};
\draw (532,123.4) node [anchor=north west][inner sep=0.75pt]  [xscale=0.8,yscale=0.8]  {$\delta _{j_{k-3}}$};

\end{tikzpicture}
\caption{The case $J=I$}
\label{fig: RightProjcaseA}
\end{minipage}\hfill
\begin{minipage}{0.48\textwidth}
  \centering

\tikzset{every picture/.style={line width=0.75pt}} 

\begin{tikzpicture}[x=0.75pt,y=0.75pt,yscale=-0.6,xscale=0.6]

\draw  [dash pattern={on 4.5pt off 4.5pt}]  (285,253.14) -- (519,253) ;
\draw  [dash pattern={on 4.5pt off 4.5pt}]  (254.36,298.7) -- (519.07,298.7) ;
\draw    (252.1,68.35) -- (320.85,193.44) ;
\draw    (106.25,68) -- (254.36,298.7) ;
\draw    (399.37,68) -- (251.27,304.56) ;
\draw    (221.99,67.55) -- (305.58,218.43) ;
\draw  [fill={rgb, 255:red, 208; green, 2; blue, 27 }  ,fill opacity=1 ] (248.75,298.7) .. controls (248.75,295.46) and (251.26,292.84) .. (254.36,292.84) .. controls (257.45,292.84) and (259.97,295.46) .. (259.97,298.7) .. controls (259.97,301.93) and (257.45,304.56) .. (254.36,304.56) .. controls (251.26,304.56) and (248.75,301.93) .. (248.75,298.7) -- cycle ;
\draw  [fill={rgb, 255:red, 208; green, 2; blue, 27 }  ,fill opacity=1 ] (100.64,68) .. controls (100.64,64.77) and (103.16,62.14) .. (106.25,62.14) .. controls (109.35,62.14) and (111.86,64.77) .. (111.86,68) .. controls (111.86,71.24) and (109.35,73.86) .. (106.25,73.86) .. controls (103.16,73.86) and (100.64,71.24) .. (100.64,68) -- cycle ;
\draw  [fill={rgb, 255:red, 208; green, 2; blue, 27 }  ,fill opacity=1 ] (216.39,67.55) .. controls (216.39,64.32) and (218.9,61.69) .. (221.99,61.69) .. controls (225.09,61.69) and (227.6,64.32) .. (227.6,67.55) .. controls (227.6,70.79) and (225.09,73.41) .. (221.99,73.41) .. controls (218.9,73.41) and (216.39,70.79) .. (216.39,67.55) -- cycle ;
\draw  [fill={rgb, 255:red, 208; green, 2; blue, 27 }  ,fill opacity=1 ] (246.49,68.35) .. controls (246.49,65.12) and (249,62.49) .. (252.1,62.49) .. controls (255.19,62.49) and (257.71,65.12) .. (257.71,68.35) .. controls (257.71,71.59) and (255.19,74.21) .. (252.1,74.21) .. controls (249,74.21) and (246.49,71.59) .. (246.49,68.35) -- cycle ;
\draw  [fill={rgb, 255:red, 208; green, 2; blue, 52 }  ,fill opacity=1 ] (393.76,68) .. controls (393.76,64.77) and (396.27,62.14) .. (399.37,62.14) .. controls (402.47,62.14) and (404.98,64.77) .. (404.98,68) .. controls (404.98,71.24) and (402.47,73.86) .. (399.37,73.86) .. controls (396.27,73.86) and (393.76,71.24) .. (393.76,68) -- cycle ;
\draw    (321.68,67.63) -- (356.8,135.6) ;
\draw  [fill={rgb, 255:red, 208; green, 2; blue, 52 }  ,fill opacity=1 ] (316.07,67.63) .. controls (316.07,64.39) and (318.58,61.77) .. (321.68,61.77) .. controls (324.77,61.77) and (327.29,64.39) .. (327.29,67.63) .. controls (327.29,70.86) and (324.77,73.48) .. (321.68,73.48) .. controls (318.58,73.48) and (316.07,70.86) .. (316.07,67.63) -- cycle ;
\draw    (282.92,67.63) -- (336.95,167.5) ;
\draw  [fill={rgb, 255:red, 208; green, 2; blue, 27 }  ,fill opacity=1 ] (277.31,67.63) .. controls (277.31,64.39) and (279.82,61.77) .. (282.92,61.77) .. controls (286.02,61.77) and (288.53,64.39) .. (288.53,67.63) .. controls (288.53,70.86) and (286.02,73.48) .. (282.92,73.48) .. controls (279.82,73.48) and (277.31,70.86) .. (277.31,67.63) -- cycle ;
\draw    (137.68,68.03) .. controls (117.11,110.55) and (296.31,179.3) .. (282,258) ;
\draw  [fill={rgb, 255:red, 208; green, 2; blue, 52 }  ,fill opacity=1 ] (132.07,68.03) .. controls (132.07,64.79) and (134.58,62.17) .. (137.68,62.17) .. controls (140.78,62.17) and (143.29,64.79) .. (143.29,68.03) .. controls (143.29,71.26) and (140.78,73.89) .. (137.68,73.89) .. controls (134.58,73.89) and (132.07,71.26) .. (132.07,68.03) -- cycle ;
\draw    (518.17,56.84) -- (518.17,321.58) ;
\draw  [dash pattern={on 4.5pt off 4.5pt}]  (356.8,135.6) -- (519.07,135.6) ;
\draw  [dash pattern={on 4.5pt off 4.5pt}]  (336.95,167.5) -- (519.07,167.5) ;
\draw  [dash pattern={on 4.5pt off 4.5pt}]  (320.85,193.44) -- (518.29,193.44) ;
\draw  [dash pattern={on 4.5pt off 4.5pt}]  (305.58,218.43) -- (518.68,218.43) ;
\draw  [fill={rgb, 255:red, 65; green, 117; blue, 5 }  ,fill opacity=1 ] (515.37,135.6) .. controls (515.37,133.98) and (516.62,132.67) .. (518.17,132.67) .. controls (519.72,132.67) and (520.98,133.98) .. (520.98,135.6) .. controls (520.98,137.22) and (519.72,138.53) .. (518.17,138.53) .. controls (516.62,138.53) and (515.37,137.22) .. (515.37,135.6) -- cycle ;
\draw  [fill={rgb, 255:red, 65; green, 117; blue, 5 }  ,fill opacity=1 ] (515.37,167.5) .. controls (515.37,165.88) and (516.62,164.57) .. (518.17,164.57) .. controls (519.72,164.57) and (520.98,165.88) .. (520.98,167.5) .. controls (520.98,169.12) and (519.72,170.43) .. (518.17,170.43) .. controls (516.62,170.43) and (515.37,169.12) .. (515.37,167.5) -- cycle ;
\draw  [fill={rgb, 255:red, 65; green, 117; blue, 5 }  ,fill opacity=1 ] (515.37,193.93) .. controls (515.37,192.31) and (516.62,191) .. (518.17,191) .. controls (519.72,191) and (520.98,192.31) .. (520.98,193.93) .. controls (520.98,195.55) and (519.72,196.86) .. (518.17,196.86) .. controls (516.62,196.86) and (515.37,195.55) .. (515.37,193.93) -- cycle ;
\draw  [fill={rgb, 255:red, 65; green, 117; blue, 5 }  ,fill opacity=1 ] (515.88,218.43) .. controls (515.88,216.82) and (517.13,215.5) .. (518.68,215.5) .. controls (520.23,215.5) and (521.48,216.82) .. (521.48,218.43) .. controls (521.48,220.05) and (520.23,221.36) .. (518.68,221.36) .. controls (517.13,221.36) and (515.88,220.05) .. (515.88,218.43) -- cycle ;
\draw  [fill={rgb, 255:red, 65; green, 117; blue, 5 }  ,fill opacity=1 ] (515.75,299.27) .. controls (515.75,297.65) and (517.01,296.34) .. (518.56,296.34) .. controls (520.11,296.34) and (521.36,297.65) .. (521.36,299.27) .. controls (521.36,300.88) and (520.11,302.2) .. (518.56,302.2) .. controls (517.01,302.2) and (515.75,300.88) .. (515.75,299.27) -- cycle ;
\draw  [fill={rgb, 255:red, 208; green, 2; blue, 27 }  ,fill opacity=1 ] (279.39,253.14) .. controls (279.39,249.9) and (281.9,247.28) .. (285,247.28) .. controls (288.1,247.28) and (290.61,249.9) .. (290.61,253.14) .. controls (290.61,256.38) and (288.1,259) .. (285,259) .. controls (281.9,259) and (279.39,256.38) .. (279.39,253.14) -- cycle ;
\draw  [fill={rgb, 255:red, 65; green, 117; blue, 5 }  ,fill opacity=1 ] (516.2,253) .. controls (516.2,251.38) and (517.45,250.07) .. (519,250.07) .. controls (520.55,250.07) and (521.8,251.38) .. (521.8,253) .. controls (521.8,254.62) and (520.55,255.93) .. (519,255.93) .. controls (517.45,255.93) and (516.2,254.62) .. (516.2,253) -- cycle ;

\draw (252.42,306.8) node [anchor=north west][inner sep=0.75pt]  [xscale=0.8,yscale=0.8]  {$u_{Y}$};
\draw (96.06,32.06) node [anchor=north west][inner sep=0.75pt]  [font=\large,xscale=0.8,yscale=0.8]  {$y_{0}$};
\draw (212.8,32.71) node [anchor=north west][inner sep=0.75pt]  [font=\large,xscale=0.8,yscale=0.8]  {$y_{1}$};
\draw (243.39,31.97) node [anchor=north west][inner sep=0.75pt]  [font=\large,xscale=0.8,yscale=0.8]  {$y_{j_{1}{}}$};
\draw (390.12,31.97) node [anchor=north west][inner sep=0.75pt]  [font=\large,xscale=0.8,yscale=0.8]  {$y_{\ell } = \max J$};
\draw (286.76,59.53) node [anchor=north west][inner sep=0.75pt]  [xscale=0.8,yscale=0.8]  {$\cdots $};
\draw (273.67,32.07) node [anchor=north west][inner sep=0.75pt]  [font=\large,xscale=0.8,yscale=0.8]  {$y_{j_{2}}$};
\draw (133.73,32.95) node [anchor=north west][inner sep=0.75pt]  [font=\large,xscale=0.8,yscale=0.8]  {$y_{J}$};
\draw (316.85,31.71) node [anchor=north west][inner sep=0.75pt]  [font=\large,xscale=0.8,yscale=0.8]  {$y_{j_{k-2}}$};
\draw (532,287.07) node [anchor=north west][inner sep=0.75pt]  [xscale=0.8,yscale=0.8]  {$\delta _{0} =\delta _{I'}$};
\draw (532,209.29) node [anchor=north west][inner sep=0.75pt]  [xscale=0.8,yscale=0.8]  {$\delta _{j_1}$};
\draw (532,182.97) node [anchor=north west][inner sep=0.75pt]  [xscale=0.8,yscale=0.8]  {$\delta _{j_2}$};
\draw (532,155.59) node [anchor=north west][inner sep=0.75pt]  [xscale=0.8,yscale=0.8]  {$\delta _{j_3}$};
\draw (532,123.4) node [anchor=north west][inner sep=0.75pt]  [xscale=0.8,yscale=0.8]  {$\delta _{j_{k -2}}$};
\draw (295,256.54) node [anchor=north west][inner sep=0.75pt]  [xscale=0.8,yscale=0.8]  {$u_{J}$};
\draw (532,240.29) node [anchor=north west][inner sep=0.75pt]  [xscale=0.8,yscale=0.8]  {$\delta _{J'}$};
\end{tikzpicture}
\caption{The case $J\neq I$.}
\label{fig: RightProjcaseB}
\end{minipage}
\end{figure}
To show that (\ref{eqn: rightcombeqn1}) holds, it will be sufficient to show that
     \begin{align}
        \label{eqn: right2}
        \delta(y_{J},y_{j_1})<\delta(y_{j_1},y_{j_2}) = \delta_{j_1}.
    \end{align} 
    To this end, we consider the cases $J= I$ and $J\ne I$ separately. 
    \\
    In the first case, i.e., $J = I =\{j_1< \dots < j_{k-1}\}$  (see Figure \ref{fig: RightProjcaseA}), since $1\in I$, we have $j_1 = 1$, and $\delta_1 = \delta_{j_1}$. Note that by \ref{it: rightc}, $\delta(y_I, y_1)=\delta_0$. In view of (\ref{eqn: deltasrightproj}), $\delta_0 < \delta_{1}$ and thus (\ref{eqn: right2}) holds. 
    \\
    On the other hand, if $J\neq I$ (see Figure \ref{fig: RightProjcaseB}), then $J'\in \{2,\dots,\ell-1\}^{(k-2)}$ and thus $j_1\ge 2$. Note that $\delta(y_J,y_{j_1})=\min\{\delta(y_J,y_1),\delta(y_1,y_{j_1})=\delta_1\}\le \delta_{1}.$ Since $j_1>1$, in view of (\ref{eqn: deltasrightproj}), this implies that $\delta(y_J,y_{j_1})\le\delta_{1}<\delta_{j_1},$ and consequently (\ref{eqn: right2}) holds. 

   To summarize, we have shown that in either case $Y^J$ satisfies (\ref{eqn: rightcombeqn1}) and hence forms a right comb. By \ref{it: comb2L} in Remark \ref{fact: combstructure}, $Y^J$ has projection (see also Fig \ref{fig: RightProjcaseA} and \ref{fig: RightProjcaseB}):  
    \begin{align}
        \label{eqn: right1}
        \pi(Y^J)=\{\delta_{J'}\}\cup \set{\delta_j:j\in J'},
    \end{align}
   for every $J'\in \{2,\dots, \ell-1\}^{(k-2)}\cup \{I'\}.$  Using the definition of $\chi_c$ in (\ref{eqn: coloring4}), and noting that $\chi_c(Y^{J})=\alpha$, this shows that $$ \chi_c(Y^J)= c(\pi(Y^J))= c\left(\set{\delta_{J'}}\cup\set{\delta_i:i\in J'}\right) = \alpha,$$ 
    for every $J'\in \{2,\dots, \ell-1\}^{(k-2)}\cup \{I'\}.$ 
    This concludes the proof that $c^{-1}(\alpha)$ contains a monochromatic copy of $(F,<)$. 
\end{proof}
\subsubsection{Case analysis of leaf sets and proof of Theorem \ref{thm: steppingup}}
\label{sec: caseanalysis}
Formally, the proof of Theorem \ref{thm: steppingup} follows by induction and so we first state and prove the base case, i.e., the case of $3$-uniform hypergraphs. This is Proposition \ref{prop: basecase} below. Proposition \ref{prop: inductionstep} consists of the induction step for $k \geq 4$. 
\begin{proposition}[The case $k = 3$]
    \label{prop: basecase}
    For all positive integers $N\geq n > t$, if the coloring $c:[N]^{(2)}\to \{0,1\}$ does not contain a monochromatic $K_t^{(2)}$, then $\chi_c:[2^N]^{(3)} \to \Z_4$ 
    \begin{enumerate}[label = (\alph*)]
        \item\label{it: basecasea}  does not contain  a monochromatic copy of any member of $\cF_{I}^{(3)}(n)$ in the colors $\{0,1\}$, and 
        \item\label{it: basecaseb}  does not contain a monochromatic copy of  any member of $\rev \cF_{I}^{(3)}(n)$ in the colors $\{2,3\}$, 
    \end{enumerate}
where $I = \{1,2\}.$
\end{proposition}
\begin{proof}
    We will only prove \ref{it: basecasea}, since the proof of \ref{it: basecaseb} is analogous. The proof is by contradiction and to this end let us assume, without loss of generality, that $\chi_c$ contains a monochromatic copy of $(F,<)\in \cF_{I}^{(3)}(n)$ in color 0. Recall that, in view of Definition \ref{def:F_I}, the vertex set $V(F)$ is given by 
    $$V(F)=\{x_0,x_1,\ldots,x_n\}\cup\set{x_J:J\in\{2,\ldots,n\}^{(2)}}.$$ 
    The ordering of $V(F)$ satisfies $x_0<x_1<x_2<\cdots<x_n$ and $x_0\le x_J\le x_1$ for all $J\in [n]^{(2)}$. Let the edge set be 
    \begin{align}
    \label{eqn: Finbasecase}
        F= \set{x_0,x_1,x_2}\cup \set{\set{x_J}\cup\set{x_j:j\in J}:J\in\{2,\ldots,n\}^{(2)}}.
    \end{align}
    Since the copy of $(F,<)$ is monochromatic of color $0$, in view of (\ref{eqn: coloring3}), every edge of $F$ forms a left comb. As a consequence, we have the following claim. 
\begin{claim*}
        The set $X = \set{x_0<x_1<\cdots<x_n}$ forms a left comb. 
    \end{claim*}
    \begin{proof}
    
    Note that the special edge  $\{x_0 < x_1< x_2\}\in F$ forms a left comb since it has color 0. Let $$\delta_i:= \delta(x_{i-1}, x_i) \text{ for each } 1\le i \le n.$$
    We prove the statement by induction.  Assume that $\{x_0 < x_1 < \cdots < x_\ell\}$ forms a left comb for some integer $2\le \l < n$, or in other words, we have  
    $\delta_1 > \delta_2  > \cdots > \delta_{\l}.$
    We will show that $\{x_0 < \cdots < x_{\l+1}\}$ forms a left comb as well, or in other words, that $\delta_{\ell+1}< \delta_\ell.$
    We omit the standard formal proof of the induction step and instead provide a proof by picture. 
    \begin{figure}[h!]
    \centering
\tikzset{every picture/.style={line width=0.75pt}} 

\begin{tikzpicture}[x=0.75pt,y=0.75pt,yscale=-0.7,xscale=0.7]

\draw    (251.95,126.53) -- (146.9,267.28) ;
\draw    (38.97,125.4) -- (166.36,293.3) ;
\draw    (291.08,125.4) -- (163.7,297.56) ;
\draw    (210.7,126.39) -- (125.83,240.68) ;
\draw    (171.33,125.36) -- (104.19,212.43) ;
\draw    (133.7,125.07) -- (86.71,188.1) ;
\draw  [fill={rgb, 255:red, 208; green, 2; blue, 27 }  ,fill opacity=1 ] (161.53,293.3) .. controls (161.53,290.94) and (163.69,289.03) .. (166.36,289.03) .. controls (169.02,289.03) and (171.18,290.94) .. (171.18,293.3) .. controls (171.18,295.65) and (169.02,297.56) .. (166.36,297.56) .. controls (163.69,297.56) and (161.53,295.65) .. (161.53,293.3) -- cycle ;
\draw  [fill={rgb, 255:red, 208; green, 2; blue, 27 }  ,fill opacity=1 ] (34.15,125.4) .. controls (34.15,123.04) and (36.31,121.13) .. (38.97,121.13) .. controls (41.64,121.13) and (43.8,123.04) .. (43.8,125.4) .. controls (43.8,127.75) and (41.64,129.66) .. (38.97,129.66) .. controls (36.31,129.66) and (34.15,127.75) .. (34.15,125.4) -- cycle ;
\draw  [fill={rgb, 255:red, 208; green, 2; blue, 27 }  ,fill opacity=1 ] (128.87,125.07) .. controls (128.87,122.72) and (131.03,120.81) .. (133.7,120.81) .. controls (136.36,120.81) and (138.52,122.72) .. (138.52,125.07) .. controls (138.52,127.43) and (136.36,129.34) .. (133.7,129.34) .. controls (131.03,129.34) and (128.87,127.43) .. (128.87,125.07) -- cycle ;
\draw  [fill={rgb, 255:red, 208; green, 2; blue, 27 }  ,fill opacity=1 ] (166.51,125.36) .. controls (166.51,123.01) and (168.67,121.1) .. (171.33,121.1) .. controls (174,121.1) and (176.16,123.01) .. (176.16,125.36) .. controls (176.16,127.72) and (174,129.63) .. (171.33,129.63) .. controls (168.67,129.63) and (166.51,127.72) .. (166.51,125.36) -- cycle ;
\draw  [fill={rgb, 255:red, 208; green, 2; blue, 27 }  ,fill opacity=1 ] (205.88,126.39) .. controls (205.88,124.04) and (208.04,122.13) .. (210.7,122.13) .. controls (213.37,122.13) and (215.53,124.04) .. (215.53,126.39) .. controls (215.53,128.75) and (213.37,130.66) .. (210.7,130.66) .. controls (208.04,130.66) and (205.88,128.75) .. (205.88,126.39) -- cycle ;
\draw  [fill={rgb, 255:red, 208; green, 2; blue, 27 }  ,fill opacity=1 ] (247.12,126.53) .. controls (247.12,124.17) and (249.28,122.27) .. (251.95,122.27) .. controls (254.61,122.27) and (256.77,124.17) .. (256.77,126.53) .. controls (256.77,128.89) and (254.61,130.79) .. (251.95,130.79) .. controls (249.28,130.79) and (247.12,128.89) .. (247.12,126.53) -- cycle ;
\draw  [fill={rgb, 255:red, 65; green, 117; blue, 5 }  ,fill opacity=1 ] (286.26,125.4) .. controls (286.26,123.04) and (288.42,121.13) .. (291.08,121.13) .. controls (293.74,121.13) and (295.9,123.04) .. (295.9,125.4) .. controls (295.9,127.75) and (293.74,129.66) .. (291.08,129.66) .. controls (288.42,129.66) and (286.26,127.75) .. (286.26,125.4) -- cycle ;
\draw [color={rgb, 255:red, 14; green, 18; blue, 193 }  ,draw opacity=1 ] [dash pattern={on 4.5pt off 4.5pt}]  (264.27,161.33) .. controls (279.26,167.61) and (290.91,152.7) .. (325.04,124.45) ;
\draw [color={rgb, 255:red, 208; green, 2; blue, 27 }  ,draw opacity=1 ] [dash pattern={on 4.5pt off 4.5pt}]  (325.04,124.45) -- (188.52,323) ;
\draw [color={rgb, 255:red, 208; green, 2; blue, 27 }  ,draw opacity=1 ] [dash pattern={on 4.5pt off 4.5pt}]  (166.36,293.3) -- (188.52,323) ;
\draw  [fill={rgb, 255:red, 65; green, 117; blue, 5 }  ,fill opacity=1 ] (320.21,124.45) .. controls (320.21,122.09) and (322.37,120.18) .. (325.04,120.18) .. controls (327.7,120.18) and (329.86,122.09) .. (329.86,124.45) .. controls (329.86,126.8) and (327.7,128.71) .. (325.04,128.71) .. controls (322.37,128.71) and (320.21,126.8) .. (320.21,124.45) -- cycle ;
\draw    (81.26,125.86) .. controls (69.49,157.41) and (97.79,146.42) .. (106.95,161.33) ;
\draw  [fill={rgb, 255:red, 65; green, 117; blue, 5 }  ,fill opacity=1 ] (76.43,125.86) .. controls (76.43,123.5) and (78.59,121.59) .. (81.26,121.59) .. controls (83.92,121.59) and (86.08,123.5) .. (86.08,125.86) .. controls (86.08,128.21) and (83.92,130.12) .. (81.26,130.12) .. controls (78.59,130.12) and (76.43,128.21) .. (76.43,125.86) -- cycle ;
\draw  [fill={rgb, 255:red, 208; green, 2; blue, 27 }  ,fill opacity=1 ] (259.45,161.33) .. controls (259.45,158.98) and (261.61,157.07) .. (264.27,157.07) .. controls (266.94,157.07) and (269.1,158.98) .. (269.1,161.33) .. controls (269.1,163.69) and (266.94,165.6) .. (264.27,165.6) .. controls (261.61,165.6) and (259.45,163.69) .. (259.45,161.33) -- cycle ;
\draw    (452.55,182.81) -- (515.48,271.61) ;
\draw    (576.52,183.5) -- (515.48,271.61) ;
\draw    (518.8,181.38) -- (549.34,223.41) ;
\draw  [fill={rgb, 255:red, 208; green, 2; blue, 27 }  ,fill opacity=1 ] (513.41,271.61) .. controls (513.41,270.35) and (514.34,269.33) .. (515.48,269.33) .. controls (516.63,269.33) and (517.56,270.35) .. (517.56,271.61) .. controls (517.56,272.87) and (516.63,273.89) .. (515.48,273.89) .. controls (514.34,273.89) and (513.41,272.87) .. (513.41,271.61) -- cycle ;
\draw  [fill={rgb, 255:red, 65; green, 117; blue, 5 }  ,fill opacity=1 ] (450.47,182.81) .. controls (450.47,181.55) and (451.4,180.53) .. (452.55,180.53) .. controls (453.69,180.53) and (454.62,181.55) .. (454.62,182.81) .. controls (454.62,184.06) and (453.69,185.08) .. (452.55,185.08) .. controls (451.4,185.08) and (450.47,184.06) .. (450.47,182.81) -- cycle ;
\draw  [fill={rgb, 255:red, 65; green, 117; blue, 5 }  ,fill opacity=1 ] (516.73,181.38) .. controls (516.73,180.12) and (517.66,179.1) .. (518.8,179.1) .. controls (519.95,179.1) and (520.88,180.12) .. (520.88,181.38) .. controls (520.88,182.63) and (519.95,183.65) .. (518.8,183.65) .. controls (517.66,183.65) and (516.73,182.63) .. (516.73,181.38) -- cycle ;
\draw  [fill={rgb, 255:red, 65; green, 117; blue, 5 }  ,fill opacity=1 ] (575.7,181.22) .. controls (575.7,179.96) and (576.63,178.94) .. (577.78,178.94) .. controls (578.93,178.94) and (579.86,179.96) .. (579.86,181.22) .. controls (579.86,182.48) and (578.93,183.5) .. (577.78,183.5) .. controls (576.63,183.5) and (575.7,182.48) .. (575.7,181.22) -- cycle ;
\draw  [fill={rgb, 255:red, 208; green, 2; blue, 27 }  ,fill opacity=1 ] (547.26,223.41) .. controls (547.26,222.15) and (548.19,221.14) .. (549.34,221.14) .. controls (550.48,221.14) and (551.42,222.15) .. (551.42,223.41) .. controls (551.42,224.67) and (550.48,225.69) .. (549.34,225.69) .. controls (548.19,225.69) and (547.26,224.67) .. (547.26,223.41) -- cycle ;
\draw    (291,220) -- (430,220.99) ;
\draw [shift={(432,221)}, rotate = 180.41] [color={rgb, 255:red, 0; green, 0; blue, 0 }  ][line width=0.75]    (10.93,-3.29) .. controls (6.95,-1.4) and (3.31,-0.3) .. (0,0) .. controls (3.31,0.3) and (6.95,1.4) .. (10.93,3.29)   ;

\draw (176.72,289.27) node [anchor=north west][inner sep=0.75pt]  [xscale=0.8,yscale=0.8]  {$u$};
\draw (31.53,100.57) node [anchor=north west][inner sep=0.75pt]  [xscale=0.8,yscale=0.8]  {$x_{0}$};
\draw (127.25,101.36) node [anchor=north west][inner sep=0.75pt]  [xscale=0.8,yscale=0.8]  {$x_{1}$};
\draw (163.88,102.14) node [anchor=north west][inner sep=0.75pt]  [xscale=0.8,yscale=0.8]  {$x_{2}$};
\draw (200.5,102.14) node [anchor=north west][inner sep=0.75pt]  [xscale=0.8,yscale=0.8]  {$x_{3}$};
\draw (234.04,102.14) node [anchor=north west][inner sep=0.75pt]  [xscale=0.8,yscale=0.8]  {$x_{\ell -1}$};
\draw (284.32,102.93) node [anchor=north west][inner sep=0.75pt]  [xscale=0.8,yscale=0.8]  {$x_{\ell }$};
\draw (192.11,155.14) node [anchor=north west][inner sep=0.75pt]  [rotate=-42.24,xscale=0.8,yscale=0.8]  {$\cdots $};
\draw (328.1,103.71) node [anchor=north west][inner sep=0.75pt]  [xscale=0.8,yscale=0.8]  {$x_{\ell +1}$};
\draw (74.31,102.14) node [anchor=north west][inner sep=0.75pt]  [xscale=0.8,yscale=0.8]  {$x_{J}{}$};
\draw (507.72,275.27) node [anchor=north west][inner sep=0.75pt]  [xscale=0.8,yscale=0.8]  {$u$};
\draw (441.72,157.27) node [anchor=north west][inner sep=0.75pt]  [xscale=0.8,yscale=0.8]  {$x_{J}$};
\draw (509.72,156.27) node [anchor=north west][inner sep=0.75pt]  [xscale=0.8,yscale=0.8]  {$x_{\ell }$};
\draw (569.72,155.27) node [anchor=north west][inner sep=0.75pt]  [xscale=0.8,yscale=0.8]  {$x_{\ell +1}$};
\draw (302,196.4) node [anchor=north west][inner sep=0.75pt]  [color={rgb, 255:red, 14; green, 18; blue, 193 }  ,opacity=1 ,xscale=0.8,yscale=0.8]  {$\delta _{\ell +1}  >\delta _{\ell } \ \text{implies}$};
\draw (66,180.4) node [anchor=north west][inner sep=0.75pt]  [xscale=0.8,yscale=0.8]  {$\delta _{1}$};
\draw (86,205.4) node [anchor=north west][inner sep=0.75pt]  [xscale=0.8,yscale=0.8]  {$\delta _{2}$};
\draw (108,233.4) node [anchor=north west][inner sep=0.75pt]  [xscale=0.8,yscale=0.8]  {$\delta _{3}$};
\draw (143,283.4) node [anchor=north west][inner sep=0.75pt]  [xscale=0.8,yscale=0.8]  {$\delta _{\ell }$};
\end{tikzpicture}
  \caption{
    There are only two possibilities for the vertex $x_{\l+1}$ where the red case corresponds to $\delta_\l > \delta_{\l+1}$ and the blue case corresponds to $\delta_{\l+1}< \delta_\l$. }
    \label{fig: subcomb}
\end{figure}
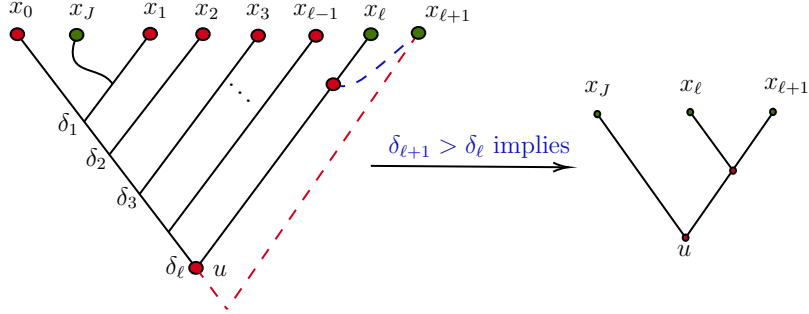
    
    As shown in Figure \ref{fig: subcomb}, note that $x_{\ell+1}$ has only two possibilities -- either $\delta_{\ell+1}<\delta_\ell$ (red case) or $\delta_{\ell+1}> \delta_\ell$ (blue case). Let $J=\set{\ell,\ell+1}$. From Figure \ref{fig: subcomb}, it is clear that in the blue case, i.e., if  $\delta(x_\ell,x_{\ell+1}) = \delta_{\ell+1} > \delta_\l$,  then $\{x_J, x_\ell, x_{\ell+1}\}$ forms a right comb, which is a contradiction, since $\{x_J, x_{\ell}, x_{\ell+1}\}$ is an edge of color 0. Consequently, the red case holds, i.e., $\delta_{\ell} > \delta_{\ell+1}$, and $\{x_0, x_1,\dots, x_{\ell+1}\}$ forms a left comb.  
    \end{proof}
    Having obtained the left comb $X$, we apply Proposition \ref{prop: leftcomb} to $X=\set{x_0<\cdots<x_n}$ with the coloring $c:[N]^{(2)}\to \set{0,1}$. Condition \ref{it: leftlema} is satisfied, since $X$ forms a left comb. Condition \ref{it: leftlemb} is satisfied due to (\ref{eqn: Finbasecase}). Condition \ref{it: leftlemmac} also holds, since $F$ in (\ref{eqn: Finbasecase}) is monochromatic. Consequently, Proposition \ref{prop: leftcomb} implies that $\pi(X\setminus \set{x_0})^{(2)}$ is a monochromatic clique of size $n-1\ge t$. This, however, contradicts the assumption that $c$ does not contain a monochromatic copy of $K_t^{(2)}$, and consequently \ref{it: basecasea} holds. 
\end{proof}
Note that Proposition \ref{prop: basecase}, together with the exponential lower bound on $R(K_t^{(2)};2)$ in \cite{Erd47} proves Theorem \ref{thm: steppingup} for $k = 3$ (we will prove this formally at the end of this section). 

We will now state and prove Proposition \ref{prop: inductionstep}, which we use to prove Theorem \ref{thm: steppingup} for $k\ge 4$. 

Recall that in Definition \ref{def:F_I}, for a set $1\in I\in [n]^{(k-1)}$, we define the family of ordered $k$-graphs  $\cF_I^{(k)}(n)$. An ordered $k$-graph $(F,<)\in \cF_I^{(k)}(n)$ has vertex set $$V(F) = \{x_0,x_1,\dots, x_n\}\cup \{x_J:J\in \{2,\dots,n\}^{(k-1)}\cup\{I\}\},$$ satisfying the linear order:
$$x_0=x_I\le \{x_J:J\in \{2,\dots,n\}^{(k-1)}\} \le x_1 < x_2<\dots< x_n,$$
with \textit{distinguished  set of vertices} $X:=\{x_0,\ldots,x_n\}$. The edge $X^I := \{x_0\}\cup\{x_i\}_{i\in I}$ is called the \textit{special edge}.  
Recall that the ordered $k$-graph $(F,<_{\rev})$ was defined to be the hypergraph with the vertices of $F$ in the ``reverse order'', and $\rev \cF_I^{(k)}$ was the collection of $(F, <_{\rev})$ where $(F,<)\in \cF_I^{(k)}(n)$. 

Let $t= n/8k$ and $I\in [n]^{(k-1)}$ and $I'\in [t]^{(k-2)}$ be conveniently chosen. In Proposition \ref{prop: inductionstep} we will show that if the projection coloring $c:[N]^{(k-1)}\to \ZZ_4$ avoids all members of $\cF_{I'}^{(k-1)}(t)$ in colors $\{0,1\}$ and all members of $\rev \cF_{I'}^{(k-1)}(t)$ in colors $\{2,3\}$, then the coloring $\chi_c:[2^N]^{(k)}\to \ZZ_4$ avoids all members of  $\cF_{I}^{(k)}(n)$ in colors $\{0,1\}$ and all members of $\rev \cF_{I}^{(k)}(n)$ in colors $\{2,3\}$. To this end, we will fix $I\in [n]^{(k-1)}$, where
$$I = \{1,2,a_{3}, \dots, a_{k-1}\}.$$
However, we will need $I$ to satisfy an additional property, which we briefly motivate. Assuming that there is a monochromatic copy of $(F,<)\in \cF_I^{(k)}(n)$ in the leaf set, our plan is to find a subset $Y\subseteq V(F)$ that forms a comb (left or right), such that $x_{a_i}\le Y \le x_{a_{i+1}}$ for some $i\in \{2,\dots, k-2\}$. To ensure that $Y$ is ``large'' (of the order of size $t$), we will require the set $I$ to be ``separated'', i.e., that $(a_{i+1}- a_i)> t$ for every $i\in \{2,\dots, k-2\}$. To make sure that the induction goes through, the forbidden monochromatic copy $(F',<)\in \cF_{I'}^{(k-1)}(t)$ in the projection must also belong to a family $\cF_{I'}^{(k-1)}(t)$ where $I'$ is ``separated''. 
\begin{definition}[Separated set]
\label{def: separated}
    Given positive integers $n \geq k \geq 3$ and a subset $I\in [n]^{(k-1)}$, we say that $I$ is  $(n,k)$-\textit{separated} if $I = \{a_1 <a_2 < a_3, \dots < a_{k-1}\}$ satisfies: 
    $$a_1 = 1, a_2 = 2 \text{ and } (a_{i+1}-a_i) \geq \frac{n}{2k} \text{ for every $2\le i\le k-1$ where $a_k := n$.}\footnote{In the future, for $I\in [n]^{(k-1)}$ which is \textit{$(n,k)$-separated}, we will just say $I\in [n]^{(k-1)}$ is separated, and $(n,k)$ will be clear from the size of ground set and size of $I$.}$$
\end{definition}
We remark that by the above definition, for $k =3$, there is only one ``separated set'', $I=\{1,2\}$. Now we state the induction step formally. 
\begin{proposition}[Induction Step]
\label{prop: inductionstep}
    For any choice of  positive integers $k \geq 4$, there exists a sufficiently large $n_k$ such that for any $N\geq  n\geq n_k$ the following holds.  
    
    Let $t = n/8k$ and $I\in [n]^{(k-1)}$ be an $(n,k)$-separated set and $I'\in [t]^{(k-2)}$ be a $(t,k-1)$-separated set. If the coloring $c: [N]^{(k-1)}\to \Z_4$ does not contain:
    \begin{enumerate}[label = (\roman*)]
        \item\label{it1: inductionstepi} a monochromatic copy of any member of $\cF_{I'}^{(k-1)}(t)$ in colors $\{0,1\}$ and 
        \item\label{it2: inductionstepii} a monochromatic copy of any member of $\rev \cF_{I'}^{(k-1)}(t)$ in colors $\{2,3\}$, 
    \end{enumerate}
    then the coloring $\chi_c: [2^N]^{(k)} \to \Z_4$ does not contain
    \begin{enumerate}[label = (\alph*)]
        \item \label{it1: inductionstepa} a monochromatic copy of any member of $\cF_I^{(k)}(n)$ in colors $\{0,1\}$ and 
        \item \label{it1: inductionstepb} a monochromatic copy of any member of $\rev \cF_I^{(k)}(n)$ in colors $\{2,3\}$. 
    \end{enumerate}
\end{proposition}
    Before we begin the proof, we state the following technical fact, which will be useful in the proof of Proposition \ref{prop: inductionstep}. Assuming that there is a monochromatic copy of  $(F,<)\in \cF_{I}^{(k)}(n)$ in color 0 or 1 in the leaf set $[2^N]$, we will use the following Fact \ref{fact: bigfact} to find large right combs $Y \subseteq V(F)$. Note that in this case, the edges of $F$ are not in color 2, and thus, in view of (\ref{eqn: coloring4}), will not form $(1,k-1)$-splits.    
\begin{fact} \label{fact: bigfact}
    Let $4\le k\le \ell$ and let $ \set{y_0<y_1<\cdots<y_\ell}\subseteq \left[2^N\right]$. If
    \begin{enumerate}[label=(\alph*)]
        \item\label{it: bigfactb} $\set{y_1,y_2,y_\ell}$ forms a right comb, and
        \item\label{it: bigfacta} for all $J\in \set{2,\ldots,\ell}^{(k-1)}$, there exists $y_J$ such that  $y_0\le y_J\le y_1$ and $Y^J=\{y_J\}\cup\{y_j: j\in J\}$ does not form a $(1,k-1)$-split
    \end{enumerate}
    then $\set{y_0,y_1,\cdots, y_\ell}\setminus \{y_1\} = \{y_0, y_2,\dots, y_\ell\}$ forms a right comb.
\end{fact}

\begin{proof}
    We prove by induction on $j$ that $\set{y_0<y_2<\cdots<y_j<y_\ell}$ forms a right comb. As the base case, consider $\set{y_0,y_2,y_\ell}$. A moment of inspection reveals that  $\set{y_0,y_2,y_\ell}$ is a right comb. This follows from condition \ref{it: bigfactb}, i.e, that $\set{y_1,y_2,y_\ell}$ forms a right comb, together with (\ref{eqn: mindelta}) and the fact that $y_0<y_1$.

    Now assume that $\set{y_0<y_2<\cdots<y_j<y_\ell}$, with $2\le j<\ell-1$, forms a right comb, but   $\set{y_0<y_2<\cdots<y_j< y_{j+1}<y_\ell}$ does not form  a right comb.  Fix a set $J\in \set{2,\ldots,\ell}^{(k-1)}$ that contains $j$, $j+1$ and $\ell$. Set $u=u_{Y^J} = a(y_J,y_\ell)$. In Figure \ref{fig: bigfact} below, $Y^J$ consists of the green vertices. We show that in this case $Y^J$ forms a $(1,k-1)$-split, which contradicts assumption \ref{it: bigfacta}. 
    \begin{figure}[h]
    \centering
\tikzset{every picture/.style={line width=0.75pt}} 

\begin{tikzpicture}[x=0.75pt,y=0.75pt,yscale=-1,xscale=1]

\draw    (559.95,191.96) -- (540.57,222.62) ;
\draw    (34.88,109.29) -- (179.99,335.33) ;
\draw    (322.08,109.29) -- (179.99,335.33) ;
\draw  [fill={rgb, 255:red, 208; green, 2; blue, 27 }  ,fill opacity=1 ] (29.38,109.29) .. controls (29.38,106.12) and (31.84,103.55) .. (34.88,103.55) .. controls (37.91,103.55) and (40.37,106.12) .. (40.37,109.29) .. controls (40.37,112.46) and (37.91,115.03) .. (34.88,115.03) .. controls (31.84,115.03) and (29.38,112.46) .. (29.38,109.29) -- cycle ;
\draw  [fill={rgb, 255:red, 65; green, 117; blue, 5 }  ,fill opacity=1 ] (316.59,109.29) .. controls (316.59,106.12) and (319.05,103.55) .. (322.08,103.55) .. controls (325.12,103.55) and (327.58,106.12) .. (327.58,109.29) .. controls (327.58,112.46) and (325.12,115.03) .. (322.08,115.03) .. controls (319.05,115.03) and (316.59,112.46) .. (316.59,109.29) -- cycle ;
\draw [color={rgb, 255:red, 14; green, 18; blue, 193 }  ,draw opacity=1 ] [dash pattern={on 4.5pt off 4.5pt}]  (229,148.77) -- (254.08,109.79) ;
\draw    (207.98,108.92) -- (260.92,206.78) ;
\draw  [fill={rgb, 255:red, 65; green, 117; blue, 5 }  ,fill opacity=1 ] (202.49,108.92) .. controls (202.49,105.75) and (204.95,103.18) .. (207.98,103.18) .. controls (211.02,103.18) and (213.48,105.75) .. (213.48,108.92) .. controls (213.48,112.09) and (211.02,114.66) .. (207.98,114.66) .. controls (204.95,114.66) and (202.49,112.09) .. (202.49,108.92) -- cycle ;
\draw  [fill={rgb, 255:red, 65; green, 117; blue, 5 }  ,fill opacity=1 ] (248.58,108.92) .. controls (248.58,105.75) and (251.05,103.18) .. (254.08,103.18) .. controls (257.12,103.18) and (259.58,105.75) .. (259.58,108.92) .. controls (259.58,112.09) and (257.12,114.66) .. (254.08,114.66) .. controls (251.05,114.66) and (248.58,112.09) .. (248.58,108.92) -- cycle ;
\draw  [fill={rgb, 255:red, 65; green, 117; blue, 5 }  ,fill opacity=1 ] (62.8,109.31) .. controls (62.8,106.14) and (65.26,103.57) .. (68.3,103.57) .. controls (71.33,103.57) and (73.79,106.14) .. (73.79,109.31) .. controls (73.79,112.48) and (71.33,115.05) .. (68.3,115.05) .. controls (65.26,115.05) and (62.8,112.48) .. (62.8,109.31) -- cycle ;
\draw    (68.3,115.93) .. controls (65.08,131.99) and (161.48,282.14) .. (194.78,310.77) ;
\draw  [fill={rgb, 255:red, 208; green, 2; blue, 27 }  ,fill opacity=1 ] (189.29,310.77) .. controls (189.29,307.6) and (191.75,305.03) .. (194.78,305.03) .. controls (197.82,305.03) and (200.28,307.6) .. (200.28,310.77) .. controls (200.28,313.94) and (197.82,316.51) .. (194.78,316.51) .. controls (191.75,316.51) and (189.29,313.94) .. (189.29,310.77) -- cycle ;
\draw    (98.73,109.72) -- (207.93,289.74) ;
\draw  [fill={rgb, 255:red, 208; green, 2; blue, 27 }  ,fill opacity=1 ] (93.24,109.72) .. controls (93.24,106.55) and (95.7,103.98) .. (98.73,103.98) .. controls (101.77,103.98) and (104.23,106.55) .. (104.23,109.72) .. controls (104.23,112.9) and (101.77,115.47) .. (98.73,115.47) .. controls (95.7,115.47) and (93.24,112.9) .. (93.24,109.72) -- cycle ;
\draw    (131.27,110.16) -- (224.58,264.32) ;
\draw  [fill={rgb, 255:red, 208; green, 2; blue, 27 }  ,fill opacity=1 ] (125.78,110.16) .. controls (125.78,106.99) and (128.24,104.42) .. (131.27,104.42) .. controls (134.31,104.42) and (136.77,106.99) .. (136.77,110.16) .. controls (136.77,113.34) and (134.31,115.91) .. (131.27,115.91) .. controls (128.24,115.91) and (125.78,113.34) .. (125.78,110.16) -- cycle ;
\draw    (451.34,191.6) -- (534.64,321.36) ;
\draw    (616.2,191.6) -- (534.64,321.36) ;
\draw  [fill={rgb, 255:red, 65; green, 117; blue, 5 }  ,fill opacity=1 ] (448.18,191.6) .. controls (448.18,189.78) and (449.59,188.31) .. (451.34,188.31) .. controls (453.08,188.31) and (454.49,189.78) .. (454.49,191.6) .. controls (454.49,193.42) and (453.08,194.9) .. (451.34,194.9) .. controls (449.59,194.9) and (448.18,193.42) .. (448.18,191.6) -- cycle ;
\draw  [fill={rgb, 255:red, 65; green, 117; blue, 5 }  ,fill opacity=1 ] (613.05,191.6) .. controls (613.05,189.78) and (614.46,188.31) .. (616.2,188.31) .. controls (617.94,188.31) and (619.36,189.78) .. (619.36,191.6) .. controls (619.36,193.42) and (617.94,194.9) .. (616.2,194.9) .. controls (614.46,194.9) and (613.05,193.42) .. (613.05,191.6) -- cycle ;
\draw    (522.78,192.78) -- (567.31,269.38) ;
\draw  [fill={rgb, 255:red, 65; green, 117; blue, 5 }  ,fill opacity=1 ] (520.48,192.2) .. controls (520.48,190.38) and (521.9,188.91) .. (523.64,188.91) .. controls (525.38,188.91) and (526.79,190.38) .. (526.79,192.2) .. controls (526.79,194.02) and (525.38,195.5) .. (523.64,195.5) .. controls (521.9,195.5) and (520.48,194.02) .. (520.48,192.2) -- cycle ;
\draw  [fill={rgb, 255:red, 65; green, 117; blue, 5 }  ,fill opacity=1 ] (556.79,191.96) .. controls (556.79,190.14) and (558.2,188.67) .. (559.95,188.67) .. controls (561.69,188.67) and (563.1,190.14) .. (563.1,191.96) .. controls (563.1,193.78) and (561.69,195.26) .. (559.95,195.26) .. controls (558.2,195.26) and (556.79,193.78) .. (556.79,191.96) -- cycle ;
\draw  [fill={rgb, 255:red, 208; green, 2; blue, 27 }  ,fill opacity=1 ] (531.48,321.36) .. controls (531.48,319.54) and (532.9,318.06) .. (534.64,318.06) .. controls (536.38,318.06) and (537.79,319.54) .. (537.79,321.36) .. controls (537.79,323.18) and (536.38,324.65) .. (534.64,324.65) .. controls (532.9,324.65) and (531.48,323.18) .. (531.48,321.36) -- cycle ;
\draw    (268.33,260.33) -- (441.33,260.33) ;
\draw [shift={(443.33,260.33)}, rotate = 180] [color={rgb, 255:red, 0; green, 0; blue, 0 }  ][line width=0.75]    (10.93,-3.29) .. controls (6.95,-1.4) and (3.31,-0.3) .. (0,0) .. controls (3.31,0.3) and (6.95,1.4) .. (10.93,3.29)   ;
\draw   (475,181.33) .. controls (475,164.76) and (512.46,151.33) .. (558.67,151.33) .. controls (604.87,151.33) and (642.33,164.76) .. (642.33,181.33) .. controls (642.33,197.9) and (604.87,211.33) .. (558.67,211.33) .. controls (512.46,211.33) and (475,197.9) .. (475,181.33) -- cycle ;
\draw    (479.33,191.33) -- (549.33,299.33) ;

\draw (27.71,82.35) node [anchor=north west][inner sep=0.75pt]    {$y_{0}$};
\draw (94.7,82.53) node [anchor=north west][inner sep=0.75pt]    {$y_{1}$};
\draw (316.1,82.89) node [anchor=north west][inner sep=0.75pt]    {$y_{\ell }$};
\draw (243.87,81.03) node [anchor=north west][inner sep=0.75pt]    {$y_{j+1}$};
\draw (196.53,81.9) node [anchor=north west][inner sep=0.75pt]    {$y_{j}$};
\draw (59.38,82.35) node [anchor=north west][inner sep=0.75pt]    {$y_{J}$};
\draw (127,82.29) node [anchor=north west][inner sep=0.75pt]    {$y_2$};
\draw (155.89,102.19) node [anchor=north west][inner sep=0.75pt]    {$\cdots $};
\draw (202.38,306.63) node [anchor=north west][inner sep=0.75pt]    {$u=u_{Y^{J}}$};
\draw (443.6,166.47) node [anchor=north west][inner sep=0.75pt]    {$y_{J}$};
\draw (609.14,166.78) node [anchor=north west][inner sep=0.75pt]    {$y_{\ell }$};
\draw (548.91,166.29) node [anchor=north west][inner sep=0.75pt]    {$y_{j+1}$};
\draw (516.25,167.22) node [anchor=north west][inner sep=0.75pt]    {$y_{j}$};
\draw (522.18,326.63) node [anchor=north west][inner sep=0.75pt]    {$u=u_{Y^{J}}$};
\draw (533,124.73) node [anchor=north west][inner sep=0.75pt]    {$Y_{R}^{J}( u)$};
\draw (267,238.73) node [anchor=north west][inner sep=0.75pt]  [font=\footnotesize]  {$\textcolor[rgb]{0.05,0.07,0.76}{\{y_{j} ,y_{j+1},y_\ell\} \ \text{forms a left comb}}$};

\end{tikzpicture}

    \caption{If $\{y_0, y_2,\dots, y_j, y_{j+1},y_\ell\}$ does not form a  right comb, equivalently $\set{y_j,y_{j+1},y_\ell}$ forms a left comb, then $Y^J$ (green vertices) must form a $(1,k-1)$-split. This contradicts assumption \ref{it: bigfactb}.}
    \label{fig: bigfact}
\end{figure}
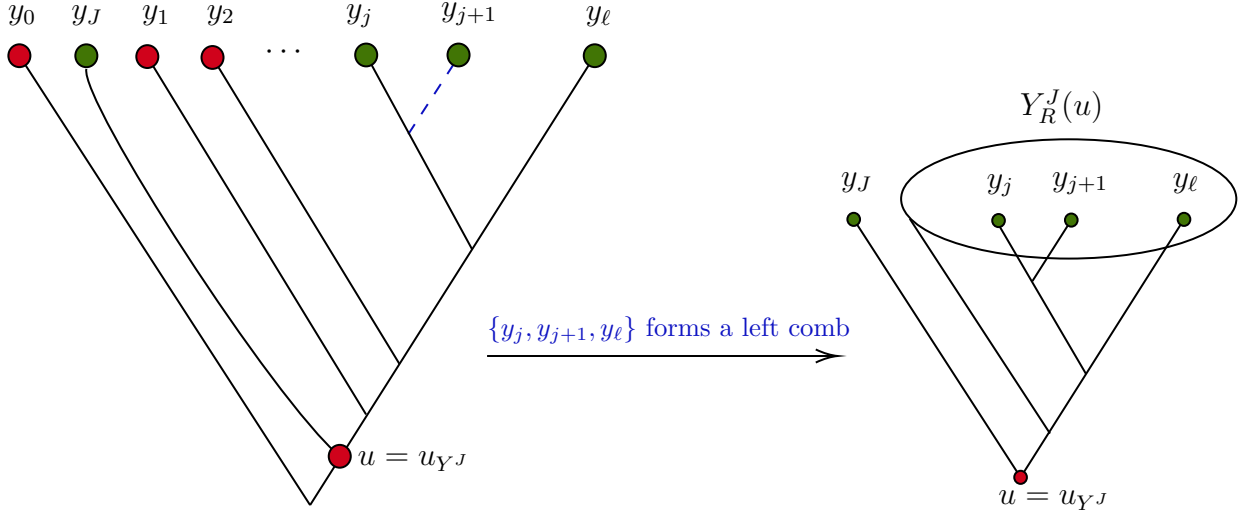
\\
Now we give a formal proof. 
    \begin{claim*}
        $Y^J_L(u)=\{y_J\}$ and $Y^J_R(u)=\{y_j:j\in J\}$.
    \end{claim*}
    \begin{proof}
        Since $u = u_{Y^J}= a(y_J,y_\ell)$, we have $y_\ell\in Y^J_R(u)$ and $y_J\in Y^J_L(u)$. We now show that for all the remaining  $i\in J\setminus\set{\ell}$, we have $y_i \in Y^J_R(u)$. 
        To this end, let $i\in J\setminus\set{\ell}$. We observe that  $\set{y_J,y_i,y_\ell}$ forms a right comb.
        Indeed, since $y_J \le y_1<y_2\le y_i$, by (\ref{eqn: mindelta}) we have:
        $$\delta(y_J,y_i) = \min\{\delta(y_J, y_{1}), \delta(y_1, y_2), \delta(y_2,y_i)\}\le \delta(y_1,y_2).$$
        On the other hand, since $y_2\le y_i\le y_\ell$, (\ref{eqn: mindelta}) implies: 
        $$\delta(y_2,y_\ell) = \min \{\delta(y_2,y_i), \delta(y_i,y_\ell)\} \le \delta(y_i,y_\ell).$$
        In view of \ref{it: bigfactb}, $\{y_1,y_2,y_\ell\}$ is a right comb, so $\delta(y_1,y_2) < \delta(y_2,y_\ell)$ and consequently 
        $$\delta(y_J,y_i)< \delta(y_i, y_\ell),$$
        equivalently, $\{y_J,y_i,y_\ell\}$ forms a right comb. Hence, $y_i\in Y^J_R(u)$ and $y_J\in Y^J_L(u)$.
    \end{proof}
    By the above claim, $|Y^J_L(u)|=1$ and $|Y^J_R(u)|=k-1$. However, $Y^J$ does not form a $(1,k-1)$-split in view of assumption \ref{it: bigfacta}. Consequently, $Y^J$ must form a right comb. In particular, $\set{y_j,y_{j+1},y_\ell}\subseteq Y^J$ forms a right comb, and hence 
    $$ \delta(y_j,y_{j+1}) < \delta(y_{j+1},y_\ell).$$
    In view of the induction hypothesis, we have that  $\set{y_0<y_2<\cdots<y_j<y_\ell}$ forms a right comb, and consequently, 
    $$\delta(y_{j-1},y_j) < \delta(y_j, y_{\ell}) = \min\{\delta(y_j,y_{j+1}), \delta(y_{j+1}, y_\l)\} = \delta(y_{j},y_{j+1}) <\delta(y_{j+1}, y_\ell). $$
   This contradicts the assumption that $\{y_0<y_2 < \cdots < y_j < y_{j+1} < y_\ell\}$ does not form a right comb.  
\end{proof}
Now we prove Proposition \ref{prop: inductionstep}. 
\begin{proof}[Proof of Proposition \ref{prop: inductionstep}]
    Fix $k\geq 4$ and let $n$ be a sufficiently large integer. Let $N \geq n$ and fix a separated set $I\in [n]^{(k-1)}$  with elements 
    $$I = \{ 1,2,a_3, \cdots , a_{k-1} \}.$$
    Let 
    \begin{align}
        \label{eqn: tI'}
        t = \frac{n}{8k}, \ \text{and fix a separated set } I' \in [t]^{(k-2)}.
    \end{align}
    Let $c:[N]^{(k-1)}\to \Z_4$ be a coloring that doesn't contain a monochromatic $\cF_{I'}^{(k-1)}(t)$ in color $\{0,1\}$ and $\rev \cF_{I'}^{(k-1)}(t)$ in $\{2,3\}$. 
    
    We will first show \ref{it1: inductionstepa} and to this end, assume for the sake of contradiction that the coloring given by $\chi_c: [2^N]^{(k)}\to \Z_4$ contains a monochromatic copy of some $(F_{\phi},<)\in \cF_{I}^{(k)}(n)$ in color $0$ or $1$. 
    
    We first describe the vertex set and edge set of $(F, <)$.  Recall from Definition \ref{def:F_I}, that the vertex set of $F$ is
    \[V(F)=\set{x_0,x_1,\ldots,x_n}\cup\set{x_J:\; J\in \{2,\dots,n\}^{(k-1)}},\]
    where we call $X=\{x_0,\ldots,x_n\}$ the \textit{distinguished set of vertices}. The ordering of vertices satisfies
    \begin{equation}
        \label{eqn: VFord}
        x_0\le \set{x_J:J\in \{2,\dots,n\}^{(k-1)}} \le x_1 < x_2<\dots< x_n.
    \end{equation}
    The edge set of $F$ is
    $$ F=\left\{X^J= \{x_J\}\cup\{x_j\}_{j\in J}:\; J\in \{2,\dots,n\}^{(k-1)}\cup\{I\}\right\}.$$
    Here $X^I = \{x_0\}\cup\{x_i\}_{i\in I}$ is the \textit{special edge}. 
    
    We now consider the cases when $(F,<)$ is in color 0 (Case I) and color 1 (Case II) separately. Each of these cases are further subdivided into sub-cases depending on the shape of the special edge $X^I$. \\
    \noindent \underline{\textbf{Case I: When $(F, < )$ is color 0}}: 
   In this case, we divide our analysis into three sub-cases based on the shape of the special edge  $X^I$, 
     $$X^I = \{x_0\}\cup\{x_i:i\in I\} =  \{x_0 < x_1< x_2< x_{a_3}< \cdots < x_{a_{k-1}}\}.$$
   Since $(F,<)$ is in color $0$, we have $\chi_c(X^I) = 0$. In view of the definition of $\chi_c$ in (\ref{eqn: coloring4}), there are 3 possible scenarios in which this may occur:
    \begin{itemize}
        \item \textbf{Case I.1:}  $X^I$ forms a \textit{balanced split}, 
         \item \textbf{Case I.2:} $X^I$ forms a \textit{right comb} and $c(\pi(X^I))= 0$,
         \item \textbf{Case I.3:} $X^I$ forms a \textit{left comb} and $c(\pi(X^I)) = 3$.
         
    \end{itemize}
    We now analyze each of these cases separately. As discussed in the outline of the proof at the beginning of Subsection \ref{subsec:orderedstepup}, we divide each case into two parts, and   arrive at a contradiction to \ref{it1: inductionstepi} or \ref{it2: inductionstepii}. 
    \\
    
    \noindent\underline{\textbf{Case I.1: $X^I$ forms a balanced split:}} In this case, let $u = u_{X^I} = a(x_0, x_{a_{k-1}})$. Since $X^I$ forms a balanced split, we must have  $|X_L^I(u)|\ge 2$ and $ |X_R^I(u)|\ge 2$. Consequently, 
    \begin{align}
    \label{eqn: caseI.1split}
        \{x_0,x_1\}\subseteq X^I_L(u) \quad \text{ and } \quad \{x_{a_{k-2}},x_{a_{k-1}}\}\subseteq X_R^I(u).
    \end{align}
First we prove Claim \ref{claim: CaseI.1St}, where we find a large right comb $Y\subseteq V(F)$.
    \begin{claim}
    \label{claim: CaseI.1St}
        The set $Y:= \set{x_0,x_{a_{k-2}},\ldots,x_{a_{k-1}}}$ forms a right comb.
    \end{claim}
    \begin{proof}
    Let
     $$y_0=x_0, \,y_1=x_1  \text{ and } y_2=x_{a_{k-2}},\; y_3 = x_{a_{k-2}+1},\;\ldots,\; y_{\ell-1} = x_{a_{k-1}-1},\; y_{\ell} = x_{a_{k-1}},$$
     i.e., $y_2$ through $y_\ell$ is the interval between $x_{a_{k-2}}$ and $x_{a_{k-1}}$.  We now apply Fact \ref{fact: bigfact} on $\{y_0,y_1,\dots, y_\ell\}$. To this end, we verify its assumptions. Note that $\ell = a_{k-1} - a_{k-2}+2$ and
        since $I$ is separated and $n_k$ is chosen to be sufficiently large, $a_{k-1} - a_{k-2}+2 \ge n/2k \ge k$. Consequently, $\ell \ge k$ as required in Fact \ref{fact: bigfact}.  
        \begin{itemize}
            \item Condition \ref{it: bigfactb} follows from (\ref{eqn: caseI.1split}), i.e.,  $x_1\in X^I_L(u)$ and $\{x_{a_{k-2}},x_{a_{k-1}}\}\subseteq X_R^I(u)$ and hence $\{x_1, x_{a_{k-2}}, x_{a_{k-1}}\}$ forms a right comb.
            \item  Condition \ref{it: bigfacta} holds since $(F, <)$ is monochromatic in color 0. Indeed for all ~$J\in \set{a_{k-2},\ldots,a_{k-1}}^{(k-1)}$, and in view of (\ref{eqn: VFord}) we have $x_0\le x_J\le x_1$. The edge $X^J= \set{x_J}\cup \set{x_j:j\in J}$ is of color 0 and hence does not form a $(1,k-1)$-split.
        \end{itemize}   
         Consequently, Fact \ref{fact: bigfact} implies that $Y$ forms a right comb.
    \end{proof}
    Recall that $t = n/8k$ and let $\tilde{Y} = \{x_0, x_{a_{k-2}}, \dots, x_{a_{k-2}+t-1}, x_{a_{k-1}}\}$. Thus $\tilde{Y}$ is a $(t+2)$ element subset of $Y$ containing the first and last elements of $Y$. Since $Y$ forms a right comb, so does $\tilde{Y}$. We will now apply Proposition \ref{prop: rightprojection} to obtain a member of the forbidden family in the projection. 

    \begin{figure}[h]
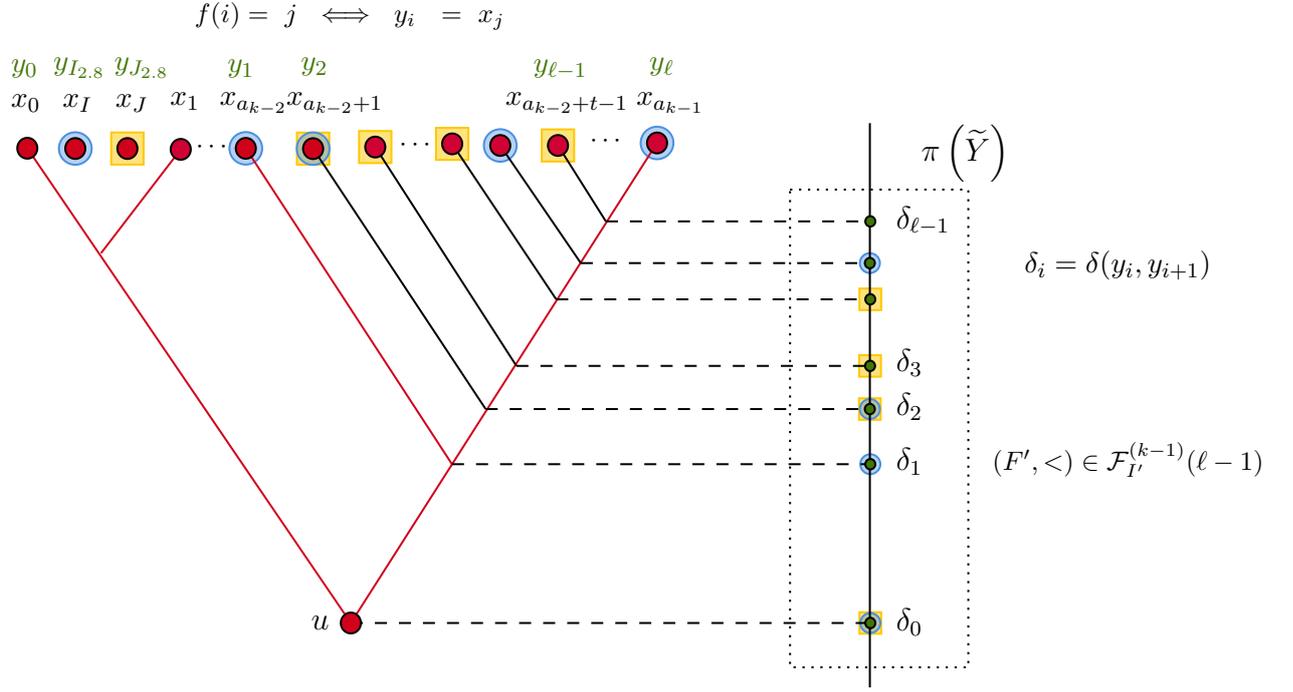

        \centering

\tikzset{every picture/.style={line width=0.75pt}} 


        \caption{Case I.1: The set $\tilde{Y}$ projects to a copy of $(F',<)\in \cF_{I'}^{(k-1)}(t)$}
        \label{fig: CaseI.1}
    \end{figure}
    
    \begin{claim}
    \label{claim: CaseI.1Proj}
        $c^{-1}(0)\subseteq [N]^{(k-1)}$ contains a monochromatic copy of $(F', <)\in \cF_{I'}^{(k-1)}(t)$. 
    \end{claim}
    \begin{proof}
    With the goal of applying Proposition \ref{prop: rightprojection}, we let $\ell = t+1$ and  let $$f:[t+1]=[\ell] \to \{a_{k-2}, \dots, a_{{k-2}+t-1}, a_{k-1}\}$$ be the increasing bijection, in particular, $f(1) = a_{k-2}$ and $ f(\ell) = a_{k-1}.$  Let
        $$y_0 = x_0 \quad \text{ and } \quad y_i = x_{f(i)} \text{ for all $1\le i\le \ell$.}$$
        Recall that $I'\in [\ell -1]^{(k-2)} = [t]^{(k-2)}$ was fixed at the beginning of the proof. Let 
        $$I_{\ref{prop: rightprojection}} = I' \cup \{\ell\}\in [\ell]^{(k-1)}.$$
       
        We will now apply Proposition \ref{prop: rightprojection} with $Y= \{y_0,\dots, y_\ell\}$, and $N$ and $c:[N]^{(k-1)}\to \ZZ_4$ as in the statement of Proposition \ref{prop: inductionstep}. To this end, we need to verify conditions \ref{it: righta}--\ref{it: rightc} of Proposition \ref{prop: rightprojection}:
        \begin{itemize}
            \item  In view of Claim \ref{claim: CaseI.1St}, the set $\tilde{Y} = \{y_0, y_1,\dots, y_\ell\}$ as defined above is a right comb, and hence \ref{it: righta} holds.
            \item 
        Given any $J_{\ref{prop: rightprojection}}\in \{2,\dots, \ell\}^{(k-1)}\cup \{I_{\ref{prop: rightprojection}}\}$, let $J = f(J_{\ref{prop: rightprojection}})$. In other words, $J\in \{a_{k-2},a_{k-2}+1,\dots, a_{k-2}+t-1, a_{k-1}\}^{(k-1)}$.   
        Let $y_{J_{\ref{prop: rightprojection}}}= x_{J}$. In view of (\ref{eqn: VFord}), we have  $$y_0 = x_0 \le y_{J_{\ref{prop: rightprojection}}} = x_J\leq x_1 \le y_1,$$
        assumption \ref{it: rightb} holds. See Figure \ref{fig: CaseI.1}.
        \item   To observe that \ref{it: rightc} holds, i.e., $a(y_{I_{\ref{prop: rightprojection}}}, y_1) = u_Y = a(y_0,y_\ell)$, see Figure \ref{fig: CaseI.1}.
        \end{itemize}
        Now we verify that $(\star)$ in Proposition \ref{prop: rightprojection} holds. 
       Fix $J_{\ref{prop: rightprojection}}\in \{2,\dots,\ell\}^{(k-1)}\cup\{I_{\ref{prop: rightprojection}}\}$ where $J_{\ref{prop: rightprojection}}= \{i_1,\dots, i_{k-1}\}$. Let $J = \{f(i_1), \dots, f(i_{k-1})\} = f(J_{\ref{prop: rightprojection}}).$ By the assumption of Case I:  
        $$\chi_c(Y^J) = \chi_c( \{y_{J_{\ref{prop: rightprojection}}}\}\cup \{y_j:j\in J_{\ref{prop: rightprojection}}\}) = \chi_c(\{x_J\}\cup \{x_{j}:j\in J\}) =0.$$
        
        Consequently, $c^{-1}(0)$ contains a monochromatic copy of $(F',<)$ in the family $\cF_{I'}^{(k-1)}(\ell-1) = \cF_{I'}^{(k-1)}(t)$ where $I' = I_{\ref{prop: rightprojection}}\setminus \max I_{\ref{prop: rightprojection}}$.
    \end{proof}
  The conclusion of Claim \ref{claim: CaseI.1Proj} contradicts assumption  \ref{it1: inductionstepi} in Proposition \ref{prop: inductionstep}. This completes Case I.1.
\\[0.3cm]
\noindent\underline{\textbf{Case I.2: $X^I$ forms a right comb:}} In this case, we focus on the following four element subset of $X^I$. Consider $\{x_0< x_1< x_2< x_{a_{k-1}}\}\subseteq X^I$, which also forms a right comb. The rest of the analysis of Case I.2 proceeds  similarly to Case I.1. 
\begin{remark}
    \label{rem: notofcolor2}
    We assume here that $(F,<)$ is in color 0 and hence not in color 2. In the rest of the case we only use that $(F, <)$ is not in color 2. 
\end{remark}

    \begin{claim}
    \label{claim: CaseI.3St}
        The set $Y = \set{x_0,x_1, x_2,\ldots,x_{a_{k-1}}}$ forms a right comb.
    \end{claim}
    \begin{proof}
      For $\ell = a_{k-1}$, let
        $$y_0=x_0,\; y_1=x_1,\; y_2 = x_2,\; \dots,\; y_\ell = x_{a_{k-1}}.$$
        We now apply Fact \ref{fact: bigfact} with $\{y_0,y_1,\dots, y_\ell\}$. To this end, we verify its assumptions.   
       Since the set $I$ is a separated subset of $[n]^{(k-1)}$ and $n$ is chosen to be large enough, $\ell = a_{k-1}\ge n/2k > k$.  
        \begin{itemize}
            \item Condition \ref{it: bigfactb} holds since the set  $\set{x_1,x_2,x_{a_{k-1}}}$ is a subset of the right comb $X^I$ and thus itself also forms a right comb.
            \item  Condition \ref{it: bigfacta} holds since $(F, <)$ is monochromatic not in color 2. Indeed, for all $J\in \set{2,\ldots,a_{k-1}}^{(k-1)}$, the leaf $y_J= x_J$ satisfies $y_0 =x_0\le y_J= x_J\le x_1=y_1$. The edge $X^J= \set{x_J}\cup \set{x_j:j\in J}$ is not in color 2, and hence, in view of (\ref{eqn: coloring4}), does not form a $(1,k-1)$-split. 
        \end{itemize} 
        Thus Fact \ref{fact: bigfact} implies that $\set{x_0,x_2,\ldots,x_{a_{k-1}}}=Y\setminus \{x_1\}$ forms a right comb. Since $\set{x_0,x_1,x_2,x_{a_{k-1}}}$ is a subset of $X^I$, it forms a right comb, and consequently so does $Y = \set{x_0,x_1,x_2,\ldots,x_{a_{k-1}}}$.
    \end{proof}
Recall that $t = n/8k$ and let
\begin{align}
\label{eqn: CaseI.3l,t}
    t' = a_{k-1}-1 \text{ and } I'' = I\setminus \max I \in [t']^{(k-2)}.
\end{align}

\begin{figure}[h]
        \centering

\tikzset{every picture/.style={line width=0.75pt}} 


        \caption{Case I.2: The set $Y$ projects to a copy of $(F'',<)\in \cF_{I''}^{(k-1)}(t')$}
        \label{fig: CaseI.2}
    \end{figure}

\begin{claim}
\label{claim: CaseI.3Proj}
       $c^{-1}(0)\subseteq [N]^{(k-1)}$ contains a monochromatic copy of a member of $\cF_{I''}^{(k-1)}(t')$.
    \end{claim}
\begin{proof}
     Let  
        $$y_0 = x_0,\;  y_1 = x_{1},\; y_2 = x_{2},\; \cdots,\; y_\ell = x_{a_{k-1}}.$$
         We apply Proposition \ref{prop: rightprojection} with $Y=\{y_0,\dots,y_\ell\}$ and with $N$, $k$ and $c$ as in the statement of Proposition \ref{prop: inductionstep}. To this end, we verify conditions \ref{it: righta}--\ref{it: rightc} of Proposition \ref{prop: rightprojection}. 
        \begin{itemize}
            \item  In view of Claim \ref{claim: CaseI.3St}, the set $Y = \{y_0, y_1,\dots, y_\ell\}$ as defined above is a right comb and thus \ref{it: righta} holds.
            \item    Let $I_{\ref{prop: rightprojection}} = I \in [\ell]^{(k-1)}$,   $J_{\ref{prop: rightprojection}} = J\in \{2,\dots, \ell\}^{(k-1)}\cup \{I\}$,  and let $y_{J_{\ref{prop: rightprojection}}}=x_J$. In view of (\ref{eqn: VFord}), we have  $$y_0 = x_0 \le y_{J_{\ref{prop: rightprojection}}} = x_J\leq x_1 \le y_1,$$
        so assumption \ref{it: rightb} holds. 
            \item     We now verify that \ref{it: rightc} holds, i.e., $a(y_{I_{\ref{prop: rightprojection}}}, y_1) = u_Y = a(y_0,y_\ell).$ Indeed, since $I_{\ref{prop: rightprojection}} = I$, we have that $y_{I_{\ref{prop: rightprojection}}} = x_0$, and so we need to show that $a(y_0, y_1) = u_Y = a(y_0,y_\ell)$. This is true since $\{y_0, y_1, y_\ell\}\subseteq X^I$ forms a right comb.         
        \end{itemize}
        The condition $(\star)$ holds due to the assumption in Case I that $(F_{\phi},<)$ is monochromatic in color 0. 
        Consequently, we have that $c^{-1}(0)$ contains a member of   $\cF_{I''}^{(k-1)}(\ell-1) = \cF_{I''}^{(k-1)}(t')$ where $I''$ and $t'$ as defined in (\ref{eqn: CaseI.3l,t}). 
\end{proof}
Claim \ref{claim: CaseI.3Proj} above implies that the subgraph in color $0$ in the projection, i.e., $c^{-1}(0)\subseteq [N]^{(k-1)}$, contains some $(F'',<)\in \cF_{I''}^{(k-1)}(t')$. Denote the distinguished set of vertices of $V(F'')$ by $\{\delta_0, \delta_1,\dots ,\delta_{t'}\}$.  Recall that in (\ref{eqn: tI'}) we fixed a separated set $I'\in [t]^{(k-2)}$. We claim that $(F'', <)$ contains a member of $\cF_{I'}^{(k-1)}(t)$ as an ordered subgraph, thus contradicting \ref{it1: inductionstepi} of Proposition \ref{prop: inductionstep}.

Recall that $I=\{1,2,a_{3}, a_4, \dots, a_{k-2},  a_{k-1}\}$ is a separated set. Consequently, for every $2\le i\le k-2$, we have: 
$$a_{i+1}- a_i \ge \frac{n}{2k} = 4t.$$
Hence, the set $I'' = \{1,2,a_{3}, \dots, a_{k-2}\}$ satisfies $a_{i+1}- a_i \ge n/2k = 4t$ for all $2\le i\le k-3$.  

Since $a_{i+1}-a_i\ge 4t$ for every $2\le i\le k-3$, by removing appropriate vertices of $F''$ between $\delta_{a_i}$ and $\delta_{a_{i+1}}$ for each $2\le i\le k-3$, we are left with $(F',<)\in \cF_{I'}^{(k-1)}(t)$ as an ordered subgraph. 

In particular, for $I' = \{b_1=1< b_2 = 2< b_3 <\dots < b_{k-2}\}$, 
$V(F')$ will have distinguished vertex set $Z\subseteq \{\delta_0,\delta_1,\dots, \delta_{t'}\}$, with $|Z| = t$ and $Z$ containing the set
$$\{\delta_0,\delta_1, \delta_2, \delta_{a_3}, \delta_{a_4},\dots, \delta_{a_{k-2}}\},$$ 
where $\delta_{a_i}$ is at the $(b_i+1)^{\text{st}}$ position in $Z$.

However, this implies that $c^{-1}(0)\subseteq [N]^{(k-1)}$ contains a member of $\cF_{I'}^{(k-1)}(t)$ which contradicts \ref{it1: inductionstepi} of Proposition \ref{prop: inductionstep}. This completes the analysis of Case I.2. 
\\[0.3cm]
    \noindent\underline{\textbf{Case I.3: $X^I$ forms a left comb:}} In this case, we focus on the subset $\{x_0< x_1< x_2< x_{a_3}\}$ of $X^I$, which also forms a left comb. Before we formally analyze this case, we provide a brief summary. Unlike in the previous cases, here in Claim \ref{claim: CaseI.2St}, we will find $Z\subseteq V(F)$ with  $x_{k-1}\le Z\le x_{a_{3-1}}$ which forms either a left comb or a right comb. In particular, we will show that if $x_m$ is the least vertex such that $\{x_1, x_m, x_{a_3}\}$ is a right comb, then $\{x_{k-1},\dots, x_m\}$ is a left comb and $\{x_m,\dots, x_{a_3-1}\}$ is a right comb. Depending on whether $x_m$ is closer to $x_{k-1}$ or $x_{a_3}$, we obtain a large right comb or left comb respectively. 
    
    We recall that $n$ is chosen to be sufficiently large, $I = \{1,2,a_3, \dots, a_{k-1}\}\in [n]^{(k-1)}$ and $t = n/8k$ and thus the following relation holds: 
        \begin{align}
        \label{eqn: CaseI.2St1}
            a_3 -a_2 = a_3-2 \geq \frac{n}{2k} = 4t > 4k.
        \end{align}
    \begin{claim}
    \label{claim: CaseI.2St}
        There exists an integer $m$ with $k-1\le m\le a_3$ such that the following holds: 
        \begin{enumerate}[label = (\roman*)]
            \item\label{it: CaseI.2Sti} If $m < a_3-t$, then $\{x_0\} \dcup \{x_m,\dots, x_{a_{3}-1}, x_{a_3}\}$ forms a right comb.
             \item\label{it: CaseI.2Stii} If $m\ge a_3 -t$, then $\{x_0, x_1, x_2,x_{k-1}, x_k,\dots, x_{m-1}\} \dcup \{x_{a_3}\}$ forms a left comb. 
        \end{enumerate}
    \end{claim}
    \begin{proof}
        Let $x_m\in \{x_{k-1}<\cdots < x_{a_3-1}\}$ be the least vertex such that $\set{x_1,x_m,x_{a_3}}$ forms a right comb. If no such $x_m$ exists, we set $x_m=x_{a_3}$.

        First, we consider the case where $m < a_3-t$. Let   $Z=\set{x_m,\ldots,x_{a_3-1}}$ and note that $|Z| = a_3 - m>t$. We will show that in this case $\{x_0\} \dcup Z\dcup  \{x_{a_3}\}$ forms a right comb. To this end, let
        $$y_0=x_0,\; y_1=x_1 \quad  \text{and} \quad y_2 = x_m,\; y_3 = x_{m+1},\; \dots,\; y_\ell = x_{a_3}.$$ 
        We now apply Fact \ref{fact: bigfact} with $Y = \{y_0,y_1,\dots,y_\ell\}.$ For this we need to verify its assumptions. Note that in view of (\ref{eqn: CaseI.2St1}) and the assumption that $m< a_3-t$, we have $\ell = a_3 - m +3> t+3\geq k$. 
        \begin{itemize}
            \item Condition \ref{it: bigfactb} holds, since $\set{x_1,x_m,x_{a_3}}$ forms a right comb by the definition of $x_m$.  
            \item   Condition \ref{it: bigfacta} follows from  $(F, <)$ being monochromatic in color 0. Indeed, for all $J\in \set{m,\ldots,a_3}^{(k-1)}$, $x_J$ satisfies $x_0\le x_J\le x_1$ and $X^J= \set{x_J}\cup \set{x_j:j\in J}$ is of color 0, so in view of (\ref{eqn: coloring4}), does not form a $(1,k-1)$-split. 
        \end{itemize}
        Consequently, Fact \ref{fact: bigfact} implies that $\set{x_0,x_m,\ldots,x_{a_3}}$ forms a right comb, proving \ref{it: CaseI.2Sti}. 
\\[0.3cm]
        Next, we consider the case when $m\ge a_3 -t$. Let $Z=\set{x_{k-1},\ldots,x_{m-1}}$. 
        We will now show that for $k-1\le \ell\le m$, the set
        $$\set{x_0,x_1,x_2}\dcup \{x_{k-1}, \dots, x_{\l}\}\dcup \set{x_{a_3}}$$ 
        forms a left comb, by induction on $\l$. Recall that by the assumption of Case I.3, we have that  $\set{x_0,x_1,x_2,x_{a_3}}\subseteq X^I$ forms a left comb. 
        
      We first show the base case $\ell = k-1$, i.e., that the set $\set{x_0,x_1,x_2,x_{k-1},x_{a_3}}$ forms a left comb. To this end, consider  the $(k-1)$-set $J=\set{2,\ldots,k-1,a_3}$. First, we claim that the set $X^J = \{x_J\}\cup \{x_j:j\in J\}$ forms a left comb. Since $\{x_0, x_1, x_2, x_{a_3}\}\subseteq X^I$ forms a left comb and $x_0 \le x_J \le x_1$, we have that for $u = a(x_0,x_{a_3})$:
      $$\{x_0, x_J, x_1\} \subseteq L(u) \text{ and } x_{a_3} \in R(u).$$
      This implies $u = a(x_1, x_{a_3})$ as well. By definition of $m$,
      Since we assume in this case that $m\ge a_3 -t$, we have that $k-1<m$. So by the definition of $m$, $\set{x_1,x_{k-1},x_{a_3}}$ forms a left comb as well. This implies that $x_{k-1}\in L(u)$. Consequently, we obtain that $\{x_J<x_2 < \cdots < x_{k-1}\}\subseteq L(u)$ while $x_{a_3}\in R(u)$, and thus
        \[|X^J_L(u)|=k-1 \quad \text{and} \quad |X^J_R(u)|=1.\]
        Since $X^J$ has color 0, by (\ref{eqn: coloring4}) it cannot form a $(k-1,1)$-split and so $X^J$ must form a left comb. In particular, $\set{x_J,x_{2},x_{k-1},x_{a_3}}\subseteq X^J$ forms a left comb, or equivalently,
        $$\delta(x_J,x_2)>\delta(x_2, x_{k-1}) > \delta(x_{k-1}, a_3).$$
        Since $x_J\le x_1 < x_2$, by (\ref{eqn: mindelta}) we have that $\delta(x_1, x_2)\ge \delta(x_J, x_2)$. 
        By assumption, the set $\set{x_0,x_1,x_2}\subseteq X^I$ forms a left comb as well, i.e., $\delta(x_0, x_1)>\delta(x_1,x_2).$
        Consequently,
        $$\delta(x_0,x_1) > \delta(x_1,x_2)> \delta(x_2,x_{k-1}) > \delta(x_{k-1}, x_{a_3}),$$ 
        or equivalently,  $\set{x_0,x_1,x_2,x_{k-1}, x_{a_3}}$ also forms a left comb. This completes the base case. 
        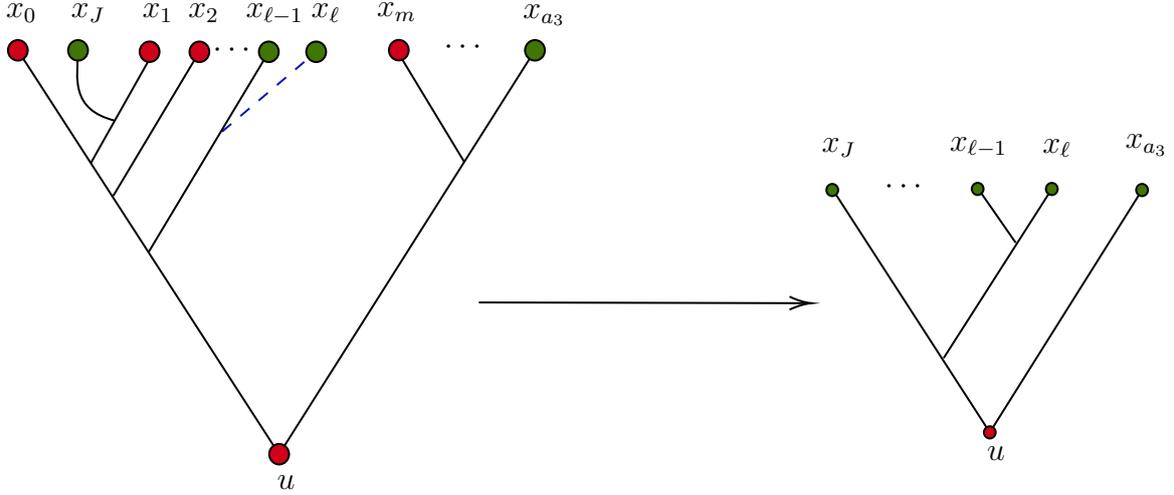
\begin{figure}[h!]
\centering
\tikzset{every picture/.style={line width=0.75pt}} 

\begin{tikzpicture}[x=0.75pt,y=0.75pt,yscale=-1,xscale=1]

\draw    (50.21,123.52) -- (180.63,326.66) ;
\draw    (308.32,123.52) -- (180.63,326.66) ;
\draw    (117.53,125.02) -- (86.49,180.48) ;
\draw  [fill={rgb, 255:red, 208; green, 2; blue, 27 }  ,fill opacity=1 ] (175.69,326.66) .. controls (175.69,323.81) and (177.9,321.5) .. (180.63,321.5) .. controls (183.35,321.5) and (185.57,323.81) .. (185.57,326.66) .. controls (185.57,329.5) and (183.35,331.81) .. (180.63,331.81) .. controls (177.9,331.81) and (175.69,329.5) .. (175.69,326.66) -- cycle ;
\draw  [fill={rgb, 255:red, 208; green, 2; blue, 27 }  ,fill opacity=1 ] (45.28,123.52) .. controls (45.28,120.67) and (47.49,118.36) .. (50.21,118.36) .. controls (52.94,118.36) and (55.15,120.67) .. (55.15,123.52) .. controls (55.15,126.37) and (52.94,128.68) .. (50.21,128.68) .. controls (47.49,128.68) and (45.28,126.37) .. (45.28,123.52) -- cycle ;
\draw  [fill={rgb, 255:red, 208; green, 2; blue, 27 }  ,fill opacity=1 ] (111.01,124.23) .. controls (111.01,121.38) and (113.22,119.07) .. (115.95,119.07) .. controls (118.68,119.07) and (120.89,121.38) .. (120.89,124.23) .. controls (120.89,127.08) and (118.68,129.39) .. (115.95,129.39) .. controls (113.22,129.39) and (111.01,127.08) .. (111.01,124.23) -- cycle ;
\draw  [fill={rgb, 255:red, 65; green, 117; blue, 5 }  ,fill opacity=1 ] (303.38,123.52) .. controls (303.38,120.67) and (305.59,118.36) .. (308.32,118.36) .. controls (311.05,118.36) and (313.26,120.67) .. (313.26,123.52) .. controls (313.26,126.37) and (311.05,128.68) .. (308.32,128.68) .. controls (305.59,128.68) and (303.38,126.37) .. (303.38,123.52) -- cycle ;
\draw [color={rgb, 255:red, 14; green, 18; blue, 193 }  ,draw opacity=1 ] [dash pattern={on 4.5pt off 4.5pt}]  (152.18,164.4) -- (199.11,124.26) ;
\draw    (238.81,122.66) -- (273.04,179.63) ;
\draw  [fill={rgb, 255:red, 208; green, 2; blue, 27 }  ,fill opacity=1 ] (235.44,123.45) .. controls (235.44,120.6) and (237.65,118.29) .. (240.38,118.29) .. controls (243.11,118.29) and (245.32,120.6) .. (245.32,123.45) .. controls (245.32,126.3) and (243.11,128.6) .. (240.38,128.6) .. controls (237.65,128.6) and (235.44,126.3) .. (235.44,123.45) -- cycle ;
\draw    (80.25,123.54) .. controls (79.72,137.62) and (76.57,153.37) .. (98.63,158.88) ;
\draw  [fill={rgb, 255:red, 65; green, 117; blue, 5 }  ,fill opacity=1 ] (75.31,123.54) .. controls (75.31,120.69) and (77.52,118.38) .. (80.25,118.38) .. controls (82.98,118.38) and (85.19,120.69) .. (85.19,123.54) .. controls (85.19,126.39) and (82.98,128.7) .. (80.25,128.7) .. controls (77.52,128.7) and (75.31,126.39) .. (75.31,123.54) -- cycle ;
\draw    (140.83,124.26) -- (97.84,196.69) ;
\draw  [fill={rgb, 255:red, 208; green, 2; blue, 27 }  ,fill opacity=1 ] (135.9,124.26) .. controls (135.9,121.41) and (138.11,119.1) .. (140.83,119.1) .. controls (143.56,119.1) and (145.77,121.41) .. (145.77,124.26) .. controls (145.77,127.11) and (143.56,129.42) .. (140.83,129.42) .. controls (138.11,129.42) and (135.9,127.11) .. (135.9,124.26) -- cycle ;
\draw  [fill={rgb, 255:red, 65; green, 117; blue, 5 }  ,fill opacity=1 ] (194.17,124.26) .. controls (194.17,121.41) and (196.38,119.1) .. (199.11,119.1) .. controls (201.84,119.1) and (204.05,121.41) .. (204.05,124.26) .. controls (204.05,127.11) and (201.84,129.42) .. (199.11,129.42) .. controls (196.38,129.42) and (194.17,127.11) .. (194.17,124.26) -- cycle ;
\draw    (175.49,124.26) -- (115.42,225.09) ;
\draw  [fill={rgb, 255:red, 65; green, 117; blue, 5 }  ,fill opacity=1 ] (170.55,124.26) .. controls (170.55,121.41) and (172.76,119.1) .. (175.49,119.1) .. controls (178.21,119.1) and (180.43,121.41) .. (180.43,124.26) .. controls (180.43,127.11) and (178.21,129.42) .. (175.49,129.42) .. controls (172.76,129.42) and (170.55,127.11) .. (170.55,124.26) -- cycle ;
\draw    (280.33,250.33) -- (444.33,250.71) ;
\draw [shift={(446.33,250.71)}, rotate = 180.13] [color={rgb, 255:red, 0; green, 0; blue, 0 }  ][line width=0.75]    (10.93,-3.29) .. controls (6.95,-1.4) and (3.31,-0.3) .. (0,0) .. controls (3.31,0.3) and (6.95,1.4) .. (10.93,3.29)   ;
\draw    (457.75,194.78) -- (535.37,315.69) ;
\draw    (611.37,193.78) -- (535.37,315.69) ;
\draw  [fill={rgb, 255:red, 208; green, 2; blue, 27 }  ,fill opacity=1 ] (532.43,315.69) .. controls (532.43,314) and (533.75,312.62) .. (535.37,312.62) .. controls (536.99,312.62) and (538.31,314) .. (538.31,315.69) .. controls (538.31,317.39) and (536.99,318.76) .. (535.37,318.76) .. controls (533.75,318.76) and (532.43,317.39) .. (532.43,315.69) -- cycle ;
\draw  [fill={rgb, 255:red, 65; green, 117; blue, 5 }  ,fill opacity=1 ] (453.81,193.78) .. controls (453.81,192.09) and (455.13,190.71) .. (456.75,190.71) .. controls (458.37,190.71) and (459.69,192.09) .. (459.69,193.78) .. controls (459.69,195.48) and (458.37,196.85) .. (456.75,196.85) .. controls (455.13,196.85) and (453.81,195.48) .. (453.81,193.78) -- cycle ;
\draw  [fill={rgb, 255:red, 65; green, 117; blue, 5 }  ,fill opacity=1 ] (608.43,193.78) .. controls (608.43,192.09) and (609.75,190.71) .. (611.37,190.71) .. controls (613,190.71) and (614.31,192.09) .. (614.31,193.78) .. controls (614.31,195.48) and (613,196.85) .. (611.37,196.85) .. controls (609.75,196.85) and (608.43,195.48) .. (608.43,193.78) -- cycle ;
\draw    (530.31,194.22) -- (548.33,220.33) ;
\draw  [fill={rgb, 255:red, 65; green, 117; blue, 5 }  ,fill opacity=1 ] (526.37,193.22) .. controls (526.37,191.53) and (527.69,190.15) .. (529.31,190.15) .. controls (530.93,190.15) and (532.25,191.53) .. (532.25,193.22) .. controls (532.25,194.92) and (530.93,196.3) .. (529.31,196.3) .. controls (527.69,196.3) and (526.37,194.92) .. (526.37,193.22) -- cycle ;
\draw    (566.33,193.26) -- (512.33,278.33) ;
\draw  [fill={rgb, 255:red, 65; green, 117; blue, 5 }  ,fill opacity=1 ] (563.39,193.26) .. controls (563.39,191.57) and (564.71,190.19) .. (566.33,190.19) .. controls (567.96,190.19) and (569.27,191.57) .. (569.27,193.26) .. controls (569.27,194.96) and (567.96,196.33) .. (566.33,196.33) .. controls (564.71,196.33) and (563.39,194.96) .. (563.39,193.26) -- cycle ;

\draw (178.21,335.88) node [anchor=north west][inner sep=0.75pt]    {$u$};
\draw (42.91,98.43) node [anchor=north west][inner sep=0.75pt]    {$x_{0}$};
\draw (110.99,98.6) node [anchor=north west][inner sep=0.75pt]    {$x_{1}$};
\draw (301.02,98.81) node [anchor=north west][inner sep=0.75pt]    {$x_{a_{3}{}}$};
\draw (261.45,116.8) node [anchor=north west][inner sep=0.75pt]    {$\cdots $};
\draw (227.96,98.04) node [anchor=north west][inner sep=0.75pt]    {$x_{m}$};
\draw (75.41,98.43) node [anchor=north west][inner sep=0.75pt]    {$x_{J}$};
\draw (133.83,98.6) node [anchor=north west][inner sep=0.75pt]    {$x_{2}$};
\draw (162.58,98.6) node [anchor=north west][inner sep=0.75pt]    {$x_{\l-1}$};
\draw (195.32,98.6) node [anchor=north west][inner sep=0.75pt]    {$x_{\l}$};
\draw (146.11,118.16) node [anchor=north west][inner sep=0.75pt]    {$\cdots $};
\draw (532.5,321.32) node [anchor=north west][inner sep=0.75pt]    {$u$};
\draw (450.18,166.37) node [anchor=north west][inner sep=0.75pt]    {$x_{J}$};
\draw (601.56,164.4) node [anchor=north west][inner sep=0.75pt]    {$x_{a_{3}{}}$};
\draw (513.96,165.47) node [anchor=north west][inner sep=0.75pt]    {$x_{\l-1}$};
\draw (560.29,166.47) node [anchor=north west][inner sep=0.75pt]    {$x_{\l}$};
\draw (481.45,186.8) node [anchor=north west][inner sep=0.75pt]    {$\cdots $};
\draw (295,217) node [anchor=north west][inner sep=0.75pt]   [align=left] {};

\end{tikzpicture}
            
            \caption{If $x_\l$ is as shown in the figure, then $X^J$ forms a $(k-1,1)$-split, contradicting that $X^J$ has color 0.}
            \label{fig: casei.2}
        \end{figure}
         
        Next, we continue with the induction step. Here we only provide a ``proof by picture'', omitting the formal details, since they follow similarly to the base case above. Assume that for some $\ell$ with $k \le \l< m$, the set
        $$\{x_0\}\dcup \{x_1,x_2,x_{k-1},\ldots,x_{\l-1}\}\dcup \{x_{a_3}\}$$ forms a left comb.   If the above set together with $x_\l$, i.e., 
        $$\{x_0\}\dcup \{x_1,x_2,x_{k-1},\ldots,x_{\l-1},x_\l\}\dcup \{x_{a_3}\}$$ 
         does not form a left comb (as shown in Figure \ref{fig: casei.2}), then it contradicts the assumption that $(F,<)$ is monochromatic in color 0.  Indeed, consider the $(k-1)$-set $J=\set{3,\ldots,k-2,\l-1, \l,a_3}$ and let $u = a(x_0, x_{a_3})$. Note that since $\l < m$, the set $\set{x_1,x_\l,x_{a_3}}$ forms a left comb by definition of $m$. Consequently, as shown in the picture, $x_{\l} \in L(u)$. Observe that the set  $X^J$  (the green vertices in Fig. \ref{fig: casei.2}) must then form a $(k-1,1)$-split, which  in view of (\ref{eqn: coloring4}) contradicts that  $\chi_c(X^J)=0$.

    \end{proof}
 We will now use Claim \ref{claim: CaseI.2St} together with Propositions \ref{prop: leftcomb} and \ref{prop: rightprojection} to find a member of the forbidden family in the projection. Note that as a consequence of Claim \ref{claim: CaseI.2St} above, we obtain a set $Z\subseteq \{x_{k-1}, \dots, x_{a_{3}-1}\}$ of size $t$ such that either:
 \begin{enumerate}
     \item $\{x_0\} \dcup Z\dcup \{x_{a_3}\}$ forms a right comb, or
     \item $\{x_0,x_1,x_2\} \dcup Z \dcup \{x_{a_3}\}$ forms a left comb. 
 \end{enumerate}
 Once again, recall that we set $t = n/8k$  and that we were given $I'\in [t]^{(k-2)} $ in (\ref{eqn: tI'}). 
    \begin{figure}[h]
        \centering

\tikzset{every picture/.style={line width=0.75pt}} 


        \caption{Case I.3: Here we consider the case when we find a large right comb. This gives a monochromatic copy of $(F',<)\in \cF_{I'}^{(k)}(t)$ in color $0$ in the projection.}
        \label{fig: CaseI.3Right}
    \end{figure}

    \begin{figure}[h]
        \centering

\tikzset{every picture/.style={line width=0.75pt}} 

%
        \caption{Case I.3: Here we consider the case when we find a large left comb. This gives a copy of $K_{t+2}^{(k-1)}$ in the projection. }
        \label{fig: CaseI.3Left}
    \end{figure}
    
    \begin{claim}
    \label{claim: CaseI.2Proj}
        The coloring $c:[N]^{(k-1)}\to \Z_4$ either contains a monochromatic copy of $(F',<)\in \cF^{(k-1)}_{I'}(t)$ in color 0 or a monochromatic copy of $K_{t+2}^{(k-1)}$.  
    \end{claim}
    \begin{proof}
        In view of Claim \ref{claim: CaseI.2St}, we either obtain the right comb $\{x_0\} \dcup Z\dcup \{a_3\}$ (see Fig. \ref{fig: CaseI.3Right}) or the left comb $\{x_0, x_1, x_2\}\dcup Z\dcup \{x_{a_3}\}$ (see Fig. \ref{fig: CaseI.3Left}). 
        
        We first consider the case where $Y = \{x_0\} \dcup Z\dcup \{x_{a_3}\}$ is a right comb (Fig. \ref{fig: CaseI.3Right} above), and let $\l = |Z|+1 = t+1$ and 
        $$y_0 = x_0,\,\quad \{y_1< y_2<\cdots < y_{\ell-1}\} = Z,\quad y_{\ell} = x_{a_3}.$$
        We apply Proposition \ref{prop: rightprojection} with $Y=\{y_0,\dots, y_\ell\}$,  $I_{\ref{prop: rightprojection}} = I'\cup \{t+1\} = I'\cup\{\ell\}$ and $c$, $k$, $N$ as in Proposition \ref{prop: inductionstep}. Verifying conditions \ref{it: righta}--\ref{it: rightc} and ($\star$) in the statement of Proposition \ref{prop: rightprojection} follows similarly as in Claims \ref{claim: CaseI.1Proj} and \ref{claim: CaseI.3Proj}  in Cases I.1 and I.2, and so we omit the details here. As a consequence of Proposition \ref{prop: rightprojection}, we obtain that $c^{-1}(0)$ contains some $(F',<)\in \cF_{I'}^{(k-1)}(t)$. 
        
        Now  consider the case where $Y = \{x_0, x_1, x_2\}\dcup Z\dcup \{x_{a_3}\}$ is a left comb. Let $\ell = t+3$ and  
        \begin{align*}
            \{y_0 < y_1 < y_2\} = \{x_0 < x_1 < x_2\}, \quad 
            \{y_3 < \cdots < y_{\ell-1}\}= Z,\quad y_\l= x_{a_3}. 
        \end{align*}
        Let us fix $f$ to be the increasing bijection from $[0,\ell]$ such that  $y_i = x_{f(i)}$ for each $i\in \{0,\dots, \ell\}$. For example, such a bijection needs to satisfy $x_0 =y_0 = x_{f(0)}$ and hence $f(0)=0$. Similarly, $f(3) = j_0$ where $x_{j_0} = \min Z=y_3$. We now apply Proposition \ref{prop: leftcomb} on $Y= \{y_0,y_1,\dots, y_\ell\}$ with $c$, $k$, $N$ as in Proposition \ref{prop: inductionstep}. To this end, we verify conditions \ref{it: leftlema}--\ref{it: leftlemmac} in the statement of Proposition \ref{prop: leftcomb}. 
        \begin{itemize}
            \item[] \ref{it: leftlema} holds since $Y$ forms a left comb. 
            \item[] \ref{it: leftlemb} holds since for every $J\in \{2,\dots, \l\}^{(k-1)}$, we set $y_J = x_{f(J)}$ and we have (see also Fig. \ref{fig: CaseI.3Left} above) $$x_0=y_0\le y_J = x_{f(J)}\leq y_1 = x_1.$$ 
            \item[] \ref{it: leftlemmac} holds  since the sets $Y^J= \{y_J\} \cup \{y_j:j\in J\} = X^{f(J)}$ are edges of the monochromatic copy of $(F, <)$ and thus $\chi_c(Y^J) = 0$ for every $J \in \{2,\dots, \l\}^{(k-1)}$. 
        \end{itemize}
        Consequently, $c$ is constant over the set $\pi(Y\setminus \{y_0\})^{(k-1)}$. Note that $|Y| = t+4$. Consequently, $\pi(Y\setminus \{y_0\})$ has size $t+2$, and we find a monochromatic $K_{t+2}^{(k-1)}$ in the projection, as shown in Figure \ref{fig: CaseI.3Left}. 
    \end{proof}
Note that in view of above Claim \ref{claim: CaseI.2Proj}, if we obtain that $c^{-1}(0)$ contains $(F',<)\in \cF_{I'}^{(k-1)}(t)$, then it contradicts assumption \ref{it1: inductionstepi} of Proposition \ref{prop: inductionstep}. On the other hand, if Claim \ref{claim: CaseI.2Proj} yields a monochromatic $K_{t+2}^{(k-1)}$ in any color, it contains some member of $\cF_{I'}^{(k-1)}(t)$ and some member of $\rev \cF_{I'}^{(k-1)}(t)$ as ordered subgraphs. This contradicts either \ref{it1: inductionstepi} or \ref{it2: inductionstepii} of Proposition \ref{prop: inductionstep}. 
\\[0.3cm]
\noindent \underline{\textbf{Case II: When $(F, < )$ is color 1}}:  Now we consider the case when $\chi_c$ contains a monochromatic copy of $(F, <)$ in color 1.  In this case, we have $\chi_c(X^I) = 1$, where 
   $$X^I = \{x_0\}\cup\{x_i:i\in I\} =  \{x_0\} \cup \{ x_1< x_2< x_{a_3}< \cdots < x_{a_{k-1}}\}.$$
In view of (\ref{eqn: coloring4}), this can occur in the following 3 situations:
    \begin{itemize}
        \item \textbf{Case II.1:}  $X^I$ forms a \textit{$(k-1,1)$-split}, 
         \item \textbf{Case II.2:} $X^I$ forms a \textit{left comb} and $c(\pi(X^I)) = 2$,
          \item \textbf{Case II.3:} $X^I$ forms a \textit{right comb} and $c(\pi(X^I))= 1$.
    \end{itemize}
    In each of these cases, we will show that one can arrive at a contradiction by using \ref{it1: inductionstepi} or \ref{it2: inductionstepii}. 
\\[0.3cm]
    \noindent\underline{\textbf{Case II.1 and II.2: $X^I$ forms a $(k-1,1)$-split or a left comb:}} We analyse these two cases together since in each of these two cases, for $u=u_{X^I}=a(x_0,x_{a_{k-1}})$, we have
    \begin{align}
    \label{eqn: CaseII.12split}
        X^I_L(u)= X^I\setminus \{x_{a_{k-1}}\} = \{x_0,x_1, x_2,x_{a_3},\dots, x_{a_{k-2}}\} \quad \text{and} \quad X^I_R(u)=\{x_{a_{k-1}}\}.
    \end{align}
    \begin{claim}
    \label{claim: CaseII.12St}
        The set $Y = \set{x_0,x_{a_{k-2}},x_{a_{k-1}},x_{a_{k-1}+1},\ldots,x_{n}}$ forms a left comb.
    \end{claim}
    \begin{proof}
        We'll prove by induction on $\l$ that $\set{x_0,x_{a_{k-2}},x_{a_{k-1}},x_{a_{k-1}+1},\ldots,x_{\l}}$ forms a left comb. Note that  $\set{x_0,x_{a_{k-2}},x_{a_{k-1}}}$ forms a left comb in view of (\ref{eqn: CaseII.12split}). For the induction step, assume that the set
        $$Z= \set{x_0,x_{a_{k-2}},x_{a_{k-1}},x_{a_{k-1}+1},\ldots,x_{\l-1},x_{\l}}$$ 
        forms a left comb.\footnote{Note that when $\ell={a_{k-1}}$, the vertex $x_{\ell-1}$ is not part of our left comb. In this case, the same proof works, but with $x_{\ell-1}$ replaced by $x_{a_{k-2}}$.} Equivalently, for $  u' = a(x_0, x_\ell)$, 
        \begin{align}
        \label{eqn: CaseII.2Steqn}
          \{x_0, x_{a_{k-2}}, x_{\l-1}\}\subseteq L(u') \text{ while } x_\l \in R(u').
        \end{align}
        For $k\ge 5$, let $J=\{2,\dots, k-3\}\cup\{ a_{k-2}, \l, \l+1\}$ be a $(k-1)$-element set, and for $k = 4$, we let $J= \{2, \l, \l+1\}$. Note that in view of Definition \ref{def: separated}, in this case $a_{k-2} = 2$. 
        
        We will prove the induction step by showing that if $Z\cup\{ x_{\l+1}\}$ does not form a left comb, then the edge $X^J = \{x_J\}\cup \{x_j: j\in J\}$ must form a balanced split, which in view of (\ref{eqn: coloring4}) contradicts the assumption that it has color 1.

        Indeed, assume for the sake of contradiction that $Z\cup \{x_{\l+1}\}$ does not form a left comb, i.e., $\delta(x_{\l-1},x_\l) <\delta(x_{\l}, x_{\l+1}).$
        Note that since $Z$ forms a left comb, we have that  $\delta(x_0, x_\ell) = \delta(x_{\ell-1}, x_\ell)$, and so we have that $\delta(x_{0}, x_\ell) < \delta(x_{\ell}, x_{\ell + 1})$.  This is equivalent to $\{x_{0}, x_\ell, x_{\l+1}\}$ forming a right comb. As a consequence, $\{x_\ell,x_{\ell+1}\}\subseteq R(u')$.  Furthermore, in view of (\ref{eqn: CaseII.2Steqn}) and since $x_0 \le x_J \le x_{a_{k-2}}$,  we must have $x_J\in L(u')$ as well. Consequently, 
        $$\{x_J,x_{a_{k-2}}\}\subseteq L(u')\text{ and } \{x_\l, x_{\l+1}\}\subseteq R(u'),$$
        implying that $X^J$ forms a balanced split, which in view of (\ref{eqn: coloring4}) contradicts $\chi_c(X^J)=1$. 
    \end{proof}
    Having shown that $Y = \{x_0, x_{a_{k-2}},x_{a_{k-1}}, \ldots, x_n \}$ forms a left comb, we will use Proposition \ref{prop: leftcomb} to arrive at a contradiction to the assumptions \ref{it1: inductionstepi} and \ref{it2: inductionstepii} of Proposition \ref{prop: inductionstep}. To this end, let 
    \begin{align}
        \label{eqn: caseII.12l,t}
        \ell = |Y|-1 = n-a_{k-1}+2\ge n/2k \ge t+2. 
    \end{align}
    \begin{claim}
            The coloring $c:[N]^{(k-1)}\to \Z_4$  contains a monochromatic copy of $K_{\ell-1}^{(k-1)}.$
    \end{claim}
    \begin{proof}
       Let 
        $$y_0 = x_0,\; y_1 = x_{a_{k-2}},\; y_2 = x_{a_{k-1}},\; y_3 = x_{{a_{k-1}+1}},\; \ldots ,\; y_{\l-1} = x_{n-1},\; y_\l = x_{n}.$$
        We apply Proposition \ref{prop: leftcomb} with $Y = \{y_0,y_1,\dots,y_\ell\}$ and $c$, $k$, $N$ as in the statement of Proposition \ref{prop: inductionstep}. To this end, we verify the assumptions. 
        \begin{itemize}
            \item Condition \ref{it: leftlema} holds since in view of Claim \ref{claim: CaseII.12St} the set $Y$ forms a left comb.
            \item Conditions \ref{it: leftlemb} and \ref{it: leftlemmac} hold since the set $Y$ is a subset of the distinguished set of vertices $\{x_0,x_1, \dots, x_n\}$ of $(F_{\phi},<)$ and  $\chi_c$ takes the value 1 over all edges $X^J\in F$.
        \end{itemize}
       Consequently, $c$ is constant over  $\pi(Y\setminus \{y_0\})^{(k-1)}$ and thus contains a monochromatic clique of size $|Y|-2 = \l-1\ge t+1$.
    \end{proof}
    In view of the above claim, $c$ contains a monochromatic copy of $K_{t+1}^{(k-1)}$. Such a clique contains some member of  $\cF_{I'}^{(k-1)}(t)$ and some member of $\rev \cF_{I'}^{(k-1)}(t)$. Consequently, regardless of which color the clique has, we contradict either assumption \ref{it1: inductionstepi} or \ref{it2: inductionstepii} of Proposition \ref{prop: inductionstep}. 
    \\[0.3cm]
    \noindent\underline{\textbf{Case II.3: $X^I$ forms a right comb:}}
   As mentioned in Remark \ref{rem: notofcolor2}, in the analysis of  Case I.2 where $X^I$ formed a right comb, we only used the fact that the monochromatic copy of $(F,<)$ is not in color 2. Since here we assume that $(F, <)$ is in color 1, the same analysis applies  here. Let
 $$t' = a_{k-1}-1 \text{ and } I'' = I\setminus \max I = \{a_1< a_2 < \dots< a_{k-2}\}\in [t']^{(k-2)}.$$ 
In the same way that we prove Claim \ref{claim: CaseI.3Proj} in Case I.2, it can be shown that
    \begin{claim}
          $c^{-1}(1)\subseteq [N]^{(k-1)}$ contains a monochromatic copy of some $(F'',<)\in \cF_{I''}^{(k-1)}(t')$.
    \end{claim}
Denote the distinguished set of vertices of $V(F'')$  by $\{\delta_0, \delta_1,\dots ,\delta_{t'}\}$.  Recall that in (\ref{eqn: tI'}) we are given an arbitrary separated set $I'\in [t]^{(k-2)}$. Similarly as in Case I.2, for given $I'$, by deleting appropriate vertices, we obtain from $(F'', <)$ a member $(F', <)$ of $\cF_{I'}^{(k-1)}(t)$ as an ordered subgraph.

In particular, if $I' = \{b_1=1< b_2 = 2< b_3 <\dots < b_{k-2}\}$, 
$V(F')$ will have distinguished vertex set $Z\subseteq \{\delta_0,\delta_1,\dots, \delta_{t'}\}$ where $|Z| =t$ and is such that:
 $$\{\delta_0\} \cup\{\delta_{a_i}: 1\le i\le k-2\}\subseteq Z,$$
 where $\delta_{a_i}$ is at $(b_i+1)^{\text{st}}$ position in $Z$.
This  contradicts \ref{it1: inductionstepi} of Proposition \ref{prop: inductionstep}. This completes Case II.3.
\\[0.3cm]
To summarize, we have  shown \ref{it1: inductionstepa} of Proposition \ref{prop: inductionstep}, i.e., we have verified that $\chi_c$ does not contain a monochromatic copy of some $(F, <)\in \cF_{I}^{(k)}(n)$ in color 0 or 1. It remains to show \ref{it1: inductionstepb}, i.e., that $\chi_c$ does not contain a monochromatic copy of some member of $ \rev \cF_{I}^{(k)}(n)$ in color 2 or 3.

The proof of \ref{it1: inductionstepb} can be deduced from \ref{it1: inductionstepa}. Indeed, define the ``reversing map'' $\theta:[2^N]\to[2^N]$, given by $\theta(i)=2^N+1-i$. Note that $\theta$ preserves projections, that is, $\pi(X)=\pi(\theta (X))$ for any $X\subset[2^N]$. Note that in view of (\ref{eqn: coloring4}), if $\chi_c(X) = 2$, then there are three possibilities, each implying  $\chi_c(\theta(X)) = 1$. 
\begin{itemize}
    \item $X$ forms a $(1,k-1)$-split. In this case $\theta(X)$ forms a $(k-1,1)$-split and in view of (\ref{eqn: coloring4}), we have   $\chi_c(\theta(X)) = 1$. 
    \item $X$ forms a left comb and $3-c(\pi(X)) = 2$. In this case $\theta(X)$ forms a right comb, and since $\pi(X) = \pi(\theta( X))$ and in view of (\ref{eqn: coloring4}),
    $$\chi_c(\theta(X)) = c(\pi(\theta( X)) = c(\pi(X))=1.$$
    \item $X$ forms a right comb with $c(\pi(X)) = 2$. In this case $\theta(X)$ forms a left comb, and since $\pi(X) = \pi(\theta (X))$ and in view of (\ref{eqn: coloring4}),
    $$\chi_c(\theta(X)) = 3-c(\pi(\theta (X))) = 3- c(\pi(X)) = 1.$$
\end{itemize}
Similarly, it can be checked that if $\chi_c(X) = 3$, then $\chi_c(\theta(X)) = 0$. As a consequence, we obtain that if $\chi_c:[2^N]^{(k)}\to \ZZ_4$ contains a monochromatic copy of some $(F_{\phi},<_{\rev})\in \rev \cF_{I}^{(k)}(n)$ of color 2 or 3, then $\theta(V(F))$ forms a copy of $\cF_{I}^{(k)}(n)$ in color 0 or 1. This however contradicts \ref{it1: inductionstepa} of Proposition \ref{prop: inductionstep}. Consequently, \ref{it1: inductionstepb} of Proposition \ref{prop: inductionstep} also holds and $\chi_c$ does not contain any member of $\rev \cF_{I}^{(k)}(n)$ in colors 2 or 3. 
\end{proof}
We now complete the proof of Theorem \ref{thm: steppingup}.
\begin{proof}[Proof of Theorem \ref{thm: steppingup}]
    We will first prove the case $k = 3$. Due to \cite{Erd47}, there exists $n_0(3)$ such that for all $n\ge n_0(3)$, we have 
    $$[2^{n/2}]\narrow (K_{n-1})_2^2.$$
   Let $b_3 = 1/2$, and
   $$N = 2^{n^{b_3}} \le 2^{n/2}, \text{ and  } 2^N = t_2(n^{b_3}).$$
   Let $I = \{1,2\}$ and $c:[N]^{(2)}\to \{0,1\}$ be a coloring that doesn't contain a monochromatic copy of $K_{n-1}^{(2)}$. Let $(F, <)$ be any ordered $3$-graph that contains subgraphs $(F_1, <)\in \cF_{I}^{(3)}(n)$ and  $(F_2, <_{\rev})\in \rev\cF_{I}^{(3)}(n)$. In view of Proposition \ref{prop: basecase}, the coloring $\chi_c: [2^N]^{(3)}\to \Z_4$ does not contain any monochromatic copy of $(F_1,<)\in \cF_{I}^{(3)}(n)$ in colors $0,1$ and does not contain any monochromatic copy of $(F_2,<_{\rev})\in \rev \cF_{I}^{(3)}(n)$ in colors $2,3$.  In particular, it implies that  
   $$[2^N] = [t_2(n^{b_3})]\narrow (F,<)^{3}_4.$$
   Now we prove the lower bound in Theorem \ref{thm: steppingup} for $k\ge4$ by induction. Let our induction hypothesis for $k\ge 4$ be the following stronger assumption:
   \begin{enumerate}[label=($\star$)]
       \item There exist integers $n_0(k-1)$ and $b_{k-1}>0$ such that for all $t\ge n_0(k-1)$ and separated set $I'\in [t]^{(k-2)}$ the following holds. For $N = t_{k-2}(t^{b_{k-1}})$, there exists a coloring $c:[N]^{(k-1)}\to \Z_4$ that does not contain any member of $\cF_{I'}^{(k-1)}(t)$ in colors $0,1$ and any member of $\rev \cF_{I'}^{(k-1)}(t)$ in colors $2,3$. 
   \end{enumerate}
   We now fix the following parameters. Let $n_0(k) = \max\{8k\cdot n_0({k-1}), n_k\}$ where $n_k$ is obtained from Proposition \ref{prop: inductionstep}. Let $n \ge n_0(k)$ and  $I\in [n]^{(k-1)}$ be a separated set. For $t = n/8k \ge n_0(k-1)$, let $I'\in [t]^{(k-2)}$ be a separated set.  We now apply Proposition \ref{prop: inductionstep} with $n$, $k$, $t$, $I$ and $I'$ and with $c:[N]^{(k-1)}\to \ZZ_4$ as in the induction hypothesis. 
   
  In view of the induction hypothesis, $c:[N]^{(k-1)}\to \ZZ_4$  satisfies \ref{it1: inductionstepi} and \ref{it2: inductionstepii} of Proposition \ref{prop: inductionstep}, and consequently we obtain that $\chi_c: [2^N]^{(k)}\to \Z_4$ does not contain any member of  $\cF^{(k)}_{I}(n)$ in colors $0,1$, and any member of $\rev \cF_{I}^{(k)}(n)$ in colors $2,3$. Note that for $b_k>0$ chosen to be sufficiently smaller than $b_{k-1}$, we have
   $$2^N = t_{k-1}\left(\left(\frac{n}{8k}\right)^{b_{k-1}}\right)\geq t_{k-1}(n^{b_k}).$$
   Consequently, $\chi_c:\left[t_{k-1}(n^{b_k})\right]^{(k)}\to \ZZ_4$ does not contain any member of  $\cF_{I}^{(k)}(n)$ in colors 0 or 1 and any member of $\rev \cF_I^{(k)}(n)$ in colors 2 or 3. 
\end{proof}
\section{Random orderings}\label{sec:orderings}

Given a family of ordered $k$-graphs $\cF$ and an unordered $k$-graph $H$ with the property that $F\subseteq H$ for every $(F,<)\in \cF$, we denote by
\begin{align*}
    H \ordarrow \cF,
\end{align*}
the fact that any total ordering of $V(H)$ contains a member of $\cF$ as a subgraph. Note that for any such family $\cF$, one can always construct an $H$ such that $H\ordarrow \cF$ by taking $H$ to be the complete graph on $\max_{F\in\cF} v(F)$ vertices. Our goal in this section is to construct, for an appropriate family of ordered $(k,k-1)$-systems $\cF$, a $(k,k-1)$-system $H$ of polynomial size in $\max_{F\in \cF} v(F)$ with the above property. We begin by describing our family. Similarly, as in Definition \ref{def:F_I}, we fix a set $I\in [n]^{(k-1)}$ with $1\in I$ and consider the family $\Omega_I:=\{2,\ldots,n\}^{(k-1)}\cup\{I\}$. Let $X=\{x_1,\ldots,x_n\}$ and $\{x_J: J\in \Omega_{I}\}$ be set of distinct points.

\begin{definition}\label{def:G_I}
Let $\cG^{(k)}_I(n)$ to be the family of all ordered $k$-graphs $(G,<)$ with vertex set
    \begin{align*}
        V(G)=\{x_J: J\in \Omega_I\}\cup \{x_i: 1\leq i \leq n\},
    \end{align*}
    edge set
    \begin{align*}
        G=\left\{\{x_J\}\cup \{x_j:\: j\in J\}:\: J\in \Omega_I\right\},
    \end{align*}
    and vertex ordering satisfying
    \begin{align*}
        x_I<x_J<x_1<\ldots<x_n,
    \end{align*}
    for $J\in \Omega_I$. Moreover, we define $\rev\cG^{(k)}_I(n):=\{(G,<_{\rev}):\: (G,<)\in \cG^{(k)}_{I}(n)\}$ as the family obtained by reversing the ordering of the vertices in $\cG_{I}^{(k)}(n)$.
\end{definition}

We remark that Definition \ref{def:G_I} is a special case of Definition \ref{def:F_I}.  Similarly, we have $\rev \cG_I^{(k)}(n)\subseteq \rev \cF_I^{(k)}(n)$. Furthermore, note that every member $(G,<) \in \cG_I^{(k)}(n)$ is a $(k,k-1)$-system with $v(G)=n+\binom{n-1}{k-1}+1$ vertices and $e(G)=\binom{n-1}{k-1}+1$ edges.

The main result of this section can now be stated as follows.

\begin{theorem}\label{thm:winkler}
    Given $k\geq 3$, there exists $n_1:=n_1(k)$ and constant $C:=C(k)$ such that for $n\geq n_1$ the following holds. For a set $I\subseteq [n]^{(k-1)}$ with $1\in I$, there exists a $(k,k-1)$-system $H$ on $n^{C}$ vertices such that
    \begin{align*}
        H \ordarrow \cG_I^{(k)}(n) \quad \text{and} \quad H \ordarrow \rev\cG_I^{(k)}(n).
    \end{align*}
\end{theorem}

The $(k,k-1)$-system $H$ will be constructed at random using an approach similar to \cite{RW89}. Before we proceed with the proof of Theorem \ref{thm:winkler}, we describe a subconstruction that will be useful to us.

\subsection{An auxiliary construction}\label{subsec:construction}

Throughout this subsection we make extensive use of the well-known Chernoff bounds for the lower and upper tails of a hypergeometric random variable (see \cite{JLR00}). We say that $X \sim \Hyp(N,K,n)$ if $X$ is distributed as the number of successes in $n$ draws without replacement from a population of size $N$ containing exactly $K$ desired elements.

\begin{theorem}\label{thm:chernoff}
Let $X\sim \Hyp(N,K,n)$ and $\mu=\EE(X)=nK/N$. Then, for $\delta\in (0,1)$, we have
\begin{align*}
    \PP(X>\mu+s)&\leq \exp\left(-\frac{s^2}{2\mu+s/3}\right), \quad s\geq 0 \\
    \PP(X<\mu-s)&\leq \exp\left(-\frac{s^2}{2\mu}\right), \quad 0\leq s \leq \mu.
\end{align*}
\end{theorem}

Let $k,n \geq 3$ and $I\subseteq [n]^{(k-1)}$ with $1\in I$ be given. Define 
\begin{align}\label{eq:m}
    m:=n^{k+3}, \quad  N:=mn+m^k\left(\binom{n-1}{k-1}+1\right)=\Theta(n^{k^2+4k-1}),
\end{align}
and construct an unordered $k$-graph $R$ on $N$ vertices as follows. The vertex set of $R$ is partitioned into sets 
\begin{align*}
    V(R)=\bigcup_{J\in \Omega_I} V_J\cup \bigcup_{i=1}^n V_i,
\end{align*}
where $|V_J|=m^{k-1}$ for $J\in \Omega_I$ and $|V_i|=m$ for $1\leq i \leq n$. For each $J\in \Omega_I$, we identify the set $V_J$ with $\prod_{j\in J} V_j$ by labeling
\begin{align*}
V_J=\left\{v_{z}:\: z=(v_j)_{j\in J}\in\prod_{j\in J} V_j \right\}.
\end{align*}
The edge set of $R$ consists of $k$-tuples of the form
\begin{align*}
R=\left\{\{v_j\}_{j\in J}\cup \{v_z\} :\: z=(v_j)_{j\in J}\in \prod_{j\in J}V_j\right\}.
\end{align*}
Note that $R$ has $e(R)=m^k\left(\binom{n-1}{k-1}+1\right)$ edges. We also remark that $R$ is a $(k,k-1)$-system with the property
\begin{enumerate}[label=($\star$)]
\item\label{eq:star} For $J\in \Omega_I$ and every transversal $z=(v_j)_{j\in J} \in\prod_{j\in J}$, there exists a unique vertex $v_z \in V_J$ extending $\{v_j\}_{j\in J}$ to an edge.
\end{enumerate}

One can think of $R$ as a ``$(k,k-1)$-system blow-up'' of the $k$-graph $G$ (Definition \ref{def:G_I}) in the sense that it contains many unordered copies of $G$, while $R$ itself remains a $(k,k-1)$-system. Hence, one might expect that a random ordering of the vertices of $R$ contains, with high probability, an ordered copy of an element of $\cG_I^{(k)}(n)$. This is confirmed in the next lemma. 

Note that any ordering of the vertices of $V(R)$ can be viewed as a bijection $\psi:V(R)\rightarrow [N]$ by the correspondence sending the bijection $\psi$ to an ordering $<_{\psi}$ such that $x<_\psi y$ if and only if $\psi(x)<\psi(y)$. We say that a bijection $\psi$ is \emph{good} if $(R, <_{\psi})$ contains a copy of $(G,<)$ for every $(G,<)\in \cG_I^{(k)}(n)$. Otherwise, we say that $\psi$ is \emph{bad}. 

\begin{lemma}\label{lem:random}
    Let $\psi:V(R)\rightarrow [N]$ be an bijection chosen uniformly at random. Then
    \begin{align*}
    \PP(\psi \text{ is bad})= \exp(-\Omega(n^2)).
    \end{align*}
\end{lemma}

\begin{proof}
   Fix integers $n,k\geq 3$, and define
\begin{align}\label{eq:alpha}
    \alpha:=\frac{1}{4n^{k-1}}, \quad \beta:=1-\frac{1}{2n^{k-1}}, \quad \gamma:=\frac{1}{4n^{k}}
\end{align}
to be real numbers that will be used throughout the proof. Consider the partition 
\begin{align*}
    [N]=X_I\cup X_{\star}\cup \bigcup_{i=1}^n X_i,
\end{align*}
where the sets are ordered by $X_I<X_{\star}<X_1<\ldots<X_n$ and have sizes
\begin{align*}
    |X_I|=\alpha N, \quad |X_{\star}|=\beta N, \quad |X_1|=\ldots=|X_n|=\gamma N.
\end{align*}
We expose the random bijection $\psi: V(R) \rightarrow [N]$ in three rounds:
\vspace{0.2cm}

\noindent Round 1: We expose the labels of $\bigcup_{i=1}^n V_i$.\\
\noindent Round 2: We expose the labels of $\bigcup_{J\neq I} V_J$.\\
\noindent Round 3: We expose the labels of $V_I$. 
\vspace{0.2cm}

In each round, we will show that a desired event $\cE_i$ holds with high probability for $1\leq i \leq 3$ (to be described later in the proof). Finally, we will conclude by showing that any bijection satisfying all $\cE_i$ is a good one. We now proceed with the details.

\vspace{0.2cm}

\noindent \textbf{Round 1}: We first expose the labels of the vertices in $\bigcup_{i=1}^n V_i$. Note that for $1 \leq i \leq n$, the random variable $|\psi(V_i)\cap X_i|$ follows the distribution of a hypergeometric random variable $\Hyp(N,|X_i|,|V_i|)$. By Theorem \ref{thm:chernoff}, together with (\ref{eq:m}) and (\ref{eq:alpha}), we obtain
\begin{align*}
    \PP\left(\Big||\psi(V_i)\cap X_i|-\gamma m\Big|>\frac{\gamma m}{2}\right)= \exp\left(-\Omega(\gamma m)\right) =\exp(-\Omega(n^3)).
\end{align*}
Therefore, by a union bound, with probability at least $1-\exp(-\Omega(n^3))$ we have the event
\begin{align}\label{eq:E1}
    \cE_1: \, |\psi(V_i)\cap X_i|=\left(1\pm \tfrac{1}{2}\right)\gamma m, \quad \forall 1\leq i \leq n.
\end{align}

    \vspace{0.2cm}

    \noindent \textbf{Round 2}: Consider all possible bijections conditioned on the event $\cE_1$. For $1\leq i \leq n$, let $W_i\subseteq \psi(V_i)\cap X_i$ be a subset of size
    \begin{align}\label{eq:t}
        t:=\frac{\gamma m}{2}=\Theta(n^3). 
    \end{align}
    For each $J\in \Omega_I\setminus\{I\}$, define
\begin{align*}
      W_J=\{v_z:\: z=(w_j)_{j\in J}\in\prod_{j\in J} W_j\}\subseteq V_J,
\end{align*}
the set of vertices extending the transversals $\{w_j\}_{j\in J}$ to an edge in $R$.

We now expose the labels of $\bigcup_{J\neq I} V_J$. Let
\begin{align*}
    X_{I}^2=X_{I}\setminus\left(\bigcup_{i=1}^n \psi(V_i)\right), 
    \quad 
    X_{\star}^{2}=X_{\star}\setminus\left(\bigcup_{i=1}^n \psi(V_i)\right)
\end{align*}
denote the remaining available labels in $X_I$ and $X_{\star}$ after Round~1. By (\ref{eq:m}) and (\ref{eq:alpha}),
\begin{align}\label{eq:newX}
    |X_{I}^{2}|&\geq |X_I|-\sum_{i=1}^n|V_i|=\alpha N-mn> \alpha N/2, \nonumber\\
    |X_{\star}^{2}|&\geq |X_{\star}|-\sum_{i=1}^n|V_i|=\beta N-mn> \left(\beta-\frac{1}{n^k}\right)N.
\end{align}

For each $J\in \Omega_I\setminus\{I\}$, the random variable $|\psi(W_J)\cap X_I^2|$ follows a hypergeometric distribution $\Hyp(N-mn, |X_{I}^{2}|, |W_J|)$ with mean
\begin{align*}
    \EE\left(|\psi(W_J)\cap X_I^2|\,\Big| \, \cE_1\right)=\frac{|X_I^2|\cdot|W_J|}{N-mn}\leq 2\alpha t^{k-1}.
\end{align*}
Hence, Theorem \ref{thm:chernoff}, together with (\ref{eq:m}), (\ref{eq:alpha}) and (\ref{eq:t}) gives us that
\begin{align*}
    \PP\left(|\psi(W_J)\cap X_I^2|>3\alpha t^{k-1}\,\Big| \, \cE_1\right)
    = \exp\left(-\Omega(\alpha t^{k-1})\right)
    = \exp(-\Omega(n^{2k-2}))=\exp(-\Omega(n^4)),
\end{align*}
for $k\geq 3$.

Similarly, $|\psi(W_J)\cap X_{\star}^{(2)}|\sim \Hyp(N-mn, |X_{\star}^{(2)}|, |W_J|)$, and by (\ref{eq:newX}),
\begin{align*}
    t^{k-1}\geq\EE\left(|\psi(W_J)\cap X_{\star}^{(2)}|\,\Big|\, \cE_1\right)
    = \frac{|X_{\star}^{2}|\cdot|W_J|}{N-mn}
    \geq \left(\beta-\frac{1}{n^k}\right) t^{k-1}.
\end{align*}
Therefore, Theorem \ref{thm:chernoff} applied with $s=\left(\frac{1}{2n^{k-1}}-\frac{1}{n^k}\right)t^{k-1}$, (\ref{eq:m}), (\ref{eq:alpha}) and (\ref{eq:t}) gives us that
\begin{align*}
    \PP\left(|\psi(W_J)\cap X_\star^{2}|< \left(\beta - \frac{1}{2n^{k-1}}\right)t^{k-1}\, \Big |\, \cE_1\right)
    = \exp\left(-\Omega\!\left(\tfrac{t^{k-1}}{n^{2k-2}}\right)\right)
    = \exp(-\Omega(n^2)),
\end{align*}
for $k\geq 3$. Thus, by a union bound, with probability at least $1-\exp(-\Omega(n^2))$, the event
\begin{align}\label{eq:E2}
    \cE_2:\: |\psi(W_J)\cap X_{I}^{2}|\leq 3\alpha t^{k-1}
     \text{  and  } 
    |\psi(W_J)\cap X_{\star}^{2}|\geq\left(\beta-\frac{1}{2n^{k-1}}\right)t^{k-1},
    \quad \forall J\in \Omega_I\setminus\{I\},
\end{align}
holds.

\vspace{0.2cm}

\noindent \textbf{Round 3}: Consider all bijections of $V_I$ conditioned on $\cE_1$ and $\cE_2$. A transversal $z=(w_i)_{i=1}^n \in \prod_{i=1}^n W_i$ is \emph{consistent} if for every $J\in \Omega_I\setminus \{I\}$, the vertex $w_z\in W_J$ lies in $X_\star^{(2)}$. In particular, for a consistent trasnversal, the vertex set
\begin{align*}
   \{v_z:\: z=(w_j)_{j\in J}, \,J\in \Omega_I\setminus\{I\}\} \cup \{w_i:\: 1\leq i \leq n\}
\end{align*}
forms a partial copy of an element of $\cG_I^{(k)}(n)$. To complete the copy, it suffices to find a consistent transversal $(w_i)_{i=1}^n$ with vertex $v_{z}\in X_{I}\cap V_I$ for $z=(w_i)_{i\in I}$.

For each $J\in \Omega_I\setminus\{I\}$, the event $\cE_2$ guarantees that at most $(1-\beta+\frac{1}{2n^{k-1}})t^{k-1}$ tuples in $\prod_{j\in J} W_j$ correspond to vertices in $W_J$ outside $X_\star^{(2)}$. Thus, by (\ref{eq:m}), (\ref{eq:alpha}), and $k\geq 3$, the total number of non-consistent transversals is at most
\begin{align*}
    \binom{n-1}{k-1}\left(1-\beta+\frac{1}{2n^{k-1}}\right)t^{n}\leq \frac{t^n}{2}.
\end{align*}
In other words, at least $t^n/2$ transversals are consistent. By an averaging argument, there exists a set $Z \subseteq \prod_{i\in I} W_i$ of size $t^{k-1}/2$ such that each element can be extended to a consistent transversal. Let 
\begin{align*}
    \tilde{Z}:=\{v_z\in V_I:\: z=(w_i)_{i\in I}\in Z\} \subseteq V_I
\end{align*}
 be the set of vertices in $V_I$ corresponding to the $(k-1)$-tuples in $Z$, i.e., those satisfying property \ref{eq:star} of $R$. In particular, $|\tilde{Z}|\geq t^{k-1}/2$.

    We now expose the labeling of $V_I$. Let $X_I^3:=X_I^2\setminus\left(\bigcup_{J\neq I} V_J\right)$ be the set of remaining labels available in $X_I$. By (\ref{eq:m}), (\ref{eq:alpha}), (\ref{eq:newX}), and (\ref{eq:E2}), we have
\begin{align*}
    |X_I^3|\geq|X_I^2|-\sum_{J\neq I}|\psi(W_J)\cap X_I^2|> \frac{\alpha N}{2}-\binom{n-1}{k-1}\cdot3\alpha t^{k-1}\geq \frac{\alpha N}{4}.
\end{align*}
As in the previous two rounds, note that $|\psi(\tilde{Z})\cap X_I^3| \sim \Hyp(|V_I|, |X_I^3|, |\tilde{Z}|)$ with mean
\begin{align*}
    \EE\left(|\psi(\tilde{Z})\cap X_I^3| \, \big| \, \cE_1, \, \cE_2 \right)=\frac{|X_I^3|\cdot|\tilde{Z}|}{|V_I|}\geq \frac{(\alpha N/4) \cdot t^{k-1}/2}{m^k}=\Omega(n^{3k-3}).
\end{align*}
Then by Theorem \ref{thm:chernoff}, we have
\begin{align*}
    \PP(|\psi(\tilde{Z})\cap X_I^3|=0)= \exp(-\Omega(n^{3k-3}))=\exp(-\Omega(n^5))
\end{align*}
for $k\geq 3$. This implies that, with probability at least $1-\exp(-\Omega(n^5))$, the event
\begin{align}\label{eq:E3}
    \cE_3:\: \psi(\tilde{Z})\cap X_I^3\neq \emptyset
\end{align}
holds.

\vspace{0.2cm}
\noindent \textbf{Finishing the proof}: To finish the proof, note that a bijection satisfying events $\cE_1$, $\cE_2$, and $\cE_3$ contains a copy $\tilde{G}$ of $G$ with vertices
\begin{align*}
    V(\tilde{G})=\{v_i \in V_i:\:1\leq i \leq n\} \cup \left\{v_{z_J} \in V_J:\: z_J=(v_j)_{j\in J}\in \prod_{j\in J}V_j,\, J\in  \Omega_I\right\}  
\end{align*}
satisfying
\begin{align*}
    v_i\in X_i, \quad v_{z_I} \in X_I, \quad \text{and} \quad v_{z_J}\in X_\star,
\end{align*}
for all $J\in \Omega_I\setminus \{I\}$. Hence $(\tilde{G}, <_{\psi})\in \cG_I^{(k)}(n)$. Recall that a bijection $\psi$ is bad if $(R,<_{\psi})$ does not contain a member of $\cG_I^{(k)}(n)$. Therefore, if $\psi$ is bad, then at least one of the events $\cE_1$, $\cE_2$, or $\cE_3$ does not hold, and by (\ref{eq:E1}), (\ref{eq:E2}), and (\ref{eq:E3}) we obtain that
\begin{align*}
    \PP(\psi \text{ is bad})= \exp(-\Omega(n^2)).
\end{align*}
This concludes the proof of the lemma.
\end{proof}

\begin{remark}
    We remark that a similar statement in Theorem \ref{thm:winkler} holds if the map $\psi$ is injective or if we replace $\cG_I^{(k)}(n)$ by $\rev \cG_I^{(k)}(n)$.
\end{remark}

\subsection{Proof of Theorem \ref{thm:winkler}}

Let $p$ be the smallest prime greater or equal to $N$. By Betrand's postulate, one can always have 
\begin{align}\label{eq:p}
 N\leq p \leq 2N.    
\end{align}
A projective plane $\cL_p$ of order $p$ is a $(p+1)$-graph satisfying the following properties
\begin{enumerate}[label=(P\arabic*)]
    \item\label{eq:sizeofL} $v(\cL_p)=e(\cL_p)=p^2+p+1$,
    \item\label{eq:LisSteiner} $\cL_p$ is a full Steiner $(p+1,2)$-system,
    \item\label{eq:Lintersect} If $L_1,L_2 \in \cL_p$, then $|L_1\cap L_2|=1$.
\end{enumerate}
It is well-known that the projective plane $\cL_p$ exists for every prime $p$.

We construct the graph $H$ as follows: Let $R$ be the $k$-graph defined in the previous subsection (Subsection \ref{subsec:construction}) and let $\tilde{R}$ the hypergraph on $p$ vertices obtained by adding $p-N$ isolated vertices to $R$. Fix an arbitrary bijection $f: V(\tilde{R})\rightarrow [p]$. We define the vertex set of $H$ to be the vertices of $\cL_P$, i.e., $V(H)=V(\cL_p)$. For each $L\in \cL_p$ consider a bijection $\psi_L:V(L)\rightarrow [p]$ chosen uniformly at random. Let $\tilde{R}_L$ be the hypergraph with vertex set $V(L)$ induced by the map $f^{-1}\circ \psi_L$. That is, $e \in \tilde{R}_{L}$ if and only if $f^{-1}\circ\psi_L(e)\in \tilde{R}$. We set
\begin{align*}
    H=\bigcup_{L\in \cL_p} \tilde{R}_L.
\end{align*}
Note that since every $\tilde{R}_L$ is a $(k,k-1)$-system, by \ref{eq:LisSteiner} we obtain that $H$ is a $(k,k-1)$-system. Moreover, by (\ref{eq:m}) and (\ref{eq:p}) we have
\begin{align}\label{eq:sizeH}
    v(H)=p^2+p+1\leq 4N^2+2N+1=O(n^{2k^2+8k-2}).
\end{align}

Therefore, it only remains to show that with high probability
\begin{align*}
    H\ordarrow \cG_I^{(k)}(n) \quad \text{and} \quad H\ordarrow \rev\cG_I^{(k)}(n)
\end{align*}
holds. Fix a ordering $<$ of $V(H)$. For each $L \in \cL_p$, the ordering $<$ induces a bijection $g:V(L) \rightarrow [p]$. Thus, the fact that $\tilde{R}_L$ contains an element of $\cG_I^{(k)}(n)$ is equivalent to the composition $\varphi_L:=g\circ f^{-1}\circ \psi_L$ being bad. By Theorem \ref{thm:winkler} and the fact that $\cL_p$ is a linear graph, we obtain that
\begin{align*}
    \PP\left((H,<) \text{ does not contain a member of $\cG_I^{(k)}(n)$}\right)=\prod_{L\in \cL_p}\PP(\varphi_L \text{ is bad})=\exp(-\Omega(n^2\cdot e(\cL_P)))
\end{align*}
Similarly, by the remark in Theorem \ref{thm:winkler}, we have
\begin{align*}
    \PP\left((H,<) \text{ does not contain a member of $\rev\cG_I^{(k)}(n)$}\right)=\exp(-\Omega(n^2\cdot e(\cL_P))).
\end{align*}
Taking an union bound over all possible $v(H)!$ orderings, we obtain by \ref{eq:sizeofL}, (\ref{eq:sizeH}) and the definition of $V(H)$ that
\begin{align*}
    \PP(H\not\ordarrow \cG_I^{(k)}(n) \text{ or } H\not\ordarrow \cG_I^{(k)}(n))=2v(H)!\cdot \exp(-\Omega(n^2\cdot v(H)))=o(1). 
\end{align*}
This concludes the proof of the theorem. \qed

\section{Proof of Theorem \ref{thm:main}}\label{sec:main}

We finally put together the results in Section \ref{sec:steppingup} and \ref{sec:orderings} to prove the main theorem.

\begin{proof}[Proof of Theorem \ref{thm:main}]
Given $k\geq 3$, let $b_k$ and $C$ be the constants given by Theorem~\ref{thm: steppingup} and~\ref{thm:winkler}. Let $n_2=\max\{n_0, n_1\}$ where $n_0$ and $n_1$ are the integers obtained by Theorem~\ref{thm: steppingup} and~\ref{thm:winkler}. We set $h_0=n_2^C$. 

Given an integer $h\geq h_0$, let $n=\lfloor h^{1/C} \rfloor$. By definition, we have that $n\geq n_0$. Therefore, one can apply Theorem \ref{thm: steppingup} to obtain a index set $I \in [n]^{(k-1)}$ with $1\in I$ and a $4$-coloring $\chi:\left[t_{k-1}(n^{b_k})\right]^{(k)}\rightarrow \ZZ_4$ with the following property:
\begin{enumerate}[label=($\star \star$)]
\item\label{eq:chi} There are no monochromatic copies of any ordered graph $(F,<)$ which contains a copy of a member of $\cF_I^{(k)}(n)$ and $\rev\cF_I^{(k)}(n)$.
\end{enumerate}
Since $n\geq n_1$, we can apply Theorem \ref{thm:winkler} to obtain a $(k,k-1)$-system $H$ with the property that
\begin{align*}
 H \ordarrow \cG_I^{(k)}(n) \quad \text{and} \quad H \ordarrow \rev\cG_I^{(k)}(n).
\end{align*}

We claim that the coloring $\chi$ obtained by Theorem \ref{thm: steppingup} on $\left[t_{k-1}(n^{b_k})\right]^{(k)}$ does not contain a monochromatic copy of $H$. Suppose by contradiction that there exists a monochromatic copy of $H$. By Theorem \ref{thm:winkler} and the natural linear ordering of $\left[t_{k-1}(n^{b_k})\right]$, we obtain a monochromatic ordered copy of a member of $\cG_{I}^{(k)}(n)\subseteq \cF_I^{(k)}(n)$ and a monochromatic copy of a member of $\rev \cG_I^{(k)}(n)\subseteq \rev\cF_I^{(k)}(n)$. However, this contradicts property \ref{eq:chi} of $\chi$. Therefore, we obtain that 
\begin{align*}
    R(H;4)\geq t_{k-1}(n^{b_k})=t_{k-1}(h^{c_k})
\end{align*}
for some constant $c_k>0$. This finishes the proof.
\end{proof}

\section{Concluding remarks}\label{sec:remarks}

In this paper, we showed that there exist $(k,k-1)$-systems whose Ramsey numbers admit lower bounds of the same order of magnitude as those of the complete $k$-uniform hypergraph. One caveat of our proof is that we require at least $4$ colors. This is unlike in \cites{EHR65, CFS13}, where one can still apply the step-up method for $2$ colors and uniformity at least $k \geq 4$. More precisely, the authors of \cites{EHR65, CFS13} proved that
\begin{align*}
    R(K_n^{(k)};2)\geq t_{k-2}(cn^2)
\end{align*}
for some absolute positive constant $c$. Note that this result is only one power away from the upper bound described in the introduction. It would be interesting to know whether one can obtain similar lower bounds for $(k, k-1)$-systems with $2$ colors.

Perhaps the most interesting question concerns the Ramsey numbers of other $(k,\ell)$-systems.
\begin{problem}
    Determine for which $2\leq \ell \leq k$ there exist an integer $r \geq 2$ and a constant $c > 0$ such that, for all sufficiently large $n$, one can find a $(k,\ell)$-system $H$ on $n$ vertices satisfying
    \[
        R(H; r) \;\geq\; t_{k-1}(n^c).
    \]
\end{problem}
Another classical parameter for measuring sparsity in a hypergraph is its girth. A cycle of length $t$ in a $k$-graph is a sequence of vertices and edges $v_0,e_0,v_1,e_1,\ldots,v_{t-1},e_{t-1}$ such that $v_i$ belongs to both $e_{i-1}$ and $e_i$, where the indices are taken modulo $t$. A $k$-graph $H$ has girth $g$ if it does not contain cycles of length smaller than $g$. Note that a $k$-graph of girth $g\geq 3$ must be linear. It would be interesting to determine the Ramsey number of linear $3$-graphs of large girth.
\begin{problem}
    Determine whether there exists an integer $r\geq 2$, such that for every $g$ and all sufficiently large $n$, one can find a linear $3$-graph $H$ of girth $g$ on $n$ vertices with $R(H; r)$ doubly exponential.
\end{problem}
We remark that an argument similar to that of Theorem \ref{thm:main} yields a triangle-free linear $H$ with doubly exponential Ramsey number. However, we were not able to find such a hypergraph for $g=5$.

\bibliography{ref}

@article {Erd47,
    AUTHOR = {Erd\H{o}s, P.},
     TITLE = {Some remarks on the theory of graphs},
   JOURNAL = {Bull. Amer. Math. Soc.},
  FJOURNAL = {Bulletin of the American Mathematical Society},
    VOLUME = {53},
      YEAR = {1947},
     PAGES = {292--294},
      ISSN = {0002-9904},
   MRCLASS = {56.0X},
  MRNUMBER = {19911},
MRREVIEWER = {H.\ S. M. Coxeter},
       DOI = {10.1090/S0002-9904-1947-08785-1},
       URL = {https://doi.org/10.1090/S0002-9904-1947-08785-1},
}

@book {GRSBook,
    AUTHOR = {Graham, Ronald L. and Rothschild, Bruce L. and Spencer, Joel
              H.},
     TITLE = {Ramsey theory},
    SERIES = {Wiley-Interscience Series in Discrete Mathematics and
              Optimization},
   EDITION = {Second},
      NOTE = {A Wiley-Interscience Publication},
 PUBLISHER = {John Wiley \& Sons, Inc., New York},
      YEAR = {1990},
     PAGES = {xii+196},
      ISBN = {0-471-50046-1},
   MRCLASS = {05-02 (04A20 05A99 05C55 54H20)},
  MRNUMBER = {1044995},
}

@article{EHR65,
  author = {Erd{\H{o}}s, P. and Hajnal, A. and Rado, R.},
  title = {Partition relations for cardinal numbers},
  journal = {Acta Math. Acad. Sci. Hungar.},
  volume = {16},
  year = {1965},
  pages = {93--196},
  issn = {0001-5954},
  doi = {10.1007/BF01886396},
  mrclass = {04.60},
  mrnumber = {202613},
  mrreviewer = {L. Gillman},
  url = {https://doi.org/10.1007/BF01886396},
}

@book {JLR00,
    AUTHOR = {Janson, Svante and \L{}uczak, Tomasz and Rucinski, Andrzej},
     TITLE = {Random graphs},
    SERIES = {Wiley-Interscience Series in Discrete Mathematics and
              Optimization},
 PUBLISHER = {Wiley-Interscience, New York},
      YEAR = {2000},
     PAGES = {xii+333},
      ISBN = {0-471-17541-2},
   MRCLASS = {05C80 (60C05 82B41)},
  MRNUMBER = {1782847},
MRREVIEWER = {Mark\ R.\ Jerrum},
       DOI = {10.1002/9781118032718},
       URL = {https://doi.org/10.1002/9781118032718},
}

@article {RW89,
    AUTHOR = {R\"{o}dl, Vojt\v{e}ch and Winkler, Peter},
     TITLE = {A {R}amsey-type theorem for orderings of a graph},
   JOURNAL = {SIAM J. Discrete Math.},
  FJOURNAL = {SIAM Journal on Discrete Mathematics},
    VOLUME = {2},
      YEAR = {1989},
    NUMBER = {3},
     PAGES = {402--406},
      ISSN = {0895-4801},
   MRCLASS = {05C55 (05C35 06F99)},
  MRNUMBER = {1002703},
MRREVIEWER = {R.\ H.\ Schelp},
       DOI = {10.1137/0402035},
       URL = {https://doi.org/10.1137/0402035},
}

@article {R29,
    AUTHOR = {Ramsey, F. P.},
     TITLE = {On a {P}roblem of {F}ormal {L}ogic},
   JOURNAL = {Proc. London Math. Soc. (2)},
  FJOURNAL = {Proceedings of the London Mathematical Society. Second Series},
    VOLUME = {30},
      YEAR = {1929},
    NUMBER = {4},
     PAGES = {264--286},
      ISSN = {0024-6115},
   MRCLASS = {DML},
  MRNUMBER = {1576401},
       DOI = {10.1112/plms/s2-30.1.264},
       URL = {https://doi.org/10.1112/plms/s2-30.1.264},
}

@article {ES35,
    AUTHOR = {Erd\H{o}s, P. and Szekeres, G.},
     TITLE = {A combinatorial problem in geometry},
   JOURNAL = {Compositio Math.},
  FJOURNAL = {Compositio Mathematica},
    VOLUME = {2},
      YEAR = {1935},
     PAGES = {463--470},
      ISSN = {0010-437X,1570-5846},
   MRCLASS = {99-04},
  MRNUMBER = {1556929},
       URL = {http://www.numdam.org/item?id=CM_1935__2__463_0},
}

@article {ER52,
    AUTHOR = {Erd\H{o}s, P. and Rado, R.},
     TITLE = {Combinatorial theorems on classifications of subsets of a
              given set},
   JOURNAL = {Proc. London Math. Soc. (3)},
  FJOURNAL = {Proceedings of the London Mathematical Society. Third Series},
    VOLUME = {2},
      YEAR = {1952},
     PAGES = {417--439},
      ISSN = {0024-6115,1460-244X},
   MRCLASS = {27.2X},
  MRNUMBER = {65615},
MRREVIEWER = {J.\ Riguet},
       DOI = {10.1112/plms/s3-2.1.417},
       URL = {https://doi.org/10.1112/plms/s3-2.1.417},
}

@article {Ab72,
    AUTHOR = {Abbott, H. L.},
     TITLE = {A note on {R}amsey's theorem},
   JOURNAL = {Canad. Math. Bull.},
  FJOURNAL = {Canadian Mathematical Bulletin. Bulletin Canadien de
              Math\'ematiques},
    VOLUME = {15},
      YEAR = {1972},
     PAGES = {9--10},
      ISSN = {0008-4395,1496-4287},
   MRCLASS = {05C15},
  MRNUMBER = {314673},
MRREVIEWER = {Vaclav\ Chv\'atal},
       DOI = {10.4153/CMB-1972-002-5},
       URL = {https://doi.org/10.4153/CMB-1972-002-5},
}

@article {CFS13,
    AUTHOR = {Conlon, David and Fox, Jacob and Sudakov, Benny},
     TITLE = {An improved bound for the stepping-up lemma},
   JOURNAL = {Discrete Appl. Math.},
  FJOURNAL = {Discrete Applied Mathematics. The Journal of Combinatorial
              Algorithms, Informatics and Computational Sciences},
    VOLUME = {161},
      YEAR = {2013},
    NUMBER = {9},
     PAGES = {1191--1196},
      ISSN = {0166-218X,1872-6771},
   MRCLASS = {05C55 (05C65)},
  MRNUMBER = {3030610},
       DOI = {10.1016/j.dam.2010.10.013},
       URL = {https://doi.org/10.1016/j.dam.2010.10.013},
}

@article {CF21,
    AUTHOR = {Conlon, David and Ferber, Asaf},
     TITLE = {Lower bounds for multicolor {R}amsey numbers},
   JOURNAL = {Adv. Math.},
  FJOURNAL = {Advances in Mathematics},
    VOLUME = {378},
      YEAR = {2021},
     PAGES = {Paper No. 107528, 5},
      ISSN = {0001-8708,1090-2082},
   MRCLASS = {05C55},
  MRNUMBER = {4186575},
MRREVIEWER = {N.\ Hindman},
       DOI = {10.1016/j.aim.2020.107528},
       URL = {https://doi.org/10.1016/j.aim.2020.107528},
}

@article {NORS08,
    AUTHOR = {Nagle, Brendan and Olsen, Sayaka and R\"odl, Vojt\v{e}ch and
              Schacht, Mathias},
     TITLE = {On the {R}amsey number of sparse 3-graphs},
   JOURNAL = {Graphs Combin.},
  FJOURNAL = {Graphs and Combinatorics},
    VOLUME = {24},
      YEAR = {2008},
    NUMBER = {3},
     PAGES = {205--228},
      ISSN = {0911-0119,1435-5914},
   MRCLASS = {05C65 (05C15 05C55)},
  MRNUMBER = {2410941},
MRREVIEWER = {Ewa\ Drgas-Burchardt},
       DOI = {10.1007/s00373-008-0784-x},
       URL = {https://doi.org/10.1007/s00373-008-0784-x},
}

@article {CFKO08,
    AUTHOR = {Cooley, Oliver and Fountoulakis, Nikolaos and K\"uhn, Daniela
              and Osthus, Deryk},
     TITLE = {3-uniform hypergraphs of bounded degree have linear {R}amsey
              numbers},
   JOURNAL = {J. Combin. Theory Ser. B},
  FJOURNAL = {Journal of Combinatorial Theory. Series B},
    VOLUME = {98},
      YEAR = {2008},
    NUMBER = {3},
     PAGES = {484--505},
      ISSN = {0095-8956,1096-0902},
   MRCLASS = {05C55 (05C65)},
  MRNUMBER = {2401125},
MRREVIEWER = {Boris\ Bukh},
       DOI = {10.1016/j.jctb.2007.08.008},
       URL = {https://doi.org/10.1016/j.jctb.2007.08.008},
}

@article {CFKO09,
    AUTHOR = {Cooley, Oliver and Fountoulakis, Nikolaos and K\"uhn, Daniela
              and Osthus, Deryk},
     TITLE = {Embeddings and {R}amsey numbers of sparse {$k$}-uniform
              hypergraphs},
   JOURNAL = {Combinatorica},
  FJOURNAL = {Combinatorica. An International Journal on Combinatorics and
              the Theory of Computing},
    VOLUME = {29},
      YEAR = {2009},
    NUMBER = {3},
     PAGES = {263--297},
      ISSN = {0209-9683,1439-6912},
   MRCLASS = {05C55 (05C65 05D10)},
  MRNUMBER = {2520273},
MRREVIEWER = {Andr\'as\ Gy\'arf\'as},
       DOI = {10.1007/s00493-009-2356-y},
       URL = {https://doi.org/10.1007/s00493-009-2356-y},
}

@article {CRST83,
    AUTHOR = {Chvat\'al, C. and R\"odl, V. and Szemer\'edi, E. and Trotter,
              Jr., W. T.},
     TITLE = {The {R}amsey number of a graph with bounded maximum degree},
   JOURNAL = {J. Combin. Theory Ser. B},
  FJOURNAL = {Journal of Combinatorial Theory. Series B},
    VOLUME = {34},
      YEAR = {1983},
    NUMBER = {3},
     PAGES = {239--243},
      ISSN = {0095-8956,1096-0902},
   MRCLASS = {05C55},
  MRNUMBER = {714447},
MRREVIEWER = {Stefan\ A.\ Burr},
       DOI = {10.1016/0095-8956(83)90037-0},
       URL = {https://doi.org/10.1016/0095-8956(83)90037-0},
}

@misc{ CFSunp,
	author={Conlon, D.},
	note={Personal communication},
	date={2025},
}

@article {CFS09,
    AUTHOR = {Conlon, David and Fox, Jacob and Sudakov, Benny},
     TITLE = {Ramsey numbers of sparse hypergraphs},
   JOURNAL = {Random Structures Algorithms},
  FJOURNAL = {Random Structures \& Algorithms},
    VOLUME = {35},
      YEAR = {2009},
    NUMBER = {1},
     PAGES = {1--14},
      ISSN = {1042-9832,1098-2418},
   MRCLASS = {05D10 (05C55 05C65)},
  MRNUMBER = {2532871},
MRREVIEWER = {J.\ Spencer},
       DOI = {10.1002/rsa.20260},
       URL = {https://doi.org/10.1002/rsa.20260},
}

\end{document}